\documentclass{article}

\usepackage{amssymb}
\usepackage{latexsym}
\usepackage{amsmath}
\usepackage{mathrsfs}
\usepackage{color}
\usepackage{CJK}

\newtheorem{theorem}{Theorem}[section]
\newtheorem{corollary}{Corollary}[section]
\newtheorem{lemma}{Lemma}[section]
\newtheorem{proposition}{Proposition}[section]
\newtheorem{remark}{Remark}[section]
\newtheorem{definition}{Definition}[section]

\numberwithin{equation}{section}
\newenvironment{proof}{\medskip\par\noindent{\bf Proof.}\ }{\qquad
\raisebox{-0.5mm}{\rule{1.5mm}{4mm}}\vspace{6pt}}

\newcommand{\bbr}{\mathbb{R}}
\newcommand{\h}{H^1(\bbr^4)}

\newcommand{\bbn}{\mathbb{N}}
\newcommand{\ve}{\varepsilon}

\newcommand{\bd}{\begin{definition}}
\newcommand{\ed}{\end{definition}}

\newcommand{\br}{\begin{remark}}
\newcommand{\er}{\end{remark}}

\newcommand{\be}{\begin{equation}}
\newcommand{\ee}{\end{equation}}

\newcommand{\bc}{\begin{corollary}}
\newcommand{\ec}{\end{corollary}}

\begin{document}

\title
{\Large\bf On a doubly critical Schr\"odinger system  in $\bbr^4$  with steep potential wells\thanks{Supported by NSFC(11371212, 11271386).
E-mails: wuyz850306@cumt.edu.cn (Y. Wu); wzou@math.tsinghua.edu.cn (W. Zou)}}

\author{\bf Yuanze Wu$^{a}$    \&    Wenming Zou$^b$    \\
\footnotesize$^a${\it College of Sciences, China
University of Mining and Technology,}\\
\footnotesize{\it Xuzhou 221116, P.R. China }\\
\footnotesize$^b${\it  Department of Mathematical Sciences, Tsinghua University, Beijing 100084, P.R. China }}
\date{}
\maketitle

\begin{center}
\begin{minipage}{120mm}

\noindent{\bf Abstract:} Study the following two-component elliptic system%
\begin{equation*}
\left\{\aligned&\Delta u-(\lambda a(x)+a_0)u+u^3+\beta v^2u=0\quad&\text{in }\bbr^4,\\%
&\Delta v-(\lambda b(x)+b_0)v+v^3+\beta u^2v=0\quad&\text{in }\bbr^4,\\%
&(u,v)\in\h\times\h,\endaligned\right.%
\end{equation*}
where $a_0,b_0\in\bbr$ are constants;  $\lambda>0$ and $\beta\in\bbr$ are parameters
and $a(x), b(x)\geq0$ are potential wells which are not necessarily to be radial symmetric.
By using the variational method, we investigate the existence of ground state solutions and
general ground state solutions (i.e., possibly   semi-trivial) to this system. Indeed, to the best of  our  knowledge,
even the existence of semi-trivial solutions is also unknown in the literature.
We   observe some concentration behaviors of ground state solutions and general
ground state solutions. The phenomenon of phase separations is also excepted. It seems that this is the  first result definitely  describing  the phenomenon of phase separation for critical system in the whole space  $\bbr^4$.  Note that both the cubic nonlinearities and the coupled terms  of the system are  all of critical growth with respect to the Sobolev critical exponent.

\vspace{6mm} \noindent{\bf Keywords:} Elliptic system; Ground state; Steep potential well; Critical Sobolev exponent; Variational method.%

\vspace{6mm}\noindent {\bf AMS} Subject Classification 2010: 35B38; 35B40; 35J10; 35J20.%

\end{minipage}
\end{center}

\vskip0.26in

\section{Introduction}

We study the following two-component elliptic system%
\begin{equation*}
\left\{\aligned&\Delta u-(\lambda a(x)+a_0)u+u^3+\beta v^2u=0, &x\text{ in }\bbr^4,\\
&\Delta v-(\lambda b(x)+b_0)v+v^3+\beta u^2v=0, &x\text{ in }\bbr^4,\\
&(u,v)\in\h\times\h,\endaligned\right.\eqno{(\mathcal{P}_{\lambda,\beta})}%
\end{equation*}
where $a_0,b_0\in\bbr$ are constants and $\lambda>0$, $\beta\in\bbr$ are parameters.  The potentials $a(x)$ and $b(x)$ satisfy some conditions to be specified later.%

\vskip0.12in

It is well known that the solutions of  $(\mathcal{P}_{\lambda,\beta})$  are  related to the solitary wave solutions to
 the following two-component system of nonlinear Schr\"odinger equations%
\begin{equation*}
\left\{\aligned&-i\frac{\partial}{\partial t}\Psi_1=\Delta \Psi_1-\lambda a(x)\Psi_1+|\Psi_1|^2\Psi_1+\beta |\Psi_2|^2\Psi_1=0,
 \\%
&-i\frac{\partial}{\partial t}\Psi_2=\Delta \Psi_2-\lambda b(x)\Psi_2+|\Psi_2|^2\Psi_2+\beta |\Psi_1|^2\Psi_2=0, \\%
&\quad x\text{ in }\bbr^4, t>0;\;\;  \Psi_j=\Psi_j(t,x)\in \mathbb{C},\ \ j=1,2,\\%
&\quad \Psi_j(t,x)\to 0  \hbox{ as }|x|\to+\infty,\ \ t>0.\endaligned\right.\eqno{(\mathcal{P}_{\lambda,\beta}^*)}%
\end{equation*}
Indeed, set $\Psi_1(t,x)=e^{-ita_0}u(x)$ and $\Psi_2(t,x)=e^{-itb_0}v(x)$, then $(\Psi_1, \Psi_2)$ is called
 the solitary wave solution of $(\mathcal{P}_{\lambda,\beta}^*)$ and $(u,v)$ is a solution of the  $(\mathcal{P}_{\lambda,\beta})$
 if and only if $(\Psi_1, \Psi_2)$ is a solution of the  $(\mathcal{P}_{\lambda,\beta}^*)$.

\vskip0.2in

In the literature, the System~$(\mathcal{P}_{\lambda,\beta}^*)$ defined on  an open set $\Omega$ (in $\bbr^2$ or $\bbr^3$) is called
 the Gross-Pitaevskii equations  (e.g.  \cite{HMEWC98,TV09}), which appears in many different physical problems.
 For example, in the Hartree-Fock theory, the Gross-Pitaevskii equations can be used to describe a
 binary mixture of Bose-Einstein condensates in two different hyperfine states $|1\rangle$ and
 $|2\rangle$ (cf. \cite{EGBB97}).  The solutions $\Psi_j(j=1,2)$ are the corresponding condensate amplitudes
 and $\beta$ is the interaction of the states $|1\rangle$ and $|2\rangle$.
 The interaction is attractive if $\beta>0$ and repulsive if $\beta<0$.  When the interaction is repulsive,
 it is expected that the phenomenon of phase separation  will happen, that is, the two components of the
  system tend to separate in different regions as the interaction tends to infinity.
   The Gross-Pitaevskii equation  also arises in nonlinear optics (cf. \cite{AA99}).
   Due to the important application in physics, the Gross-Pitaevskii equation  $(\mathcal{P}_{0,\beta}^*)$ has
    been studied extensively in the last  decades.  We refer the readers to \cite{BDW10,CLZ14,LW05,LW06,R14,S07}
    and the references therein, where various existence theorems  of the solitary wave solutions were established.%

\vskip0.23in

When we consider the   equation  $(\mathcal{P}_{\lambda,\beta}^*)$  or  $(\mathcal{P}_{\lambda,\beta})$ in $\bbr^4$, the cubic nonlinearities and
 the couple terms  are all of  critical growth, since the Sobolev critical exponent $2^\ast:=2N/(N-2)=4$ in $\bbr^N=\bbr^4$.
 By the Pohozaev identity, we can easily conclude that any solution of $(\mathcal{P}_{0,\beta})$
 satisfies $\int_{\bbr^4}a_0u^2+b_0v^2dx=0$ (cf. \cite{CZ121,CZ14}).  Thus, any solution
 of $(\mathcal{P}_{0,\beta})$ must be $(0, 0)$ in the case of $a_0b_0>0$.
 Due to this reason, to some extent, it seems that $\lambda\not=0$ is a necessary condition for
 the existence of non-zero or even non-trivial solutions to $(\mathcal{P}_{\lambda,\beta})$.

\bd We call that $(u, v)\in\h\times\h$ is a non-zero solution of $(\mathcal{P}_{\lambda,\beta})$ if $(u, v)$ is a
  solution of $(\mathcal{P}_{\lambda,\beta})$ with $(u,v)\not=(0,0)$;  we say $(u, v)\in\h\times\h$ is a
   non-trivial solution of $(\mathcal{P}_{\lambda,\beta})$ if $(u, v)$ is a non-zero solution
    with both $u\not=0$ and $v\not=0$.
\ed

    To the  best of our knowledge,  few result has been established for  the System~$(\mathcal{P}_{\lambda,\beta})$.
In this paper, we will study the System~$(\mathcal{P}_{\lambda,\beta})$ with $\lambda>0$ when $a(x),b(x)$ satisfy the following conditions:%
\begin{enumerate}
\item[${\bf (D_1)}$] $a(x), b(x)\in C(\bbr^4)$ and $a(x), b(x)\geq0$ on $\bbr^4$.%
\item[${\bf (D_2)}$] There exist $a_\infty, b_\infty\in(0, +\infty)$ such that $\displaystyle \lim_{|x|\to+\infty}a(x)=a_\infty$ and
$a(x)\leq a_\infty$ for all $x\in\bbr^4$ while $\displaystyle \lim_{|x|\to+\infty}b(x)=b_\infty$ and $b(x)\leq b_\infty$ for all $x\in\bbr^4$.%
\item[${\bf (D_3)}$] $\displaystyle \Omega_a:=\text{int } a^{-1}(0)$ and $\displaystyle \Omega_b:=\text{int } b^{-1}(0)$
are bounded  non-empty domains and have smooth boundaries.  Moreover, $\overline{\Omega}_a=a^{-1}(0)$,
 $\overline{\Omega}_b=b^{-1}(0)$ and $\overline{\Omega}_a\cap\overline{\Omega}_b=\emptyset$.%
\end{enumerate}
In the sequel, $\lambda a(x)$ and $\lambda b(x)$ are called   the steep potential wells under
the conditions $(D_1)$-$(D_3)$ if the parameter $\lambda$ is sufficiently large.
The depth of the wells is controlled by the parameter $\lambda$.
 An interesting phenomenon for this kind of Schr\"odinger equations is that, one can expect to
 find the solutions which are concentrated at the bottom of the wells as the depth goes to infinity.
 Due to this interesting property, such a topic for the scalar Schr\"odinger equations was studied extensively
 in the last  decades.  We refer the readers to \cite{AFS09, BW95, BT13,DT03,DS07,LHL11,ST09,WZ09} and the references therein.
 Most of the  papers  are devoted to the subcritical case.    In  recent years, the steep potential wells were also introduced to
 some other elliptic equations and systems, see for example \cite{FSX10,GT121,JZ11,SW14,WWZ15} and the references therein.
 In particular, in \cite{WWZ15}, the Gross-Pitaevskii equations in $\bbr^3$ (subcritical case) with steep potential wells were considered and some
 existence results of the solitary wave solutions were established.


 \vskip0.1in
  Under the conditions $(D_1)$-$(D_3)$, the System~$(\mathcal{P}_{\lambda,\beta})$ has a variational structure.  Indeed, let%
\begin{equation*}
E_{a}:=\{u\in D^{1,2}(\bbr^4)\mid\int_{\bbr^4}a(x)u^2dx<+\infty\};
\end{equation*}
\begin{equation*}
E_{b}:=\{u\in D^{1,2}(\bbr^4)\mid\int_{\bbr^4}b(x)u^2dx<+\infty\}.%
\end{equation*}
Then by the condition $(D_1)$, for every $a_0,b_0\in\bbr$ and $\lambda>\max\{0, \frac{-a_0}{a_\infty}, \frac{-b_0}{b_\infty}\}$,
$E_a$ and $E_b$ are the Hilbert spaces equipped with the following inner products
\begin{equation*}
\langle u,v\rangle_{a,\lambda}:=\int_{\bbr^4}\nabla u\nabla v+(\lambda a(x)+a_0)^+uvdx,
\end{equation*}
\begin{equation*}
\langle u,v\rangle_{b,\lambda}:=\int_{\bbr^4}\nabla u\nabla v+(\lambda b(x)+b_0)^+uvdx,%
\end{equation*}
 respectively, where $(\cdot)^+:=\max\{\cdot, 0\}$.  The corresponding norms are respectively given by%
\begin{equation*}
\|u\|_{a,\lambda}:=\bigg(\int_{\bbr^4}|\nabla u|^2+(\lambda a(x)+a_0)^+u^2dx\bigg)^{\frac12}
\end{equation*}
and
\begin{equation*}
 \|v\|_{b,\lambda}:=\bigg(\int_{\bbr^4}|\nabla v|^2+(\lambda b(x)+b_0)^+v^2dx\bigg)^{\frac12}.%
\end{equation*}
We  denote the Hilbert spaces $(E_a, \|\cdot\|_{a,\lambda})$ and $(E_b, \|\cdot\|_{b,\lambda})$ by $E_{a,\lambda}$ and $E_{b,\lambda}$ respectively.  Let $E_\lambda:=E_{a,\lambda}\times E_{b,\lambda}$ be  the  Hilbert space with the inner product
\begin{equation*}
\langle (u,v), (w,\sigma)\rangle_{\lambda}:=\langle u,w\rangle_{a,\lambda}+\langle v,\sigma\rangle_{b,\lambda}.
\end{equation*}
The corresponding norm is given by $\|(u,v)\|_\lambda:=(\|u\|_{a,\lambda}^2+\|v\|_{b,\lambda}^2)^{\frac12}$.
Then by the conditions $(D_1)$-$(D_2)$ and the H\"older and Sobolev inequalities,
 for every $\lambda>\max\{0, \frac{-a_0}{a_\infty}, \frac{-b_0}{b_\infty}\}$,  there exists $d_\lambda>0$ such that%
\begin{equation}\label{eq0001}%
\|u\|_{L^2(\bbr^4)}\leq d_\lambda\|u\|_{a,\lambda},\quad\|v\|_{L^2(\bbr^4)}\leq d_\lambda\|v\|_{b,\lambda}%
\end{equation}
and%
\begin{equation}
\|u\|_{L^4(\bbr^4)}\leq S^{-\frac12}\|u\|_{a,\lambda},\quad\|v\|_{L^4(\bbr^4)}\leq S^{-\frac12}\|v\|_{b,\lambda},\label{eq0003}%
\end{equation}
for  $(u,v)\in E_\lambda$, where $\|\cdot\|_{L^p(\bbr^4)}$ is the usual norm in $L^p(\bbr^4)$ for all $p\geq1$ and $S$ is the best Sobolev embedding constant from $D^{1,2}(\bbr^4)$ to $L^4(\bbr^4)$ and given by%
\begin{equation*}
S:=\inf\{\|\nabla u\|_{L^2(\bbr^4)}^2 \mid u\in D^{1,2}(\bbr^4), \|u\|_{L^4(\bbr^4)}^2=1\}.%
\end{equation*}
It follows that $E_\lambda$ is embedded continuously into $\h\times\h$ for $\lambda>\max\{0, \frac{-a_0}{a_\infty}, \frac{-b_0}{b_\infty}\}$.
Moreover, by \eqref{eq0001}--\eqref{eq0003}, the conditions $(D_1)$--$(D_2)$ and the H\"older inequality, the energy
functional $J_{\lambda,\beta}(u,v)$ given by%
\begin{align}\label{eq0131}
&J_{\lambda,\beta}(u,v)\nonumber \\
&:=\frac12\int_{\bbr^4}|\nabla u|^2+(\lambda a(x)+a_0)u^2dx+\frac12\int_{\bbr^4}|\nabla v|^2+(\lambda b(x)+b_0)v^2dx\notag\\
&\quad\quad  -\frac{1}{4}\int_{\bbr^4}u^4dx-\frac{1}{4}\int_{\bbr^4}v^4dx-\frac\beta2\int_{\bbr^4}u^2v^2dx
\end{align}
is well defined in $E_\lambda$ for $\lambda>\max\{0, \frac{-a_0}{a_\infty}, \frac{-b_0}{b_\infty}\}$ and $\beta\in\bbr$.
 Furthermore, by a standard argument, we can also show that $J_{\lambda,\beta}(u,v)$ is of $C^2$ in $E_\lambda$ and it
 is the corresponding energy  functional to   System~$(\mathcal{P}_{\lambda,\beta})$.
  For the sake of convenience, we re-write the energy functional $J_{\lambda,\beta}(u,v)$ by%
\begin{eqnarray*}
J_{\lambda,\beta}(u,v)=\frac12\mathcal{D}_\lambda(u,v)-\frac14\mathcal{L}_\beta(u,v),%
\end{eqnarray*}
where $\mathcal{D}_\lambda(u,v):=\mathcal{D}_{a,\lambda}(u,u)+\mathcal{D}_{b,\lambda}(v,v)$ with
$$\mathcal{D}_{a,\lambda}(u,v):=\int_{\bbr^4}(\nabla u\nabla v+(\lambda a(x)+a_0)uv)dx, $$
$$
\mathcal{D}_{b,\lambda}(u,v):=\int_{\bbr^4}(\nabla u\nabla v+(\lambda b(x)+b_0)uv)dx$$
and $$\mathcal{L}_\beta(u,v):=\|u\|^4_{L^4(\bbr^4)}+\|v\|^4_{L^4(\bbr^4)}+2\beta\|u^2v^2\|_{L^1(\bbr^4)}.$$
We are interested  in finding the ground state solutions of $(\mathcal{P}_{\lambda,\beta})$ for $\lambda$
sufficiently large.

\bd We say that $(u,v)$ is a ground state solution of $(\mathcal{P}_{\lambda,\beta})$ if $(u,v)$ is
a non-trivial solution of $(\mathcal{P}_{\lambda,\beta})$ and the energy   of $(u,v)$ given by \eqref{eq0131}
is  the least one among all that  of  the  non-trivial solutions to $(\mathcal{P}_{\lambda,\beta})$.
\ed

 To the best of our knowledge,  the existence of semi-trivial solution  to  $(\mathcal{P}_{\lambda,\beta})$ is also  unknown in
 the literature. Therefore,  we are also concerned with   finding the general ground state solutions to $(\mathcal{P}_{\lambda,\beta})$ for $\lambda$
 sufficiently large.

\bd We say $(u, v)\in\h\times\h$ is a semi-trivial solution of $(\mathcal{P}_{\lambda,\beta})$ if $(u, v)$
 is a non-zero solution to $(\mathcal{P}_{\lambda,\beta})$ of the type $(u,0)$ or  $(0,v)$;  we call  $(u,v)$   a general ground
 state solution of $(\mathcal{P}_{\lambda,\beta})$ if $(u,v)$ is a non-zero solution of $(\mathcal{P}_{\lambda,\beta})$ and its energy   is  the least one among all that of   the  non-zero solutions to $(\mathcal{P}_{\lambda,\beta})$.
 \ed

\bd Let $\mu_{a,1}$ and $\mu_{b,1}$
denote  the first eigenvalues of $(-\Delta, H^1_0(\Omega_a))$ and $(-\Delta, H^1_0(\Omega_b))$, respectively.

 We denote the  sets of
all eigenvalues of   $(-\Delta, H^1_0(\Omega_a))$ and $(-\Delta, H^1_0(\Omega_b))$ by
$\sigma(-\Delta, H_0^1(\Omega_a))$  and $\sigma(-\Delta, H_0^1(\Omega_b))$, respectively.
\ed

\br Without loss of generality, we always assume $a_0\leq b_0$ throughout this paper.
\er

\newpage

\subsection{The case of  $-\mu_{a,1}<a_0$ and $ -\mu_{b,1}<b_0$.}

Clearly, $J_{\lambda,\beta}(u,v)$ is heavily rely on  the properties of  $\mathcal{D}_\lambda(u,v)$,   $a_0$ and $b_0$.
Firstly, we note that  there exists $\Lambda_0\geq0$ such that $\mathcal{D}_{\lambda}(u,v)$ is positively definite
 on $E_\lambda$ for $\lambda>\Lambda_0$ provided that  $-\mu_{a,1}<a_0$ and $ -\mu_{b,1}<b_0$ (see Lemma~\ref{lem0110} below  for more details).
In particular, $\Lambda_0=0$ if $a_0\geq0$ and $b_0\geq0$.  Let%
\begin{equation*}
\mathcal{N}_{\lambda,\beta}:=\Big\{(u,v)\in E_\lambda\mid u\not=0,v\not=0, \langle D[J_{\lambda,\beta}(u,v)],(u,0)\rangle_{E_\lambda^*,E_\lambda}\\
\end{equation*}
\begin{equation*}
\quad\quad\quad =\langle D[J_{\lambda,\beta}(u,v)],(0,v)\rangle_{E_\lambda^*,E_\lambda}=0\Big\}%
\end{equation*}
and%
\begin{equation*}
\mathcal{M}_{\lambda,\beta}:=\Big\{(u,v)\in E_\lambda\backslash\{(0,0)\} \mid \langle D[J_{\lambda,\beta}(u,v)],(u,v)\rangle_{E_\lambda^*,E_\lambda}=0\Big\},%
\end{equation*}
where $D[J_{\lambda,\beta}(u,v)]$ is the Frech\'et derivative of the functional $J_{\lambda,\beta}$ in
$E_\lambda$ at $(u,v)$ and $E_\lambda^*$ is the dual space of $E_\lambda$.
It is easy to see that $\mathcal{N}_{\lambda,\beta}$ and $\mathcal{M}_{\lambda,\beta}$ are both
nonempty and contains all non-trivial solutions and non-zero solutions of the System~$(\mathcal{P}_{\lambda,\beta})$, respectively.
Such sets are the so-called Nehari type sets to $(\mathcal{P}_{\lambda,\beta})$ and they are extensively used for finding
 the ground state solution  to  nonlinear elliptic systems
  (cf. \cite{CZ121,CZ131,CLZ14,CLZ141,LW05,LW05,S07,WWZ15}). Define
\begin{equation}\label{zou-100} m_{\lambda,\beta}:=\inf_{(u,v)\in\mathcal{N}_{\lambda,\beta}}J_{\lambda,\beta}(u,v), \quad\quad
 m_{\lambda,\beta}^*:=\inf_{(u,v)\in\mathcal{M}_{\lambda,\beta}}J_{\lambda,\beta}(u,v).\end{equation}
   Since $\mathcal{D}_{\lambda}(u,v)$ is positively  definite
   on $E_\lambda$, it is also easy to show that $m_{\lambda,\beta}$ and $m_{\lambda,\beta}^*$ are both nonnegative
   for all $\lambda>\Lambda_0$ and $\beta\in\bbr$.

\begin{theorem}\label{thm0001}
Assume  $(D_1)$-$(D_3)$  and  $-\mu_{a,1}<a_0,  -\mu_{b,1}<b_0$.  If $\lambda>\Lambda_0$, then we have the following conclusions:
\begin{enumerate}
\item[$(1)$]  If  $0\leq a_0\leq b_0$, then $$m_{\lambda,\beta}=\frac{S^2}{2(1+\max\{\beta, 0\})}; \quad
m_{\lambda,\beta}^*=\frac{S^2}{2(1+\max\{1,\beta\})}\; \hbox{ for all }\beta\in\bbr.$$
Moreover, both  $m_{\lambda,\beta}$ and $m_{\lambda,\beta}^*$ can   not be attained.

\item[$(2)$] If  $a_0<0$, then  $m_{\lambda,\beta}^*$ can be attained by a general ground state solution of $(\mathcal{P}_{\lambda,\beta})$ for all $\beta\in\bbr$.  Moreover, there exists $\Lambda_\beta>0$ such that the general ground state solution of $(\mathcal{P}_{\lambda,\beta})$ must be semi-trivial  provided that one of the following conditions holds:
    \begin{itemize}
     \item   $a_0<0\leq b_0$, $\beta<1-\frac{|a_0|}{\mu_{a,1}}$ and $\lambda>\Lambda_\beta$;
      \item  $a_0\leq b_0<0$, $\beta<\beta_0$ and $\lambda>\Lambda_\beta$, where
    \begin{equation*}
    \beta_0:=\min\bigg\{\frac12(1-\frac{|a_0|}{\mu_{a,1}})(1-\frac{|b_0|}{\mu_{b,1}}),\;\;
    \frac{1-\frac{|b_0|}{\mu_{b,1}}}{1-\frac{|a_0|}{\mu_{a,1}}},\;\;
    \frac{1-\frac{|a_0|}{\mu_{a,1}}}{1-\frac{|b_0|}{\mu_{b,1}}}\bigg\}.%
    \end{equation*}%
    \end{itemize}

\item[$(3)$]  $m_{\lambda,\beta}$ can be attained by a
 ground state solution of $(\mathcal{P}_{\lambda,\beta})$ if one of the following additional conditions
 holds:
  \begin{itemize}
     \item   $a_0\leq b_0<0$ and $\beta\leq0$;
   \item   $a_0<0$ and $\beta>\beta_\lambda$  for some  $0<\beta_\lambda<+\infty$.
   \end{itemize}
   Moreover, if $a_0<0$, then  $m_{\lambda,\beta}=m_{\lambda,\beta}^*$ for $\beta>\beta_\lambda$.
\end{enumerate}
\end{theorem}

The next is a by-product of the previous theorem.

\bc\label{zou++}
Assume  $(D_1)$-$(D_3)$  and  $-\mu_{a,1}<a_0<0,  -\mu_{b,1}<b_0<0$.  If $\lambda>\Lambda_0$, then  the following equation
\be\label{zou++a}
-\Delta u+(\lambda a(x)+a_0)u=u^3, \quad u\in \h,
\ee
\be\label{zou++b}
-\Delta v+(\lambda b(x)+b_0)v=v^3, \quad v\in \h,
\ee
have   ground state solutions, respectively.
\ec

\br\label{zou++c} The Corollary \ref{zou++} can be viewed as the generalization of the
celebrated   results in \cite{BN=1983}  obtained  by Br{\'e}zis and Nirenberg, where the equation is defined
 on   the bounded smooth domain.  On the other hand, let us  recall the following
 equation which  was studied in \cite{Benci-Cerami} by Benci and Cerami:
\be\label{bcerami}-\Delta u+V(x)u=u^{(N+2)/(N-2)},  \;\;\; u  \in  H^1(\bbr^N),\ee
where $N\geq 3$ and $V(x)$ is a nonnegative function. It was observed when
$V(x)\equiv constant \not=0$, then    (\ref{bcerami}) has only trivial solution $u=0$.  Moreover,  if $\|V(x)\|_{L^{N/2}}$ is sufficiently
small, then  (\ref{bcerami}) has at least one solution.

\er

\vskip0.23in

\subsection{The case of $a_0\leq-\mu_{a,1}$ or $b_0\leq-\mu_{b,1}$}

If either  $a_0\leq-\mu_{a,1}$ or $b_0\leq-\mu_{b,1}$, then there exists $\Lambda_1>0$ such that $\mathcal{D}_{\lambda}(u,v)$ is
indefinite on $E_\lambda$ and has finite augment Morse index for $\lambda>\Lambda_1$
(also see Lemma~\ref{lem0110} below  for more details).  In this case, $\mathcal{N}_{\lambda,\beta}$ and
 $\mathcal{M}_{\lambda,\beta}$ are not the good choice for finding the ground state solution and
 the general ground state solution of $(\mathcal{P}_{\lambda,\beta})$.  For  $\lambda>\Lambda_1,$  let $\widehat{\mathcal{F}}_{a,\lambda}^{\perp}$
  and $\widehat{\mathcal{F}}_{b,\lambda}^{\perp}$ be the negative part of $\mathcal{D}_{a,\lambda}(u,u)$
  on $E_{a,\lambda}$ and $\mathcal{D}_{b,\lambda}(v,v)$ on $E_{b,\lambda}$, respectively.
   Then we can modify $\mathcal{M}_{\lambda,\beta}$ to the following set%
\begin{eqnarray}
\mathcal{G}_{\lambda,\beta}&:=&\bigg\{(u,v)\in\widetilde{E}_\lambda \mid\langle D[J_{\lambda,\beta}(u,v)],(u,v)\rangle_{E_\lambda^*,E_\lambda}=0,\notag\\%
&&\langle D[J_{\lambda,\beta}(u,v)],(w,\sigma)\rangle_{E_\lambda^*,E_\lambda}=0,  \forall  (w,\sigma)\in\widehat{\mathcal{F}}_{a,\lambda}^{\perp}\times\widehat{\mathcal{F}}_{b,\lambda}^{\perp}\bigg\},\label{eq6001}%
\end{eqnarray}
where $\widetilde{E}_\lambda:=E_\lambda\backslash(\widehat{\mathcal{F}}_{a,\lambda}^{\perp}\times\widehat{\mathcal{F}}_{b,\lambda}^{\perp})$.
 This kind of set is the so-called Nehari-Pankov type set to $(\mathcal{P}_{\lambda,\beta})$,
 which was introduced by Pankov in \cite{P05} for   the scalar Schr\"odinger equations with indefinite potentials and
 was further studied by  Szulkin and Weth \cite{SW09}.
 For other papers devoted to the indefinite problems, we would like to refer the readers to
 \cite{BR09,BSR11,EW11} and the references therein.
 Define

\be\label{zou-300}c_{\lambda,\beta}:=\inf_{\mathcal{G}_{\lambda,\beta}}J_{\lambda,\beta}.\ee
Evidently,  $c_{\lambda,\beta}\geq 0$ whenever  $\beta\geq-1$ since $\mathcal{L}_\beta(u,v)$ is  positively definite in this case.%
\begin{theorem}\label{thm0004}
Assume $(D_1)$-$(D_3)$.    Suppose  either $a_0\leq-\mu_{a,1}$ with $-a_0\not\in\sigma(-\Delta, H_0^1(\Omega_a))$ or
$b_0\leq-\mu_{b,1}$ with $-b_0\not\in\sigma(-\Delta, H_0^1(\Omega_b))$.
If $\lambda>\Lambda_1$, then $c_{\lambda,\beta}$ can be attained by a general
ground state solution of   $(\mathcal{P}_{\lambda,\beta})$ for $0\leq\beta<1$.
Furthermore, if  $a_0\leq-\mu_{a,1}<0\leq b_0$, then there exists $\Lambda^*_\beta\geq\Lambda_1$ such that the general ground state
solution of  $(\mathcal{P}_{\lambda,\beta})$ must be semi-trivial and be of the
 type $(u_{\lambda,\beta}, 0)$ for  all $\lambda\geq\Lambda^*_\beta$.
In particular, where $u_{\lambda,\beta}$ is the ground state solution to the equation
\begin{equation}\label{eq5001}
-\Delta u+(\lambda a(x)+a_0)u=u^3,\quad u\in \h.%
\end{equation}
\end{theorem}

\begin{remark}\quad
\begin{enumerate}
\item[$(a)$] Theorem~\ref{thm0004} only gives the existence of the general ground state solution to
  $(\mathcal{P}_{\lambda,\beta})$ for $0\leq\beta<1$ and $\lambda$
 sufficiently large in the case of $a_0\leq-\mu_{a,1}$ with $-a_0\not\in\sigma(-\Delta, H_0^1(\Omega_a))$ or
  $b_0\leq-\mu_{b,1}$ with $-b_0\not\in\sigma(-\Delta, H_0^1(\Omega_b))$. However,  it is still open for us
   that whether $(\mathcal{P}_{\lambda,\beta})$ has the general ground state solution for other $\beta$ in
   such cases.  Indeed, since $\mathcal{L}_\beta(u,v)$ is not symmetric in $E_\lambda$ due to the conditions $(D_1)$--$(D_3)$
   and even indefinite on $\h\times\h$ for $\beta<-1$, Lemmas~\ref{lem0004} and \ref{lem0011} which are
   crucial in the proof of Theorem~\ref{thm0004} are invalid for $\beta\in(-\infty, 0)\cup[1, +\infty)$
   in the case of $a_0\leq-\mu_{a,1}$ with $-a_0\not\in\sigma(-\Delta, H_0^1(\Omega_a))$ or $b_0\leq-\mu_{b,1}$
    with $-b_0\not\in\sigma(-\Delta, H_0^1(\Omega_b))$.%
\item[$(b)$] By Theorem~\ref{thm0004}, it is easy to show that $(\mathcal{P}_{\lambda,0})$
has a ground state solution in the case of $a_0\leq-\mu_{a,1}$ with $-a_0\not\in\sigma(-\Delta, H_0^1(\Omega_a))$
and $b_0\leq-\mu_{b,1}$ with $-b_0\not\in\sigma(-\Delta, H_0^1(\Omega_b))$.  However, since the dimension of the set
for the semi-trivial solutions to $(\mathcal{P}_{\lambda,\beta})$ might be infinite, we do not know how to modify
the Nehari type set $\mathcal{N}_{\lambda,\beta}$ to some Nehari-Pankov type sets as $\mathcal{G}_{\lambda,\beta}$.
Therefore, it is also open to us that whether $(\mathcal{P}_{\lambda,\beta})$ has a ground state solution for
 $\beta\not=0$ in the case of $a_0\leq-\mu_{a,1}$ with $-a_0\not\in\sigma(-\Delta, H_0^1(\Omega_a))$ and
 $b_0\leq-\mu_{b,1}$ with $-b_0\not\in\sigma(-\Delta, H_0^1(\Omega_b))$.%
\item[$(c)$] To the  best of our knowledge, it seems that Theorems~\ref{thm0004} is   the first
existence result for \eqref{eq5001} in the indefinite case.  By checking the proof of Theorem~\ref{thm0001}
(more precisely, Lemma~\ref{lem1002}), we can also see that \eqref{eq5001} has a ground state solution in some
definite case but  might not have solutions in the case of $a_0\geq0$.%
\end{enumerate}
\end{remark}

\subsection{The concentration phenomenon   as
 $\lambda\to+\infty$. }

Since $a(x),b(x)$ have the potential wells, it is natural to ask whether the ground state solution and the general
ground state solution of $(\mathcal{P}_{\lambda,\beta})$ will concentrate at the bottom of $a(x),b(x)$ as
 $\lambda\to+\infty$.  Our results on this aspect can be stated as follows.%
\begin{theorem}\label{thm0002}
Let $(u_{\lambda,\beta}, v_{\lambda,\beta})$ be the solution of $(\mathcal{P}_{\lambda,\beta})$ obtained by
Theorems~\ref{thm0001} and \ref{thm0004}.  Then we have the following conclusions.
\begin{enumerate}
\item[$(1)$] If $(u_{\lambda,\beta}, v_{\lambda,\beta})$ is a ground state solution of $(\mathcal{P}_{\lambda,\beta})$
with $\beta\leq0$ in the case of $a_0\leq b_0<0$, then up to a subsequence $(u_{\lambda,\beta}, v_{\lambda,\beta})\to(u_{0,\beta}, v_{0,\beta})$
strongly in $\h\times\h$ as $\lambda\to+\infty$.
 Furthermore, $(u_{0,\beta}, v_{0,\beta})$ is also a ground state solution of the system:%
    \begin{equation}\label{eq9001}
    \left\{\aligned&\Delta u-a_0u+u^3=0\quad&\text{in }\Omega_a,\\%
    &\Delta v-b_0v+v^3=0\quad&\text{in }\Omega_b,\\%
    &(u,v)\in H_0^1(\Omega_a)\times H_0^1(\Omega_b).\endaligned\right.%
    \end{equation}
\item[$(2)$] If $(u_{\lambda,\beta}, v_{\lambda,\beta})$ is a general ground state
solution of $(\mathcal{P}_{\lambda,\beta})$ in the case of $a_0<0$,
then up to a subsequence $(u_{\lambda,\beta}, v_{\lambda,\beta})\to(u_{0,\beta}, v_{0,\beta})$ strongly
in $\h\times\h$ as $\lambda\to+\infty$. Furthermore, $(u_{0,\beta}, v_{0,\beta})$ is a semi-trivial general ground state solution of \eqref{eq9001}.%
\end{enumerate}
\end{theorem}

\begin{remark}
By checking the proof of Theorem~\ref{thm0001}, we may have $\beta_\lambda\to+\infty$ as $\lambda\to+\infty$
(see Lemmas~\ref{lem5003} and \ref{lem1003} and Proposition~\ref{prop0006} for more details).
 Thus, the concentration behaviors described in Theorem~\ref{thm0002}   may not hold in the case of $\beta\geq\beta_\lambda$.%
\end{remark}

\subsection{Phase separation }

Note that the ground state solution and the general ground state solution of $(\mathcal{P}_{\lambda,\beta})$ are
also depending  on the parameter $\beta$, it is   natural that  we are concerned with  the phenomenon of phase separation as
 $\beta\to-\infty$. Our results on this topic  now read as %
\begin{theorem}\label{thm0003}
Let $(u_{\lambda,\beta}, v_{\lambda,\beta})$ be the ground state solution of $(\mathcal{P}_{\lambda,\beta})$
obtained by Theorem~\ref{thm0001} with $\beta\leq0$.  Then there exists $\Lambda_2>0$ such
 that $\beta\int_{\bbr^4}u_{\lambda,\beta}^2v_{\lambda,\beta}^2dx\to0$ as $\beta\to-\infty$
 for each  $\lambda\geq\Lambda_2$.  Furthermore, for every $\beta_n\to-\infty$, up to a subsequence,
 we also have the following
\begin{enumerate}
\item[$(1)$] $(u_{\lambda,\beta_n}, v_{\lambda,\beta_n})\to (u_{\lambda,\infty}, v_{\lambda,\infty})$ strongly
 in $\h\times\h$ as $n\to\infty$ with $u_{\lambda,\infty}\not=0$ and $v_{\lambda,\infty}\not=0$;
\item[$(2)$] $u_{\lambda,\infty}$ is the ground state solution of the following equation%
\begin{equation*}
-\Delta u+(\lambda a(x)+a_0)u=u^3,\quad u\in H_0^1(\{u_{\lambda,\infty}>0\})%
\end{equation*}
while $v_{\lambda,\infty}$ is the ground state solution to  the following equation%
\begin{equation*}
-\Delta v+(\lambda b(x)+b_0)v=v^3,\quad v\in H_0^1(\{v_{\lambda,\infty}>0\});
\end{equation*}
\item[$(3)$]  both $\{u_{\lambda,\infty}>0\}$ and $\{v_{\lambda,\infty}>0\}$ are   connect domains and $\{u_{\lambda,\infty}>0\}=\bbr^4\backslash\overline{\{v_{\lambda,\infty}>0\}}$.%
\end{enumerate}
\end{theorem}

\begin{remark}
For the Schr\"odinger system  in $\bbr^4$ with critical Sobolev exponent defined  in the whole space,  Theorem \ref{thm0003} seems to be  the first
 result getting the  phase separation.
\end{remark}

We point out that such phenomenon for the ground state solution of the Gross-Pitaevskii equations was
observed in \cite{CLZ14,NTTV10,WW08}   on a bounded
domain of  $\bbr^2$ or $\bbr^3$;  and \cite{TV09,WWZ15} on the whole space $\bbr^2$ or $\bbr^3$.
Such phenomenon for the ground state solution of the elliptic systems with
critical Sobolev exponent on a bounded domain in $\bbr^N(N\geq4)$ was involved  in \cite{CZ121,CZ131,CLZ141}.
In fact,   the authors    of  \cite{CZ121} study the system  in $\bbr^4$ and only get an alternative theorem
which can not assert that the phase separation must happen. In  \cite{CZ131}  (see also  \cite{CLZ141}),
the phase separation is observed when the dimension $N$ of $\bbr^N$ is $\geq 6$ and the system is defined on the bounded domains.
For other kinds of elliptic systems with strong competition, the phenomenon of phase separations has also been
 well studied;  we refer the readers  to  \cite{caf-1, caf-2, con-1}  and references therein.

\subsection{Concentration behaviors  as $\lambda\to+\infty$ and $\beta\to-\infty$}

We also study the concentration behaviors of the ground state solution  obtained by Theorem~\ref{thm0001} as $\lambda\to+\infty$ and $\beta\to-\infty$.
\begin{theorem}\label{thm0007}
Let $(u_{\lambda,\beta}, v_{\lambda,\beta})$ be the ground state solution of $(\mathcal{P}_{\lambda,\beta})$ obtained by Theorem~\ref{thm0001} with $\beta\leq0$.  Then for every $\{(\lambda_n,\beta_n)\}$ satisfying $\lambda_n\to+\infty$ and $\beta_n\to-\infty$ as $n\to\infty$, we have that $(u_{\lambda_n,\beta_n}, v_{\lambda_n,\beta_n})\to(u_{0,0}, v_{0,0})$ strongly in $\h\times\h$ as $n\to\infty$ up to a subsequence for some $(u_{0,0}, v_{0,0})\in H_0^1(\Omega_a)\times H_0^1(\Omega_b)$.  Furthermore, $(u_{0,0}, v_{0,0})$ is also a ground state solution of \eqref{eq9001}.
\end{theorem}

\vskip0.3in

 The structure of the current paper is organized as follows. In section 2, we study the functionals $\mathcal{D}_\lambda(u,v)$ and $\mathcal{L}_\beta(u,v)$.   In section 3, we  explore the ``manifolds"  $\mathcal{N}_{\lambda,\beta}$, $\mathcal{M}_{\lambda,\beta}$ and $\mathcal{G}_{\lambda,\beta}$.   The section 4 will be devoted to the   existence results.   The last section is about the concentration behaviors.
Throughout  this paper, $C$ and $C'$ will be indiscriminately used to denote generic  positive constants
and $o_n(1)$ will   denote the quantities tending  to  zero as $n\to\infty$.

\vskip0.3in

\section{The functionals $\mathcal{D}_\lambda(u,v)$ and $\mathcal{L}_\beta(u,v)$}
In this section, we   give some properties of $\mathcal{D}_\lambda(u,v)$ and $\mathcal{L}_\beta(u,v)$.
We begin with the study of $\mathcal{L}_\beta(u,v)$. Clearly, $\mathcal{L}_\beta(u,v)$ is  positively definite
on $\h\times\h$ if $\beta\geq0$.  For $\beta<0$, let%
\begin{equation}\label{eq1020}
\mathcal{V}_{\beta}=\{(u,v)\in \h\times\h\mid \|u\|_{L^4(\bbr^4)}^4\|v\|_{L^4(\bbr^4)}^4-\beta^2\|u^2v^2\|^2_{L^1(\bbr^4)}>0\},%
\end{equation}
then $\mathcal{V}_{\beta}\not=\emptyset$ and it is easy to see that $\mathcal{L}_\beta(u,v)>0$ if and only if $(u,v)\in\mathcal{V}_{\beta}$.  In what follows, we will make some observations on the functional $\mathcal{D}_{\lambda}(u,v)=\mathcal{D}_{a,\lambda}(u,u)+\mathcal{D}_{b,\lambda}(v,v)$, which are inspired by \cite{BT13} and \cite{DS07}.  We first study the functional $\mathcal{D}_{a,\lambda}(u,u)$.  By the condition $(D_1)$, $\int_{\bbr^4}(\lambda a(x)+a_0)u^2dx\geq0$ for all $u\in E_\lambda$ with $\lambda>0$ in the case of $a_0\geq0$.  It follows that
$\mathcal{D}_{a,\lambda}(u,u)$ is positively definite on $E_\lambda$ with $\lambda>0$ in the case of $a_0\geq0$.
 When $a_0<0$, by the condition $(D_3)$, we have $\Omega_a\subset\mathcal{A}_\lambda$, which is given by%
\begin{equation}  \label{eq0055}
\mathcal{A}_\lambda:=\{x\in\bbr^4\mid \lambda a(x)+a_0<0\}.%
\end{equation}
Thus, $\mathcal{A}_\lambda\not=\emptyset$ for every $\lambda>0$.
Let%
\begin{equation}  \label{eq0056}
\Lambda_{a,0}:=\inf\{\lambda>0\mid|\mathcal{A}_\lambda|<+\infty\}.%
\end{equation}
By conditions $(D_1)$-$(D_2)$, we can see that $0<\Lambda_{a,0}=\frac{-a_0}{a_\infty}$.  For $\lambda>\Lambda_{a,0}$, we define%
\begin{equation*}
\mathcal{F}_{a,\lambda}:=\{u\in E_{a,\lambda}\mid\text{supp}u\subset\bbr^4\backslash\mathcal{A}_\lambda\}.%
\end{equation*}
Then by the conditions $(D_1)$--$(D_2)$, $\mathcal{F}_{a,\lambda}$ is nonempty and $\mathcal{F}_{a,\lambda}\not=E_{a,\lambda}$.  Hence, $E_{a,\lambda}=\mathcal{F}_{a,\lambda}\oplus\mathcal{F}_{a,\lambda}^\perp$ and $\mathcal{F}_{a,\lambda}^\perp\not=\emptyset$ for $\lambda>\Lambda_{a,0}$ in the case of $a_0<0$, where $\mathcal{F}_{a,\lambda}^\perp$ is the orthogonal complement of $\mathcal{F}_{a,\lambda}$ in $E_{a,\lambda}$.  Now, consider the operator $(-\Delta+(\lambda a(x)+a_0)^+)^{-1}(\lambda a(x)+a_0)^-$, where $(\lambda a(x)+a_0)^-=\max\{-(\lambda a(x)+a_0), 0\}$.  Clearly, $(-\Delta+(\lambda a(x)+a_0)^+)^{-1}(\lambda a(x)+a_0)^-$ is linear and self-conjugate on $\mathcal{F}_{a,\lambda}^\perp$ for $\lambda>\Lambda_{a,0}$ in the case of $a_0<0$.  By the definition of $\Lambda_{a,0}$, we can easily show that $(-\Delta+(\lambda a(x)+a_0)^+)^{-1}(\lambda a(x)+a_0)^-$ is also compact on $\mathcal{F}_{a,\lambda}^\perp$ for $\lambda>\Lambda_{a,0}$ in the case of $a_0<0$.
Thus, by \cite[Theorems~4.45 and~4.46]{W95}, the eigenvalue problem%
\begin{equation}\label{eq0042}
-\Delta u+(\lambda a(x)+a_0)^+u=\alpha(\lambda a(x)+a_0)^-u\quad\text{on }\mathcal{F}_{a,\lambda}^\perp%
\end{equation}
has a sequence of positive eigenvalues $\{\alpha_{a,j}(\lambda)\}$ satisfying
$$0<\alpha_{a,1}(\lambda)\leq\alpha_{a,2}(\lambda)\leq\cdots\leq\alpha_{a,j}(\lambda)\to+\infty, \hbox{ as } j\to+\infty.$$
  Furthermore, $\{\alpha_{a,j}(\lambda)\}$ can be characterized by%
\begin{equation}   \label{eq0029}
\alpha_{a,j}(\lambda):=\inf_{\text{dim}M\geq j, M\subset\mathcal{F}_{a,\lambda}^{\perp}}\sup_{u\in M\backslash\{0\}}\frac{\int_{\bbr^4}(|\nabla u|^2+(\lambda a(x)+a_0)^+u^2)dx}{\int_{\bbr^4}(\lambda a(x)+a_0)^-u^2dx}%
\end{equation}
for all $j\in\bbn$ and the corresponding eigenfunctions $\{e_{a,j}(\lambda)\}$ can be chosen
 so that $\int_{\bbr^3}(\lambda a(x)+a_0)^-e_{a,j}^2(\lambda)dx=1$ for all $j\in\bbn$ and are
 a basis of $\mathcal{F}_{a,\lambda}^{\perp}$.%

\vskip0.212in

\begin{lemma}\label{lem0003}
Assume  $(D_1)$--$(D_3)$   and $a_0<0$.  Then $\alpha_{a,j}(\lambda)$ are nondecreasing in $(\Lambda_{a,0}, +\infty)$ for all $j\in\bbn$ and $\lim_{\lambda\to+\infty}\alpha_{a,j}(\lambda)=\alpha_{a,j}^0$, where $\alpha_{a,j}^0$ are the eigenvalues of the following equation%
\begin{equation}\label{eq0130}
-\Delta u=\alpha |a_0|u,\quad u\in H^1_0(\Omega_a).%
\end{equation}
In particular, $\alpha_{a,1}^0$ is the first eigenvalue of \eqref{eq0130}.%
\end{lemma}
\begin{proof}
Let $\lambda_1\geq\lambda_2>\Lambda_{a,0}$, then by the definition of $E_{a,\lambda}$, we have $E_{a,\lambda_1}=E_{a,\lambda_2}$ in the sense of sets.  It follows from the condition $(D_1)$ that $\mathcal{F}_{a,\lambda_2}\subset\mathcal{F}_{a,\lambda_1}$, which implies $\mathcal{F}_{a,\lambda_1}^{\perp}\subset\mathcal{F}_{a,\lambda_2}^{\perp}$.  Thanks to the condition $(D_1)$ and $a_0<0$ once more, we have%
\begin{equation*}
\frac{\int_{\bbr^4}(|\nabla u|^2+(\lambda_1 a(x)+a_0)^+u^2)dx}{\int_{\bbr^4}(\lambda_1 a(x)+a_0)^-u^2dx}\geq\frac{\int_{\bbr^4}(|\nabla u|^2+(\lambda_2 a(x)+a_0)^+u^2)dx}{\int_{\bbr^4}(\lambda_2 a(x)+a_0)^-u^2dx}%
\end{equation*}
for all $u\in \mathcal{F}_{a,\lambda_1}^{\perp}$.  Thus, by the definitions of $\alpha_{a,j}(\lambda_1)$ and $\alpha_{a,j}(\lambda_2)$, we can see that $\alpha_{a,j}(\lambda_2)\leq\alpha_{a,j}(\lambda_1)$, that is, $\alpha_{a,j}(\lambda)$ are nondecreasing in $(\Lambda_{a,0}, +\infty)$ for all $j\in\bbn$.  In what follows, we will show that $\lim_{\lambda\to+\infty}\alpha_{a,j}(\lambda)=\alpha_{a,j}^0$, where $\alpha_{a,j}^0$ is an eigenvalue of \eqref{eq0130}.  Indeed, by the condition $(D_3)$, for every $j\in\bbn$, there exists $\{\varphi_m\}_{1\leq m\leq j}\subset C_0^\infty(\Omega_a)$ such that supp$\varphi_m\cap$ supp$\varphi_n=\emptyset$ for $m\not=n$.  Let $M_0=$span$\{\varphi_1, \cdots, \varphi_j\}$.  Then $M_0\subset\mathcal{F}_{a,\lambda}^{\perp}$ for $\lambda>\Lambda_{a,0}$ due to $a_0<0$ and the condition $(D_3)$ once more.  It follows from \eqref{eq0029} that $\alpha_{a,j}(\lambda)\leq \alpha_{a,j}^*$, where%
\begin{equation*}
\alpha_{a,j}^*:=\sup\bigg\{\int_{\Omega_a}|\nabla u|dx\mid u\in M_0\text{ and }\int_{\Omega_a}|a_0|u^2dx=1\bigg\}.%
\end{equation*}
Since $\alpha_{a,j}(\lambda)$ are positive and nondecreasing in $(\Lambda_{a,0}, +\infty)$ for all $j\in\bbn$, we have
\begin{equation*}
\lim_{\lambda\to+\infty}\alpha_{a,j}(\lambda)=\alpha_{a,j}^0\quad\text{with some }\alpha_{a,j}^0>0\text{ for all }j\in\bbn.
\end{equation*}
Meanwhile, by the choice of $\{e_{a,j}(\lambda)\}$, we have%
\begin{equation}\label{eq1011}
\int_{\bbr^4}(|\nabla e_{a,j}(\lambda)|^2+(\lambda a(x)+a_0)^+[e_{a,j}(\lambda)]^2)dx\leq\alpha_{a,j}^*,%
\end{equation}
which then implies that  $\{e_{a,j}(\lambda)\}$ is bounded in $D^{1,2}(\bbr^4)$ for $\lambda>\Lambda_{a,0}$.  Therefore, up to a subsequence, $e_{a,j}(\lambda)\rightharpoonup e_{a,j}$ weakly in $D^{1,2}(\bbr^4)$ and $e_{a,j}(\lambda)\to e_{a,j}$ a.e. in $\bbr^4$ as $\lambda\to+\infty$.  By the Fatou lemma and the condition $(D_1)$, we have $\int_{\bbr^4}a(x)e_{a,j}^2dx=0$.  This together with the condition $(D_3)$, implies $e_{a,j}=0$ outside $\Omega_a$ and $e_{a,j}\in H_0^1(\Omega_a)$.  It follows from the condition $(D_2)$, the Sobolev embedding theorem and \eqref{eq1011} once more that, up to a subsequence, $e_{a,j}(\lambda)\to e_{a,j}$ strongly in $L^2(\bbr^4)$ as $\lambda\to+\infty$.
Now, by the condition $(D_3)$, for every $\psi\in C_0^\infty(\Omega_a)\subset\mathcal{F}_{a,\lambda}^{\perp}$, we can see from a variant of the
Lebesgue dominated convergence theorem (cf. \cite[Theorem~2.2]{PK74}) that%
\begin{eqnarray*}
\int_{\Omega_a}\nabla e_{a,j}\nabla \psi dx&=&\lim_{\lambda\to+\infty}\int_{\bbr^4}\nabla e_{a,j}(\lambda)\nabla \psi dx\\%
&=&\lim_{\lambda\to+\infty}\alpha_{a,j}(\lambda)\int_{\bbr^4}(\lambda a(x)+a_0)^-e_{a,j}(\lambda)\psi dx\\%
&=&\alpha_{a,j}^0\int_{\Omega_a}|a_0|e_{a,j}\psi dx.%
\end{eqnarray*}
Hence, $(e_{a,j},\alpha_{a,j}^0)$ satisfies \eqref{eq0130} and $\alpha_{a,j}^0$ are the eigenvalues of \eqref{eq0130}.  Note that%
\begin{equation*}
\alpha_{a,1}(\lambda)=\inf_{u\in\mathcal{F}_{a,\lambda}^{\perp}}\bigg\{\int_{\bbr^4}(|\nabla u|^2+(\lambda a(x)+a_0)^+u^2)dx\mid \int_{\bbr^4}(\lambda a(x)+a_0)^-u^2dx=1\bigg\}.%
\end{equation*}
By the definition of the first eigenvalue to \eqref{eq0130} and the condition $(D_3)$, we can easily see that $\alpha_{a,1}^0$ is the first eigenvalue to \eqref{eq0130}.%
\end{proof}

Let $\{\alpha_{a,j}\}$ be the eigenvalues of \eqref{eq0130} and $\{e^*_{a,j}\}$ be the corresponding eigenfunctions.  Then it is well known that $\alpha_{a,j}=\frac{\mu_{a,j}}{|a_0|}$, where $\{\mu_{a,j}\}$ is the eigenvalues of the operator $-\Delta$ in $H_0^1(\Omega_a)$.  Furthermore, for every $a_0<0$, $k_a=$dim$(\text{span}\{e^*_{a,j}\mid\alpha_{a,j}\leq1\})$ is finite.  Let%
\begin{equation*}
\widehat{\mathcal{F}}_{a,\lambda}^{\perp}=\text{span}\{e_{a,j}(\lambda)\mid \alpha_{a,j}(\lambda)\leq1\}\quad\text{and}\quad
\widetilde{\mathcal{F}}_{a,\lambda}^{\perp}=\text{span}\{e_{a,j}(\lambda)\mid \alpha_{a,j}(\lambda)>1\}.%
\end{equation*}
Then dim$(\widehat{\mathcal{F}}_{a,\lambda}^{\perp})<+\infty$ and $E_{a,\lambda}=\mathcal{F}_{a,\lambda}\oplus\widehat{\mathcal{F}}_{a,\lambda}^{\perp}\oplus\widetilde{\mathcal{F}}_{a,\lambda}^{\perp}$ for all $\lambda>\Lambda_{a,0}$ in the case of $a_0<0$.
Furthermore, by Lemma~\ref{lem0003}, $\widehat{\mathcal{F}}_{a,\lambda}^{\perp}=\emptyset$ for $\lambda>\Lambda_{a,0}$ sufficiently large, say $\lambda>\overline{\Lambda}_a>\Lambda_{a,0}$, in the case of $-\mu_{a,1}<a_0<0$ and $\widehat{\mathcal{F}}_{a,\lambda}^{\perp}\not=\emptyset$ for all $\lambda>\Lambda_{a,0}$ in the case of $a_0\leq-\mu_{a,1}$, where $\mu_{a,1}$ is the first eigenvalue of $(-\Delta, H^1_0(\Omega_a))$.%
\begin{lemma}\label{lem0001}
Let the conditions $(D_1)$--$(D_3)$ hold and $a_0<0$.  Then there exists $\Lambda_{a}^*\geq\overline{\Lambda}_{a}$ such that dim$(\widehat{\mathcal{F}}_{a,\lambda}^{\perp})$ is independent of $\lambda\geq\Lambda_{a}^*$ and dim$(\widehat{\mathcal{F}}_{a,\lambda}^{\perp})\leq k_a$ for all $\lambda\geq\Lambda_{a}^*$.%
\end{lemma}
\begin{proof}
In the proof of Lemma~\ref{lem0003}, we obtain that $e_{a,j}(\lambda)\rightharpoonup e_{a,j}$ weakly in $D^{1,2}(\bbr^4)$ and $e_{a,j}(\lambda)\to e_{a,j}\in H_0^1(\Omega_a)$ strongly in $L^2(\bbr^4)$ as $\lambda\to+\infty$ up to a subsequence and $\lim_{\lambda\to+\infty}\alpha_{a,j}(\lambda)=\alpha_{a,j}^0$, where $(e_{a,j}(\lambda), \alpha_{a,j}(\lambda))$ and $(e_{a,j}, \alpha_{a,j}^0)$ satisfy \eqref{eq0042} and \eqref{eq0130}, respectively.  Since $(\lambda a(x)+a_0)^-\leq|a_0|$ due to the condition $(D_1)$, by a variant of the Lebesgue dominated convergence theorem (cf. \cite[Theorem~2.2]{PK74}), we have%
\begin{equation*}
\lim_{\lambda\to+\infty}\alpha_{a,j}(\lambda)\int_{\bbr^4}(\lambda a(x)+a_0)^-[e_{a,j}(\lambda)]^2dx=\alpha_{a,j}^0\int_{\bbr^4}|a_0|e_{a,j}^2dx.%
\end{equation*}
This together with the Fatou's lemma and the conditions $(D_1)$--$(D_3)$, implies $e_{a,j}(\alpha)\to e_{a,j}$ strongly in $D^{1,2}(\bbr^4)$ as $\lambda\to+\infty$ up to a subsequence.  Now, suppose there exist $j\not=i$ such that $\alpha_{a,j}^0=\alpha_{a,i}^0=\alpha_{a,k}$ for some $k\in\bbn$.  Then one of the following two cases must happen:%
\begin{enumerate}
\item[$(1)$] $e_{a,j}=e_{a,i}$;%
\item[$(2)$] $e_{a,j}\not=e_{a,i}$ but $\int_{\Omega_a}\nabla e_{a,j}\nabla e_{a,i}dx=0$.%
\end{enumerate}
If case~$(1)$ happen, then by $a_0<0$, Lemma~\ref{lem0003} and the definition of $e_{a,j}(\lambda)$ and $e_{a,i}(\lambda)$, we have%
\begin{eqnarray*}
& &2\alpha_{a,k}\\
&&=\lim_{\lambda\to+\infty}(\alpha_{a,j}(\lambda)+\alpha_{a,i}(\lambda))\\
& & =\lim_{\lambda\to+\infty}\bigg(\int_{\bbr^4}(|\nabla e_{a,j}(\lambda)|^2+(\lambda a(x)+a_0)^+[e_{a,j}(\lambda)]^2)dx\\%
& &\quad +\int_{\bbr^4}(|\nabla e_{a,i}(\lambda)|^2+(\lambda a(x)+a_0)^+[e_{a,i}(\lambda)]^2)dx\bigg)\\%
& &=\lim_{\lambda\to+\infty}\int_{\bbr^4}(|\nabla (e_{a,j}(\lambda)-e_{a,i}(\lambda))|^2+(\lambda a(x)+a_0)^+[e_{a,j}(\lambda)-e_{a,i}(\lambda)]^2)dx\\%
& &=0,%
\end{eqnarray*}
which is impossible.  Therefore, we must have the case~$(2)$.  Let%
\begin{equation}  \label{eq1015}
j_{a,0}^*=\inf\{j\in\bbn\mid\alpha_{a,j}^0>1\}.%
\end{equation}
Then by Lemma~\ref{lem0003}, there exists $\Lambda_a^*>\overline{\Lambda}_{a}$ such that dim$(\widehat{\mathcal{F}}_{a,\lambda}^{\perp})=j_{a,0}^*-1$ is independent of $\lambda\geq\Lambda_a^*$ and less than or equal to $k_a$ for $\lambda\geq\Lambda_a^*$.%
\end{proof}

\begin{remark}\label{rmk0001}
Clearly, the functional $\mathcal{D}_{b,\lambda}(v,v)$ is also positive definite on $E_{b,\lambda}$ for $\lambda>0$ in the case of $b_0\geq0$.  In the case of $b_0<0$, we can similarly define $\mathcal{B}_\lambda$, $\Lambda_{b,0}$, $\mathcal{F}_{b,\lambda}$, $\alpha_{b,j}(\lambda)$, $\widehat{\mathcal{F}}_{b,\lambda}^\perp$, $\widetilde{\mathcal{F}}_{b,\lambda}^\perp$, $k_b$ and $j_{b,0}^*$.  Then dim$(\widehat{\mathcal{F}}_{b,\lambda}^{\perp})<+\infty$ and $E_{b,\lambda}=\mathcal{F}_{b,\lambda}\oplus\widehat{\mathcal{F}}_{b,\lambda}^{\perp}\oplus\widetilde{\mathcal{F}}_{b,\lambda}^{\perp}$ for all $\lambda>\Lambda_{b,0}$.
Furthermore, by a similar argument as used in Lemma~\ref{lem0003}, we have $\widehat{\mathcal{F}}_{b,\lambda}^{\perp}=\emptyset$ for $\lambda>\Lambda_{b,0}$ sufficiently large, say $\lambda>\overline{\Lambda}_b>\Lambda_{b,0}$, in the case of $-\mu_{b,1}<b_0<0$ and $\widehat{\mathcal{F}}_{b,\lambda}^{\perp}\not=\emptyset$ for all $\lambda>\Lambda_{b,0}$ in the case of $b_0\leq-\mu_{b,1}$, where $\mu_{b,1}$ is the first eigenvalue of $(-\Delta, H^1_0(\Omega_b))$.  By a similar argument as used in Lemma~\ref{lem0001}, we also can see that there exists $\Lambda_b^*>\overline{\Lambda}_{b}$ such that dim$(\widehat{\mathcal{F}}_{b,\lambda}^{\perp})$ is independent of $\lambda\geq\Lambda_{b}^*$ and dim$(\widehat{\mathcal{F}}_{b,\lambda}^{\perp})=j_{b,0}^*-1\leq k_b$ for all $\lambda\geq\Lambda_{b}^*$.%
\end{remark}

\vskip0.23in

Now, we have the following decomposition of $E_\lambda$:%
\begin{enumerate}
\item[$(1)$] $E_\lambda=\mathcal{F}_{a,\lambda}
\times\mathcal{F}_{b,\lambda}$ for $\lambda>0$ in the case of $b_0\geq a_0\geq0$.%
\item[$(2)$] $E_\lambda=(\widetilde{\mathcal{F}}_{a,\lambda}^{\perp}\oplus\mathcal{F}_{a,\lambda})
\times\mathcal{F}_{b,\lambda}$ for $\lambda>\overline{\Lambda}_{a}$ in the case of $-\mu_{a,1}<a_0<0\leq b_0$.%
\item[$(3)$] $E_{\lambda}=(\widetilde{\mathcal{F}}_{a,\lambda}^{\perp}\oplus\mathcal{F}_{a,\lambda})
\times(\widetilde{\mathcal{F}}_{b,\lambda}^{\perp}\oplus\mathcal{F}_{b,\lambda})$ for $\lambda>\max\{\overline{\Lambda}_a, \overline{\Lambda}_b\}$ in the case of $-\mu_{a,1}<a_0<0$ and $-\mu_{b,1}<b_0<0$.%
\item[$(4)$] $E_\lambda=(\widehat{\mathcal{F}}_{a,\lambda}^{\perp}\oplus\widetilde{\mathcal{F}}_{a,\lambda}^{\perp}\oplus\mathcal{F}_{a,\lambda})
\times\mathcal{F}_{b,\lambda}$ for $\lambda>\Lambda_{a,0}$ in the case of $a_0\leq-\mu_{a,1}<0\leq b_0$.%
\item[$(5)$] $E_\lambda=(\widehat{\mathcal{F}}_{a,\lambda}^{\perp}\oplus\widetilde{\mathcal{F}}_{a,\lambda}^{\perp}\oplus\mathcal{F}_{a,\lambda})
\times(\widetilde{\mathcal{F}}_{b,\lambda}^{\perp}\oplus\mathcal{F}_{b,\lambda})$ for $\lambda>\max\{\Lambda_{a,0}, \overline{\Lambda}_b\}$ in the case of $a_0\leq-\mu_{a,1}$ and $-\mu_{b,1}<b_0<0$.%
\item[$(6)$] $E_\lambda=(\widetilde{\mathcal{F}}_{a,\lambda}^{\perp}\oplus\mathcal{F}_{a,\lambda})
\times(\widehat{\mathcal{F}}_{b,\lambda}^{\perp}\oplus\widetilde{\mathcal{F}}_{b,\lambda}^{\perp}\oplus\mathcal{F}_{b,\lambda})$ for $\lambda>\max\{\Lambda_{b,0}, \overline{\Lambda}_a\}$ in the case of $-\mu_{a,1}<a_0\leq b_0\leq-\mu_{b,1}$.%
\item[$(7)$] $E_\lambda=(\widehat{\mathcal{F}}_{a,\lambda}^{\perp}\oplus\widetilde{\mathcal{F}}_{a,\lambda}^{\perp}\oplus\mathcal{F}_{a,\lambda})
\times(\widehat{\mathcal{F}}_{b,\lambda}^{\perp}\oplus\widetilde{\mathcal{F}}_{b,\lambda}^{\perp}\oplus\mathcal{F}_{b,\lambda})$ for $\lambda>\max\{\Lambda_{a,0}, \Lambda_{b,0}\}$ in the case of $a_0\leq-\mu_{a,1}$, $b_0\leq-\mu_{b,1}$.%
\end{enumerate}
Moreover, we have the following estimates.%
\begin{lemma}\label{lem0110}
Let the conditions $(D_1)$--$(D_3)$ hold and $a_0,b_0\in\bbr$.  Then%
\begin{enumerate}
\item[$(i)$] $\mathcal{D}_{a,\lambda}(u,u)=\|u\|_{a,\lambda}^2$ on $\mathcal{F}_{a,\lambda}$ and
$\mathcal{D}_{b,\lambda}(v,v)=\|v\|_{b,\lambda}^2$ on $\mathcal{F}_{b,\lambda}$ for all $\lambda>0$.%
\item[$(ii)$] $\mathcal{D}_{a,\lambda}(u,u)\geq(1-\frac{1}{\alpha_{a,j_{a,\lambda}}(\lambda)})\|u\|_{a,\lambda}^2$ on
 $\widetilde{\mathcal{F}}_{a,\lambda}^{\perp}$ and $$\mathcal{D}_{b,\lambda}(v,v)\geq(1-\frac{1}{\alpha_{b,j_{b,\lambda}}(\lambda)})\|v\|_{b,\lambda}^2$$
  on $\widetilde{\mathcal{F}}_{b,\lambda}^{\perp}$ for all $\lambda>\max\{\Lambda_{a,0}, \Lambda_{b,0}\}$, where $j_{a,\lambda}=$dim$(\widehat{\mathcal{F}}_{a,\lambda}^{\perp})+1$ and $j_{b,\lambda}=$dim$(\widehat{\mathcal{F}}_{b,\lambda}^{\perp})+1$.%
\item[$(iii)$] $\mathcal{D}_{a,\lambda}(u,u)\leq0$ on $\widehat{\mathcal{F}}_{a,\lambda}^{\perp}$ and $\mathcal{D}_{b,\lambda}(v,v)\leq0$ on $\widehat{\mathcal{F}}_{b,\lambda}^{\perp}$ for $\lambda>\max\{\Lambda_{a,0}, \Lambda_{b,0}\}$.%
\end{enumerate}
\end{lemma}
\begin{proof}
The conclusions follow immediately from the definitions of $\widehat{\mathcal{F}}_{a,\lambda}^{\perp}$, $\widetilde{\mathcal{F}}_{a,\lambda}^{\perp}$, $\mathcal{F}_{a,\lambda}$ and $\widehat{\mathcal{F}}_{b,\lambda}^{\perp}$, $\widetilde{\mathcal{F}}_{b,\lambda}^{\perp}$, $\mathcal{F}_{b,\lambda}$.%
\end{proof}

By Lemma~\ref{lem0110}, we can see that the functional $\mathcal{D}_\lambda(u,v)$ is positively definite on $E_\lambda$ in the cases of $(1)$--$(3)$ and indefinite on $E_\lambda$ in the cases of $(4)$--$(7)$.  For the sake of convenience, we always denote $$E_\lambda=(\widetilde{\mathcal{F}}_{a,\lambda}^{\perp}\oplus\mathcal{F}_{a,\lambda})
\times(\widetilde{\mathcal{F}}_{b,\lambda}^{\perp}\oplus\mathcal{F}_{b,\lambda})$$ and
$$E_\lambda=(\widehat{\mathcal{F}}_{a,\lambda}^{\perp}\oplus\widetilde{\mathcal{F}}_{a,\lambda}^{\perp}\oplus\mathcal{F}_{a,\lambda})
\times(\widehat{\mathcal{F}}_{b,\lambda}^{\perp}\oplus\widetilde{\mathcal{F}}_{b,\lambda}^{\perp}\oplus\mathcal{F}_{b,\lambda})$$ in the definite case and the indefinite case, respectively.%

\vskip0.3in

\section{The sets $\mathcal{N}_{\lambda,\beta}$, $\mathcal{M}_{\lambda,\beta}$ and $\mathcal{G}_{\lambda,\beta}$}
In this section, we will drive some properties of the sets $\mathcal{N}_{\lambda,\beta}$, $\mathcal{M}_{\lambda,\beta}$
and $\mathcal{G}_{\lambda,\beta}$.  We start by the observations on $\mathcal{M}_{\lambda,\beta}$.  It is well known
 that $\mathcal{M}_{\lambda,\beta}$ is closely linked to the so-called fibering maps of $J_{\lambda,\beta}(u,v)$,
 which are the functions defined on $\bbr^+$ and given by $\overline{T}_{\lambda,\beta,u,v}(t)=J_{\lambda,\beta}(tu,tv)$ for
 each $(u,v)\in E_\lambda\backslash\{(0,0)\}$.  Clearly, $\overline{T}_{\lambda,\beta,u,v}(t)\in C^2(\bbr^+)$.
  Moreover, $\overline{T}_{\lambda,\beta,u,v}'(t)=0$ is equivalent to $(tu,tv)\in\mathcal{M}_{\lambda,\beta}$.
   In particular, $\overline{T}_{\lambda,\beta,u,v}'(1)=0$ if and only if $(u,v)\in\mathcal{M}_{\lambda,\beta}$.%

\vskip0.2in

\begin{lemma}\label{lem0008}
Assume  $(D_1)$-$(D_3)$ hold and $\mathcal{D}_{\lambda}(u,v)$ is positively definite on $E_{\lambda}$.  Then for every $(u,v)\in E_\lambda\backslash\{(0,0)\}$ with $\mathcal{L}_{\beta}(u,v)>0$, there exists a unique
$t_{u,v}=\bigg(\frac{\mathcal{D}_\lambda(u,v)}{\mathcal{L}_\beta(u,v)}\bigg)^{\frac12}$
such that $\overline{T}_{\lambda,\beta,u,v}'(t_{u,v})=0$, $(t_{u,v}u, t_{u,v}v)\in\mathcal{M}_{\lambda,\beta}$
 and $\overline{T}_{\lambda,\beta,u,v}(t_{u,v})=\max_{t\geq0}\overline{T}_{\lambda,\beta,u,v}(t)$.  Furthermore,
 for every $(u,v)\in E_\lambda\backslash\{(0,0)\}$ with $\mathcal{L}_{\beta}(u,v)\leq0$, we have
  $\mathcal{X}_{u,v}\cap\mathcal{M}_{\lambda,\beta}=\emptyset$, where $\mathcal{X}_{u,v}=\{(tu,tv)\mid t\in\bbr^+\}$.%
\end{lemma}
\begin{proof}
The proof is very standard, so we omit it here.%
\end{proof}

Due to Lemma~\ref{lem0008}, we can see that%
\begin{equation}\label{eq1003}
m_{\lambda,\beta}^*=\inf_{E_\lambda\backslash\{(0,0)\}}\frac{\mathcal{D}_\lambda(u,v)^2}{4\mathcal{L}_\beta(u,v)}.%
\end{equation}
\begin{lemma}\label{lem5001}
Let   $(D_1)$-$(D_3)$ hold and $\mathcal{D}_{\lambda}(u,v)$ be positively  definite
on $E_{\lambda}$.  If $(u,v)$ is the minimizer of $J_{\lambda,\beta}(u,v)$ on
$\mathcal{M}_{\lambda,\beta}$, then we have $D[J_{\lambda,\beta}(u,v)]=0$ in $E_\lambda^*$.%
\end{lemma}
\begin{proof}
The proof is standard.  Since $J_{\lambda,\beta}(u,v)$ is $C^2$ in $E_\lambda$, by the method of Lagrange multipliers,
 there exists $\nu\in\bbr$ such that $D[J_{\lambda,\beta}(u,v)]-\nu D[\Psi_{\lambda,\beta}(u,v)]=0$ in $E_\lambda^*$,
 where $\Psi_{\lambda,\beta}(u,v)=\langle D[J_{\lambda,\beta}(u,v)],(u,v)\rangle_{E_\lambda^*,E_\lambda}$.
  Multiplying this equation with $(u,v)$ and noting that $(u,v)\in\mathcal{M}_{\lambda,\beta}$, we have%
\begin{equation*}
\nu\langle D[\Psi_{\lambda,\beta}(u,v)],(u,v)\rangle_{E_\lambda^*,E_\lambda}=2\nu\mathcal{D}_\lambda(u,v)=0.%
\end{equation*}
Since $\mathcal{D}_{\lambda}(u,v)$ is positively definite on $E_{\lambda}$, we must have $\nu=0$.
 It follows that $D[J_{\lambda,\beta}(u,v)]=0$ in $E_\lambda^*$, which completes the proof.%
\end{proof}

We next look at  the set $\mathcal{N}_{\lambda,\beta}$.  From the point of the fibering maps,
 $\mathcal{N}_{\lambda,\beta}$ is closely linked to the functions defined on $\bbr^+\times\bbr^+$
 and given by $T_{\lambda,\beta,u,v}(t,s)=J_{\lambda,\beta}(tu,sv)$ for each
  $(u,v)\in (E_{a,\lambda}\backslash\{0\})\times(E_{b,\lambda}\backslash\{0\})$.
   $T_{\lambda,\beta,u,v}(t,s)\in C^2(\bbr^+\times\bbr^+)$ and
    $$\frac{\partial T_{\lambda,\beta,u,v}}{\partial t}(t,s)=\frac{\partial T_{\lambda,\beta,u,v}}{\partial s}(t,s)=0$$ is
     equivalent to $(tu,sv)\in\mathcal{N}_{\lambda,\beta}$.  In particular,
      $\frac{\partial T_{\lambda,\beta,u,v}}{\partial t}(1,1)=\frac{\partial T_{\lambda,\beta,u,v}}{\partial s}(1,1)=0$
       if and only if $(u,v)\in\mathcal{N}_{\lambda,\beta}$.%
\begin{lemma}        \label{lem0002}
Assume  $(D_1)$-$(D_3)$ hold and $\beta\leq0$.  If $\mathcal{D}_{a,\lambda}(u,u)$
and $\mathcal{D}_{b,\lambda}(v,v)$ are respectively definite on $E_{a,\lambda}$ and $E_{b,\lambda}$, then we have the following.%
\begin{enumerate}
\item[$(1)$] If $(u,v)\in\mathcal{V}_{\lambda,\beta}$, then there exists a unique $(t_{\lambda,\beta}(u,v),s_{\lambda,\beta}(u,v))\in\bbr^+\times\bbr^+$ such that%
    \begin{equation*}
    (t_{\lambda,\beta}(u,v)u,s_{\lambda,\beta}(u,v)v)\in\mathcal{N}_{\lambda,\beta},%
    \end{equation*}
    where $\mathcal{V}_{\lambda,\beta}=E_\lambda\cap\mathcal{V}_\beta$ and $\mathcal{V}_\beta$ is given by \eqref{eq1020} and $t_{\lambda,\beta}(u,v)$ and $s_{\lambda,\beta}(u,v)$ are respectively given by%
    \begin{equation}
    t_{\lambda,\beta}(u,v)=\bigg(\frac{\|v\|_{L^4(\bbr^4)}^4\mathcal{D}_{a,\lambda}(u,u)
    -\beta\|u^2v^2\|_{L^1(\bbr^4)}\mathcal{D}_{b,\lambda}(v,v)}
    {\|u\|_{L^4(\bbr^4)}^4\|v\|_{L^4(\bbr^4)}^4-\beta^2\|u^2v^2\|^2_{L^1(\bbr^4)}}\bigg)^{\frac12}\label{eq0005}%
    \end{equation}
    and
    \begin{equation}
    s_{\lambda,\beta}(u,v)=\bigg(\frac{\|u\|_{L^4(\bbr^4)}^4\mathcal{D}_{b,\lambda}(v,v)
    -\beta\|u^2v^2\|_{L^1(\bbr^4)}\mathcal{D}_{a,\lambda}(u,u)}
    {\|u\|_{L^4(\bbr^4)}^4\|v\|_{L^4(\bbr^4)}^4-\beta^2\|u^2v^2\|^2_{L^1(\bbr^4)}}\bigg)^{\frac12}.\label{eq0006}%
    \end{equation}
    Moreover, $T_{\lambda,\beta,u,v}(t_{\lambda,\beta}(u,v),s_{\lambda,\beta}(u,v))=\max_{t\geq0,s\geq0}T_{\lambda,\beta,u,v}(t,s)$.  In particular, we have%
    \begin{equation}\label{eq0024}
    T_{\lambda,\beta,u,v}(1,1)=\max_{t\geq0,s\geq0}T_{\lambda,\beta,u,v}(t,s)%
    \end{equation}
    for all $(u,v)\in\mathcal{N}_{\lambda,\beta}$.
\item[$(2)$] If $(u,v)\in E_\lambda\backslash\mathcal{V}_{\lambda,\beta}$, then $\mathcal{X}_{u,v}\cap\mathcal{N}_{\lambda,\beta}=\emptyset$, where $\mathcal{X}_{u,v}=\{(tu,sv)\mid (t,s)\in\bbr^+\times\bbr^+\}$.%
\end{enumerate}
\end{lemma}
\begin{proof}
Since $\mathcal{D}_{a,\lambda}(u,u)$ and $\mathcal{D}_{b,\lambda}(v,v)$ are respectively positively definite on $E_{a,\lambda}$ and $E_{b,\lambda}$, the proof is similar to that of \cite[Lemma~3.1]{WWZ15} and only some trivial modifications are needed, so we omit the details here.%
\end{proof}

The relation between $\mathcal{N}_{\lambda,\beta}$ and $T_{\lambda,\beta,u,v}(t,s)$ for $\beta>0$ is quite different from the case of $\beta\leq0$.  In the case of $0<\beta<1$, we have from the H\"older inequality that $\mathcal{V}_{\lambda,\beta}=E_\lambda\backslash\{0\}$.  However, the properties described in Lemma~\ref{lem0002} may not hold for all $(u,v)\in\mathcal{V}_{\lambda,\beta}=E_\lambda\backslash\{0\}$ except \eqref{eq0024}.%
\begin{lemma}\label{lem0006}
Assume $(D_1)$-$(D_3)$ hold and $\beta\in(0, 1)$.  If $\mathcal{D}_{a,\lambda}(u,u)$ and $\mathcal{D}_{b,\lambda}(v,v)$ are
 positively   definite on $E_{a,\lambda}$ and $E_{b,\lambda}$  respectively, then \eqref{eq0024} holds for every $(u,v)\in\mathcal{N}_{\lambda,\beta}$.%
\end{lemma}
\begin{proof}
Suppose $(u,v)\in\mathcal{N}_{\lambda,\beta}$ and consider the following two-component systems of algebraic equations%
\begin{equation}\label{eq0023}
\left\{\aligned \mathcal{D}_{a,\lambda}(u,u)-\|u\|_{L^4(\bbr^4)}^4t-\beta\|u^2v^2\|_{L^1(\bbr^4)}s&=0,\\%
\mathcal{D}_{b,\lambda}(v,v)-\|v\|_{L^4(\bbr^4)}^4s-\beta\|u^2v^2\|_{L^1(\bbr^4)}t&=0.\endaligned\right.%
\end{equation}
Since $(u,v)\in\mathcal{N}_{\lambda,\beta}$, we can see that \eqref{eq0023} has a unique solution $(1,1)$.  It follows that $(1,1)$ is the unique critical point of $T_{\lambda,\beta,u,v}(t,s)$ in $\bbr^+\times\bbr^+$.  By the fact that $\beta\in(0, 1)$, a direct calculation and the H\"older inequality, we have%
\begin{equation*}
\frac{\partial^2T_{\lambda,\beta,u,v}}{\partial t^2}(1,1)=-2\|u\|_{L^4(\bbr^4)}^4<0%
\end{equation*}
and%
\begin{eqnarray*}
&&\begin{vmatrix}
\frac{\partial^2T_{\lambda,\beta,u,v}}{\partial t^2}(1,1)&\frac{\partial^2T_{\lambda,\beta,u,v}}{\partial t\partial s}(1,1)\\%
\frac{\partial^2T_{\lambda,\beta,u,v}}{\partial s\partial t}(1,1)&\frac{\partial^2T_{\lambda,\beta,u,v}}{\partial s^2}(1,1)%
\end{vmatrix}\\
&&=4(\|u\|_{L^4(\bbr^4)}^4\|v\|_{L^4(\bbr^4)}^4-\beta^2\|u^2v^2\|^2_{L^1(\bbr^4)})>0.%
\end{eqnarray*}
Note that $\mathcal{D}_{a,\lambda}(u,u)$ and $\mathcal{D}_{b,\lambda}(v,v)$ are positively definite on $E_{a,\lambda}$ and $E_{b,\lambda}$
respectively and $\beta\in(0, 1)$, by Lemma~\ref{lem0110}, \eqref{eq0003} and a standard argument,
we can obtain that $T_{\lambda,\beta,u,v}(t,s)>0$ for $|(t,s)|$ sufficiently
small and $T_{\lambda,\beta,u,v}(t,s)\to-\infty$ as $|(t,s)|\to+\infty$.
These imply that $(1,1)$ is the global maximum point of $T_{\lambda,\beta,u,v}(t,s)$ in
 $\bbr^+\times\bbr^+$, i.e.,  \eqref{eq0024} holds.%
\end{proof}

\begin{remark}
The relation between $\mathcal{N}_{\lambda,\beta}$ and $T_{\lambda,\beta,u,v}(t,s)$ for $\beta\geq1$
 is much more complicated than that of $\beta<1$ and even $\eqref{eq0024}$ does  not hold for some $(u,v)\in\mathcal{N}_{\lambda,\beta}$ in this case.%
\end{remark}

\begin{lemma}\label{lem5005}
Suppose that $(D_1)$-$(D_3)$ hold and $\mathcal{D}_{\lambda}(u,v)$ is positively
 definite on $E_{\lambda}$.  If $(u,v)$ is the minimizer of $J_{\lambda,\beta}(u,v)$ on
 $\mathcal{N}_{\lambda,\beta}$ with $\beta<1$, then $D[J_{\lambda,\beta}(u,v)]=0$ in $E_\lambda^*$.%
\end{lemma}
\begin{proof}
The proof is also standard.  We only give the proof for the case of $-\mu_{a,1}<a_0<0$ and $-\mu_{b,1}<b_0<0$, since other cases are more simple and can be proved in a similar way due to Lemma~\ref{lem0110}.  Since $J_{\lambda,\beta}(u,v)$ is $C^2$ in $E_\lambda$, by the method of Lagrange multipliers, there exists $\nu_1,
\nu_2\in\bbr$ such that
 $$D[J_{\lambda,\beta}(u,v)]-\nu_1D[\Psi^{*,1}_{\lambda,\beta}(u,v)]-\nu_2D[\Psi^{*,2}_{\lambda,\beta}(u,v)]=0 \hbox{ in } E_\lambda^*,$$
 where $$\Psi^{*,1}_{\lambda,\beta}(u,v)=\langle D[J_{\lambda,\beta}(u,v)],(u,0)\rangle_{E_\lambda^*,E_\lambda},
 \Psi^{*,2}_{\lambda,\beta}(u,v)=\langle D[J_{\lambda,\beta}(u,v)],(0,v)\rangle_{E_\lambda^*,E_\lambda}.$$
   Multiplying this equation with $(u,0)$ and $(0,v)$ respectively and noting that $(u,v)\in\mathcal{N}_{\lambda,\beta}$, we have%
\begin{eqnarray*}
\left\{\aligned 2\nu_1\|u\|_{L^4(\bbr^4)}^4+2\nu_2\beta\|u^2v^2\|_{L^1(\bbr^4)}&=0,\\%
2\nu_2\|v\|_{L^4(\bbr^4)}^4+2\nu_1\beta\|u^2v^2\|_{L^1(\bbr^4)}&=0.\endaligned\right.%
\end{eqnarray*}
It follows that either $\nu_1=\nu_2=0$ or $\|u\|_{L^4(\bbr^4)}^4\|v\|_{L^4(\bbr^4)}^4-\beta^2\|u^2v^2\|_{L^1(\bbr^4)}^2=0$.
By the H\"older inequality and Lemma~\ref{lem0002}, we can see that
 $\mathcal{N}_{\lambda,\beta}\subset\mathcal{V}_{\lambda,\beta}$ with $\beta<1$.
 Therefore, we must have $\nu_1=\nu_2=0$, which implies $D[J_{\lambda,\beta}(u,v)]=0$ in $E_\lambda^*$.%
\end{proof}

\vskip0.31in

Next we consider  the set $\mathcal{G}_{\lambda,\beta}$.  Since $\mathcal{G}_{\lambda,\beta}$ is
modified from $\mathcal{M}_{\lambda,\beta}$, the firbering maps $\overline{T}_{\lambda,\beta,u,v}(t)$
also need to be modified.  For every $(u,v)\in E_\lambda\backslash\{(0, 0)\}$, we define
 $$G_{\lambda,\beta,u,v}(w,\sigma,t):\widehat{\mathcal{F}}_{a,\lambda}^{\perp}\times\widehat{\mathcal{F}}_{b,\lambda}^{\perp}\times\bbr^+\to\bbr$$
  by $G_{\lambda,\beta,u,v}(w,\sigma,t)=J_{\lambda,\beta}(w+t\widetilde{u},\sigma+t\widetilde{v})$, then $G(w,\sigma,t)$ is $C^2$ in $\widehat{\mathcal{F}}_{a,\lambda}^{\perp}\times\widehat{\mathcal{F}}_{b,\lambda}^{\perp}\times\bbr^+$, where $\widetilde{u}$ and $\widetilde{v}$ are the projections of $u$ and $v$ on $\widetilde{\mathcal{F}}_{a,\lambda}^{\perp}\oplus\mathcal{F}_{a,\lambda}$ and $\widetilde{\mathcal{F}}_{b,\lambda}^{\perp}\oplus\mathcal{F}_{b,\lambda}$.  In what follows, we will borrow some ideas from \cite{SW09} to observe the set $\mathcal{G}_{\lambda,\beta}$ by $G_{\lambda,\beta,u,v}(w,\sigma,t)$.%
\begin{lemma}\label{lem0004}
Assume $(D_1)$-$(D_3)$ hold and either $a_0\leq-\mu_{a,1}$ or $b_0\leq-\mu_{b,1}$.  If $0\leq\beta<1$ and $\lambda>\max\{\overline{\Lambda}_a, \overline{\Lambda}_b\}$, then for every $(u,v)\in\mathcal{G}_{\lambda,\beta}$, $G_{\lambda,\beta,u,v}(\widehat{u},\widehat{v},1)\geq G_{\lambda,\beta,u,v}(w,\sigma,t)$ for all $t\in(0, +\infty)$ and $(w,\sigma)\in\widehat{\mathcal{F}}_{a,\lambda}^{\perp}\times\widehat{\mathcal{F}}_{b,\lambda}^{\perp}$, where $(\widehat{u},\widehat{v})$ is the projection of $(u,v)$ in $\widehat{\mathcal{F}}_{a,\lambda}^{\perp}\times\widehat{\mathcal{F}}_{b,\lambda}^{\perp}$.  Furthermore, the equality holds if and only if $t=1$ and $(w,\sigma)=(\widehat{u},\widehat{v})$.%
\end{lemma}
\begin{proof}
We only give the proof for the case of $a_0\leq-\mu_{a,1}$ and $b_0\leq-\mu_{b,1}$, since the proofs of other cases are similar and more simple due to Lemma~\ref{lem0110}.  Suppose $(u,v)\in\mathcal{G}_{\lambda,\beta}$, $t\in(0, +\infty)$ and $(w,\sigma)\in\widehat{\mathcal{F}}_{a,\lambda}^{\perp}\times\widehat{\mathcal{F}}_{b,\lambda}^{\perp}$.  Then we have%
\begin{eqnarray*}
&&J_{\lambda,\beta}(u,v)-J_{\lambda,\beta}(tu+w,tv+\sigma)\notag\\%
& &=\frac12(\mathcal{D}_{\lambda}(u,v)-\mathcal{D}_{\lambda}(tu+w,tv+\sigma))\notag-\frac14\int_{\bbr^4}u^4+v^4+2\beta u^2v^2dx\notag\\%
&&\quad +\frac14\int_{\bbr^4}(tu+w)^4+(tv+\sigma)^4+2\beta(tu+w)^2(tv+\sigma)^2dx\notag\\%
& &=\frac12((1-t^2)\mathcal{D}_{a,\lambda}(u,u)-2t\mathcal{D}_{a,\lambda}(u,w)-\mathcal{D}_{a,\lambda}(w,w))\\
& & \quad +\frac14\int_{\bbr^4}(tu+w)^4-u^4dx\notag\\%
&&\quad +\frac12((1-t^2)\mathcal{D}_{b,\lambda}(v,v)-2t\mathcal{D}_{b,\lambda}(v,\sigma)-\mathcal{D}_{b,\lambda}(\sigma,\sigma))\\
&&\quad +\frac14\int_{\bbr^4}(tv+\sigma)^4-v^4dx\notag\\%
&&\quad +\frac14\int_{\bbr^4}2\beta((tu+w)^2(tv+\sigma)^2-u^2v^2)dx.%
\end{eqnarray*}
It follows from the definition of $\mathcal{G}_{\lambda,\beta}$, $(w,\sigma)\in\widehat{\mathcal{F}}_{a,\lambda}^{\perp}\times\widehat{\mathcal{F}}_{b,\lambda}^{\perp}$ and Lemma~\ref{lem0110} that%
\begin{eqnarray*}
&&J_{\lambda,\beta}(u,v)-J_{\lambda,\beta}(tu+w,tv+\sigma)\\%
&&=-\frac12\mathcal{D}_{a,\lambda}(w,w)+\frac14\int_{\bbr^4}(tu+w)^4-u^4+2(1-t^2)u^4-4tu^3wdx\\%
&&\quad -\frac12\mathcal{D}_{b,\lambda}(\sigma,\sigma)+\frac14\int_{\bbr^4}(tv+\sigma)^4-v^4+2(1-t^2)v^4-4tv^3\sigma dx\\%
&&\quad +\frac\beta2\int_{\bbr^4}(tu+w)^2(tv+\sigma)^2-u^2v^2+(2-2t^2)u^2v^2dx\\%
&&\quad -\beta\int_{\bbr^4}tv^2uw+tu^2v\sigma dx\\%
& &\geq \frac14\int_{\bbr^4}(tu+w)^4+u^4-2u^2(tu+w)^2+2u^2w^2dx\\%
&& \quad +\frac14\int_{\bbr^4}(tv+\sigma)^4+v^4-2v^2(tv+\sigma)^2+2v^2\sigma^2dx\\%
&&\quad+\frac\beta2\int_{\bbr^4}((tu+w)^2-u^2)((tv+\sigma)^2-v^2)+v^2w^2+u^2\sigma^2dx\\%
& &= \frac14\int_{\bbr^4}\Big((tu+w)^2-u^2)^2+((tv+\sigma)^2-v^2)^2+\\
& &\quad 2\beta((tu+w)^2-u^2)((tv+\sigma)^2-v^2\Big)dx\\%
&&\quad +\frac{1}{2}\int_{\bbr^4}u^2w^2+v^2\sigma^2+\beta v^2w^2+\beta u^2\sigma^2dx.%
\end{eqnarray*}
Since $\beta\in[0, 1)$, we have%
\begin{equation*}
J_{\lambda,\beta}(u,v)-J_{\lambda,\beta}(tu+w,tv+\sigma)\geq\frac14\int_{\bbr^4}(|(tu+w)^2-u^2|-|(tv+\sigma)^2-v^2|)^2dx\geq0%
\end{equation*}
and the equalities hold if and only if $t=1$ and $(w,\sigma)=(0,0)$.%
\end{proof}

\vskip0.3in

By Lemma~\ref{lem0004}, we have the following important observation for $\mathcal{G}_{\lambda,\beta}$.%
\begin{lemma}\label{lem0011}
Assume $(D_1)$-$(D_3)$ hold and either $a_0\leq-\mu_{a,1}$ or $b_0\leq-\mu_{b,1}$.  If $0\leq\beta<1$ and $\lambda>\max\{\overline{\Lambda}_a, \overline{\Lambda}_b\}$, then for every $(u,v)\in \widetilde{E}_\lambda$, there exists a unique $(w^0_\lambda,\sigma^0_\lambda,t^0_\lambda)\in \widehat{\mathcal{F}}_{a,\lambda}^{\perp}\times\widehat{\mathcal{F}}_{b,\lambda}^{\perp}\times\bbr^+$ such that $(u_{\lambda,\beta}^0,v_{\lambda,\beta}^0)=(w^0_\lambda+t^0_\lambda\widetilde{u},\sigma^0_\lambda+t^0_\lambda\widetilde{v})\in\mathcal{G}_{\lambda,\beta}$, where $\widetilde{u}$ and $\widetilde{v}$ are the projections of $u$ and $v$ on $\widetilde{\mathcal{F}}_{a,\lambda}^{\perp}\oplus\mathcal{F}_{a,\lambda}$ and $\widetilde{\mathcal{F}}_{b,\lambda}^{\perp}\oplus\mathcal{F}_{b,\lambda}$.  Furthermore, we also have%
\begin{equation}\label{eq0057}
G_{\lambda,\beta,u,v}(w^0_\lambda,\sigma^0_\lambda,t^0_\lambda)=
\max_{\widehat{\mathcal{F}}_{a,\lambda}^{\perp}\times\widehat{\mathcal{F}}_{b,\lambda}^{\perp}\times\bbr^+}G_{\lambda,\beta,u,v}(w,\sigma,t).%
\end{equation}
\end{lemma}
\begin{proof}
We only give the proof for the case of $a_0\leq-\mu_{a,1}$ and $b_0\leq-\mu_{b,1}$, since the proofs of other cases are similar and more simple due to Lemma~\ref{lem0110}.
Clearly, $(\widehat{\mathcal{F}}_{a,\lambda}^{\perp}\oplus\bbr^+ u)\times(\widehat{\mathcal{F}}_{b,\lambda}^{\perp}\oplus\bbr^+ v)=(\widehat{\mathcal{F}}_{a,\lambda}^{\perp}\oplus\bbr^+ \widetilde{u})\times(\widehat{\mathcal{F}}_{b,\lambda}^{\perp}\oplus\bbr^+ \widetilde{v})$ for every $(u,v)\in E_\lambda$ with $\lambda>\max\{\overline{\Lambda}_a, \overline{\Lambda}_b\}$.  By the definitions of $\widehat{\mathcal{F}}_{a,\lambda}^{\perp}$ and $\widehat{\mathcal{F}}_{b,\lambda}^{\perp}$, we can also see that dim$((\widehat{\mathcal{F}}_{a,\lambda}^{\perp}\oplus\bbr^+ \widetilde{u})\times(\widehat{\mathcal{F}}_{b,\lambda}^{\perp}\oplus\bbr^+ \widetilde{v}))<+\infty$ for every $(u,v)\in E_\lambda$ with $\lambda>\max\{\overline{\Lambda}_a, \overline{\Lambda}_b\}$.  On the other hand, since $(u,v)\in \widetilde{E}_\lambda$ with $\lambda>\max\{\overline{\Lambda}_a, \overline{\Lambda}_b\}$, we must have $\widetilde{u}\not=0$ or $\widetilde{v}\not=0$.  Since $0\leq\beta<1$, by Lemma~\ref{lem0110} and \eqref{eq0003}, we have that%
\begin{eqnarray}
&&J_{\lambda,\beta}(t\widetilde{u},t\widetilde{v})\notag\\%
&=&\frac{t^2}2(\mathcal{D}_{a,\lambda}(\widetilde{u},\widetilde{u})+\mathcal{D}_{b,\lambda}(\widetilde{v},\widetilde{v}))\notag\\%
&&-\frac{t^4}4(\|\widetilde{u}\|^4_{L^4(\bbr^4)}+\|\widetilde{v}\|^4_{L^4(\bbr^4)}+2\beta\|\widetilde{u}^2\widetilde{v}^2\|^1_{L^1(\bbr^4)})\notag\\%
&\geq&\frac{t^2d_{\lambda}^*}2(\|\widetilde{u}\|_{a,\lambda}^2+\|\widetilde{v}\|_{b,\lambda}^2)
-\frac{t^4S^{-2}}4(\|\widetilde{u}\|_{a,\lambda}^2+\|\widetilde{v}\|_{b,\lambda}^2)^2\notag\\%
&=&\frac{t^2}{4}(\|\widetilde{u}\|_{a,\lambda}^2+\|\widetilde{v}\|_{b,\lambda}^2)(2d_\lambda^*-t^2S^{-2}
(\|\widetilde{u}\|_{a,\lambda}^2+\|\widetilde{v}\|_{b,\lambda}^2))\label{eq0052}\\%
&>&0\notag%
\end{eqnarray}
for $t>0$ sufficiently small, where%
\begin{equation}\label{eq5002}
d_\lambda^*=\min\bigg\{1-\frac{1}{\alpha_{a, j_{a,\lambda}}(\lambda)}, 1-\frac{1}{\alpha_{a, j_{b,\lambda}}(\lambda)}\bigg\}>0%
\end{equation}
and $j_{a,\lambda}$ and $j_{b,\lambda}$ are given by Lemma~\ref{lem0110}.
Note that $0\leq\beta<1$, then for every $(w, \sigma)\in(\widehat{\mathcal{F}}_{a,\lambda}^{\perp}\oplus\bbr^+ \widetilde{u})\times(\widehat{\mathcal{F}}_{b,\lambda}^{\perp}\oplus\bbr^+\widetilde{v})$ with $\|w\|_{a,\lambda}^2+\|\sigma\|_{b,\lambda}^2=1$, we have from the H\"older inequality that $\mathcal{L}_{\beta}(w,\sigma)>0$,
which then implies%
\begin{eqnarray*}
J_{\lambda,\beta}(Rw,R\sigma)&\leq&\frac{R^2}{2}
-\frac{R^4}{4}(\|w\|_{L^4(\bbr^4)}^4+\|\sigma\|_{L^4(\bbr^4)}^4+2\beta\|w^2\sigma^2\|_{L^1(\bbr^4)})\to-\infty%
\end{eqnarray*}
as $R\to+\infty$.  Since dim$((\widehat{\mathcal{F}}_{a,\lambda}^{\perp}\oplus\bbr^+ \widetilde{u})\times(\widehat{\mathcal{F}}_{b,\lambda}^{\perp}\oplus\bbr^+ \widetilde{v}))<+\infty$, there exists $R_\lambda>0$ such that%
\begin{equation}\label{eq9003}
J_{\lambda,\beta}(R_\lambda w,R_\lambda \sigma)\leq-1%
\end{equation}
for all $(w,\sigma)\in(\widehat{\mathcal{F}}_{a,\lambda}^{\perp}\oplus\bbr^+ \widetilde{u})\times(\widehat{\mathcal{F}}_{b,\lambda}^{\perp}\oplus\bbr^+ \widetilde{v})$ with $\|w\|_{a,\lambda}^2+\|\sigma\|_{b,\lambda}^2=1$.
Since $G_{\lambda,\beta,u,v}(w,\sigma,t)$ is of $C^2$ in
 $\widehat{\mathcal{F}}_{a,\lambda}^{\perp}\times\widehat{\mathcal{F}}_{b,\lambda}^{\perp}\times\bbr^+$,
noting  $$\dim ((\widehat{\mathcal{F}}_{a,\lambda}^{\perp}\oplus\bbr^+ \widetilde{u})
\times(\widehat{\mathcal{F}}_{b,\lambda}^{\perp}\oplus\bbr^+ \widetilde{v}))<+\infty,$$ there exists $(w^0_\lambda,\sigma^0_\lambda,t^0_\lambda)\in \widehat{\mathcal{F}}_{a,\lambda}^{\perp}\times\widehat{\mathcal{F}}_{b,\lambda}^{\perp}\times\bbr^+$ such that \eqref{eq0057} holds.
It follows that $(w^0_\lambda,\sigma^0_\lambda,t^0_\lambda)$ is a critical point of $G_{\lambda,\beta,u,v}(w,\sigma,t)$ in $\widehat{\mathcal{F}}_{a,\lambda}^{\perp}\times\widehat{\mathcal{F}}_{b,\lambda}^{\perp}\times\bbr^+$.  Therefore, $(u_{\lambda,\beta}^0,v_{\lambda,\beta}^0)=(w^0_\lambda+t^0_\lambda\widetilde{u},\sigma^0_\lambda+t^0_\lambda\widetilde{v})\in\mathcal{G}_{\lambda,\beta}$.  Note that $(u_{\lambda,\beta}^0,v_{\lambda,\beta}^0)\in\mathcal{G}_{\lambda,\beta}$ and $(w^0_\lambda,\sigma^0_\lambda,t^0_\lambda)$ satisfy \eqref{eq0057}, by Lemma~\ref{lem0004}, $(w^0_\lambda,\sigma^0_\lambda,t^0_\lambda)$ must be unique, which completes the proof.%
\end{proof}

\begin{remark}\label{rmk0002}
By Lemma~\ref{lem0001} and Remark~\ref{rmk0001}, we have $j_{a,\lambda}=j_{a,0}^*+1\leq k_a+1$ for $\lambda>\Lambda_a^*$ and $j_{b,\lambda}=j_{b,0}^*+1\leq k_b+1$ for $\lambda>\Lambda_b^*$.  It follows from Lemma~\ref{lem0003} that $d_\lambda^*$ and $R_\lambda$ given by \eqref{eq5002} and \eqref{eq9003} respectively are both independent of $\lambda$ sufficient large, say $\lambda\geq\Lambda_0^*\geq\max\{\Lambda_a^*, \Lambda_b^*\}$.%
\end{remark}

In what follows, we will give some estimates of $J_{\lambda,\beta}(u,v)$ on the sets
 $\mathcal{N}_{\lambda,\beta}$, $\mathcal{M}_{\lambda,\beta}$ and $\mathcal{G}_{\lambda,\beta}$.
  More precisely, we will give some estimates of $m_{\lambda,\beta}$, $m_{\lambda,\beta}^*$
  and $c_{\lambda,\beta}$.  We begin with the estimates of the upper boundary to $m_{\lambda,\beta}$
  and $m_{\lambda,\beta}^*$.  Let $I_{\Omega_a}(u)$ and $I_{\Omega_b}(v)$ be two functionals respectively
  defined on $H_0^1(\Omega_a)$ and $H_0^1(\Omega_b)$, which are given by%
\begin{equation*}
I_{\Omega_a}(u):=\frac12\int_{\Omega_a}|\nabla u|^2+a_0u^2dx-\frac{1}{4}\int_{\Omega_a}u^4dx,
\end{equation*}
\begin{equation*}
 I_{\Omega_b}(v):=\frac12\int_{\Omega_b}|\nabla v|^2+b_0v^2dx-\frac{1}{4}\int_{\Omega_b}v^4dx.%
\end{equation*}
Then it is well known that $I_{\Omega_a}(u)$ and $I_{\Omega_b}(v)$ are of $C^2$ in $H_0^1(\Omega_a)$ and $H_0^1(\Omega_b)$, respectively.  Define%
\begin{equation*}
\mathcal{N}_a:=\{u\in H_0^1(\Omega_a)\backslash\{0\}\mid I_{\Omega_a}'(u)u=0\},
\end{equation*}
\begin{equation*}
\mathcal{N}_b:=\{v\in H_0^1(\Omega_b)\backslash\{0\}\mid I_{\Omega_b}'(v)v=0\}.%
\end{equation*}
Then it is easy to show that $\mathcal{N}_a$ and $\mathcal{N}_b$ are all nonempty.
Let
 $$m_a:=\inf_{\mathcal{N}_a}I_a(u), \quad m_b=\inf_{\mathcal{N}_b}I_b(v).$$
Then it is   well known that $m_a=\frac14 S^2$ in the case of $a_0\geq0$ and
  $m_a<\frac14S^2$ in the case of $-\mu_{a,1}<a_0<0$ while $m_b=\frac14 S^2$ in the case
   of $b_0\geq0$ and $m_b<\frac14S^2$ in the case of $-\mu_{b,1}<b_0<0$ due to the
   condition $(D_3)$ (cf. \cite{S96}).

\begin{lemma}\label{lem5003}
Let   $(D_1)$-$(D_3)$ hold and $\mathcal{D}_{\lambda}(u,v)$ be positively  definite in $E_\lambda$.  Then $m_a+m_b\geq m_{\lambda,\beta}$ and $\min\{m_a, m_b\}\geq m_{\lambda,\beta}^*$ for all $\beta\in\bbr$.%
\end{lemma}
\begin{proof}
Without loss of generality, we assume $m_a\leq m_b$.  Since $\mathcal{N}_a\times\mathcal{N}_b\subset\mathcal{N}_{\lambda,\beta}$ and $\mathcal{N}_a\times\{0\}\subset\mathcal{M}_{\lambda,\beta}$ by the condition $(D_3)$, the conclusion follows immediately from a similar argument as used in \cite[Lemma~3.2]{WWZ15}.%
\end{proof}

We next give some estimates of the lower bound  of  $m_{\lambda,\beta}$ and $m_{\lambda,\beta}^*$.  Let%
\begin{eqnarray*}
I_{a,\lambda}(u)=\frac12\mathcal{D}_{a,\lambda}(u,u)-\frac{1}{4}\|u\|_{L^4(\bbr^4)}^4\quad\text{and}\quad
I_{b,\lambda}(v)=\frac12\mathcal{D}_{b,\lambda}(v,v)-\frac{1}{4}\|v\|_{L^4(\bbr^4)}^4.%
\end{eqnarray*}
Then by \eqref{eq0001}--\eqref{eq0003}, $I_{a,\lambda}(u)$ is well defined on $E_{a,\lambda}$ and $I_{b,\lambda}(v)$ is
 well defined on $E_{b,\lambda}$ respectively for $\lambda>\max\{\overline{\Lambda}_a, \overline{\Lambda}_b\}$.
 Moreover, by a standard argument, we can see that $I_{a,\lambda}(u)$ and $I_{b,\lambda}(v)$ are of $C^2$ in
  $E_{a,\lambda}$ and $E_{b,\lambda}$, respectively.  Denote%
\begin{eqnarray}   \label{eq1149}
\mathcal{N}_{a,\lambda}=\{u\in E_{a,\lambda}\backslash\{0\}\mid I_{a,\lambda}'(u)u=0\}, \\
\mathcal{N}_{b,\lambda}=\{u\in E_{b,\lambda}\backslash\{0\}\mid I_{b,\lambda}'(u)u=0\}.%
\end{eqnarray}
Then $\mathcal{N}_{a,\lambda}$ and $\mathcal{N}_{b,\lambda}$ are nonempty
if $\mathcal{D}_{a,\lambda}(u,u)$ and $\mathcal{D}_{b,\lambda}(v,v)$ are  positively  definite in $E_{a,\lambda}$ and $E_{b,\lambda}$ respectively.%

\begin{lemma}\label{lem5002}
Assume  that $(D_1)$-$(D_3)$ hold and $\mathcal{D}_{a,\lambda}(u,u)$ and $\mathcal{D}_{b,\lambda}(v,v)$ are
  positively definite on $E_{a,\lambda}$ and $E_{b,\lambda}$ respectively, then for $\beta\leq0$,
 we have $m_{\lambda,\beta}\geq m_{a,\lambda}+m_{b,\lambda}$ and
 $m_{\lambda,\beta}^*\geq\min\{m_{a,\lambda}, m_{b,\lambda}\}$, where $m_{a,\lambda}=\inf_{\mathcal{N}_{a,\lambda}}I_{a,\lambda}(u)$
 and $m_{b,\lambda}=\inf_{\mathcal{N}_{b,\lambda}}I_{b,\lambda}(v)$.%
\end{lemma}
\begin{proof}
Since $\mathcal{D}_{a,\lambda}(u,u)$ and $\mathcal{D}_{b,\lambda}(v,v)$ are  positively  definite on
 $E_{a,\lambda}$ and $E_{b,\lambda}$ respectively and Lemma~\ref{lem0002} holds, the proof of $m_{\lambda,\beta}\geq m_{a,\lambda}+m_{b,\lambda}$ is similar to \cite[Lemma~3.2]{WWZ15}.  For the proof of $m_{\lambda,\beta}^*\geq\min\{m_{a,\lambda}, m_{b,\lambda}\}$, note that by $\beta\leq0$, we must have $\min\{\mathcal{D}_{a,\lambda}(u,u)-\|u\|_{L^4(\bbr^4)}^4, \mathcal{D}_{b,\lambda}(v,v)-\|v\|_{L^4(\bbr^4)}^4\}\leq0$ for all $(u,v)\in\mathcal{M}_{\lambda,\beta}$.  It follows that there exists $t\in(0, 1]$ such that either $tu\in\mathcal{N}_{a,\lambda}$ or $tv\in\mathcal{N}_{b,\lambda}$, which together with Lemma~\ref{lem0008}, implies $m_{\lambda,\beta}^*\geq\min\{m_{a,\lambda}, m_{b,\lambda}\}$.%
\end{proof}

When $\mathcal{D}_{\lambda}(u,v)$ is positively indefinite in $E_\lambda$, the situation is somewhat different.
In this case, due to Lemma~\ref{lem0110}, we have that either $a_0\leq-\mu_{a,1}$ or $b_0\leq-\mu_{b,1}$
 if $\lambda>\max\{\overline{\Lambda}_a, \overline{\Lambda}_b\}$.%

\begin{lemma}\label{lem0007}
Assume that  $(D_1)$-$(D_3)$ hold and   that either $a_0\leq-\mu_{a,1}$ or $b_0\leq-\mu_{b,1}$.  If $0\leq\beta<1$ and $\lambda\geq\Lambda_0^*$, then we have%
\begin{equation*}
c_{\lambda,\beta}\geq\alpha_0>0,%
\end{equation*}
where $\alpha_0>0$ is a constant independent of $\beta\in(-1, 1)$ and $\lambda\geq\Lambda_0^*$.%
\end{lemma}
\begin{proof}
As in the proof of Lemma~\ref{lem0004}, we only give the proof for the case of $a_0\leq-\mu_{a,1}$ and $b_0\leq-\mu_{b,1}$, since the proofs of other cases are similar and more simple due to Lemma~\ref{lem0110}.  Let $(u,v)\in \mathcal{G}_{\lambda,\beta}$.  Then $u=\widehat{u}+\widetilde{u}$ and $v=\widehat{v}+\widetilde{v}$ with $\widetilde{u}\not=0$ or $\widetilde{v}\not=0$, where $\widehat{u}$, $\widetilde{u}$, $\widehat{v}$ and $\widetilde{v}$ are the projections of $u$ and $v$ on $\widehat{\mathcal{F}}_{a,\lambda}^{\perp}$, $\widetilde{\mathcal{F}}_{a,\lambda}^{\perp}\oplus\mathcal{F}_{a,\lambda}$, $\widehat{\mathcal{F}}_{b,\lambda}^{\perp}$ and $\widetilde{\mathcal{F}}_{b,\lambda}^{\perp}\oplus\mathcal{F}_{b,\lambda}$, respecitvely.  By a similar argument as used in \eqref{eq0052}, we can see that%
\begin{eqnarray*}
J_{\lambda,\beta}(u,v)&\geq&\frac{t^2}{4}(\|\widetilde{u}\|_{a,\lambda}^2+\|\widetilde{v}\|_{b,\lambda}^2)(2d_\lambda^*-t^2S^{-2}
(\|\widetilde{u}\|_{a,\lambda}^2+\|\widetilde{v}\|_{b,\lambda}^2))%
\end{eqnarray*}
for all $t\geq0$, where $d_\lambda^*$ is given by \eqref{eq5002}.  Since $\widetilde{u}\not=0$ or $\widetilde{v}\not=0$, there exists $t_\lambda\in(0, +\infty)$ such that $t_\lambda^2(\|\widetilde{u}\|_{a,\lambda}^2+\|\widetilde{v}\|_{b,\lambda}^2)=d_\lambda^* S^{2}$.  It follows that $J_{\lambda,\beta}(u,v)\geq\frac{(d_\lambda^*)^2S^2}{4}$.  Note that $(u,v)\in\mathcal{G}_{\lambda,\beta}$ is arbitrary, we must have $c_{\lambda,\beta}\geq\frac{(d_\lambda^*)^2S^2}{4}:=\alpha_\lambda>0$ for $0\leq\beta<1$ and $\lambda>\Lambda_0^*$.  It remains to show that $\alpha_\lambda\geq\alpha_0>0$ for some $\alpha_0$ independent of $0\leq\beta<1$ and $\lambda\geq\Lambda_0^*$.  Indeed, by Lemma~\ref{lem0001} and Remark~\ref{rmk0001}, $j_{a,\lambda}=j_{a,0}^*+1\leq k_a+1$ for $\lambda>\Lambda_a^*$ and $j_{b,\lambda}=j_{b,0}^*+1\leq k_b+1$ for $\lambda>\Lambda_b^*$.  Then by Lemma~\ref{lem0003}, we have%
\begin{equation*}
d_{\lambda}^*\geq\min\bigg\{1-\frac{1}{\alpha_{a, j_{a,0}^*+1}(\Lambda_0^*)}, 1-\frac{1}{\alpha_{a, j_{b,0}^*+1}(\Lambda_0^*)}\bigg\}>0.%
\end{equation*}
We close the proof by taking $\alpha_0=\min\bigg\{1-\frac{1}{\alpha_{a, j_{a,0}^*+1}(\Lambda_0^*)}, 1-\frac{1}{\alpha_{a, j_{b,0}^*+1}(\Lambda_0^*)}\bigg\}$.%
\end{proof}

If we also have $-a_0\not\in\sigma(-\Delta, H_0^1(\Omega_a))$, then by the condition $(D_3)$ and the results of \cite{CSZ12,SWW09}, $I_a(u)$ has a least energy critical point $u_a$ in $H_0^1(\Omega_a)$ with the energy value $0<I_a(u_a)<\frac14 S^2$.  Similarly, if $-b_0\not\in\sigma(-\Delta, H_0^1(\Omega_b))$, then $I_b(v)$ has a least energy critical point $v_b$ in $H_0^1(\Omega_b)$ with the energy value $0<I_b(v_b)<\frac14 S^2$ due to the condition $(D_3)$ and the results of \cite{CSZ12,SWW09}.%
\begin{lemma}\label{lem0009}
Let   $(D_1)$-$(D_3)$ hold. Further, assume  that either $a_0\leq-\mu_{a,1}$ with $-a_0\not\in\sigma(-\Delta, H_0^1(\Omega_a))$ or
$b_0\leq-\mu_{b,1}$ with $-b_0\not\in\sigma(-\Delta, H_0^1(\Omega_b))$, then we have that%
\begin{equation*}
\frac14S^2> \limsup_{\lambda\to+\infty}c_{\lambda,\beta}\quad\text{for all }0\leq\beta<1.%
\end{equation*}
\end{lemma}
\begin{proof}
We only give the proof for the case of $a_0\leq-\mu_{a,1}$ with $-a_0\not\in\sigma(-\Delta, H_0^1(\Omega_a))$, since another case can be proved in a similar way.  Since $u_a$ is a nonzero critical point of $I_a(u)$ with the energy value $0<I_a(u_a)<\frac14 S^2$, we can see that $\int_{\Omega_a}(|\nabla u_a|^2+a_0u_a^2)dx>0$.  We claim that there exists $\Lambda_0\geq\Lambda_0^*$ such that $u_a\in E_{a,\lambda}\backslash\widehat{\mathcal{F}}_{a,\lambda}^{\perp}$ for $\lambda>\Lambda_0$.  Indeed, by the condition $(D_3)$ once more, we can see that $u_a\in E_{a,\lambda}$ for $\lambda>\Lambda_{a,0}$.  If there exists $\{\lambda_n\}$ with $\lambda_n\to+\infty$ as $n\to\infty$ such that $u_a\in\widehat{\mathcal{F}}_{a,\lambda_n}^{\perp}$, then by Lemmas~\ref{lem0003} and \ref{lem0001} and the definition of $\widehat{\mathcal{F}}_{a,\lambda_n}^{\perp}$, we must have $u_a=\sum_{i=1}^{j_{a,0}^*} d_ie_{a,i}$, where $j_{a,0}^*$ is given by \eqref{eq1015} and $e_{a,i}$ satisfy \eqref{eq0130} with $\alpha_{a,i}\leq1$ for all $i=1,2,\cdots j_{a,0}^*$.  This implies $\int_{\Omega_a}(|\nabla u_a|^2+a_0u_a^2)dx\leq0$ and it is impossible.  Now, by Lemma~\ref{lem0011}, there exists a unique $(w^0_\lambda,\sigma^0_\lambda,t^0_\lambda)\in \widehat{\mathcal{F}}_{a,\lambda}^{\perp}\times\widehat{\mathcal{F}}_{b,\lambda}^{\perp}\times\bbr^+$ such that $(u_{\lambda,\beta}^0,v_{\lambda,\beta}^0)=(w^0_\lambda+t^0_\lambda(u_a-\widehat{u}_{a,\lambda}),\sigma^0_\lambda+t^0_\lambda(v_b-\widehat{v}_{b,\lambda}))
\in\mathcal{G}_{\lambda,\beta}$ for $\lambda\geq\Lambda_0$, where $\widehat{u}_{a,\lambda}$ and $\widehat{v}_{b,\lambda}$ are the projections of $u_a$ and $v_b$ in $\widehat{\mathcal{F}}_{a,\lambda}^{\perp}$ and $\widehat{\mathcal{F}}_{b,\lambda}^{\perp}$ respectively.  It follows from Remark~\ref{rmk0002} that $(w^0_\lambda,\sigma^0_\lambda,t^0_\lambda)\to(w_0,\sigma_0,t_0)$ strongly in $\h\times\h\times\bbr^+$ as $\lambda\to+\infty$.  Moreover, we also have that $(\widehat{u}_{a,\lambda}, \widehat{v}_{b,\lambda})\to (\widehat{u}_{a}, \widehat{v}_{b})$, where $\widehat{u}_{a}$ and $\widehat{v}_{b}$ are the projections of $u_a$ and $v_b$ in $\text{span}\{e^*_{a,j}\mid\alpha_{a,j}\leq1\})$ and $\text{span}\{e^*_{b,j}\mid\alpha_{b,j}\leq1\})$ respectively.  Due to the condition $(D_3)$, we must have that $(w_0,\sigma_0)\in H_0^1(\Omega_a)\times H_0^1(\Omega_b)$, which together with the condition $(D_3)$ and $u_a\in H_0^1(\Omega_a)$, implies%
\begin{eqnarray}
&&\limsup_{\lambda\to+\infty}J_{\lambda,\beta}(w^0_\lambda+t^0_\lambda(u_a-\widehat{u}_{a,\lambda}), \sigma^0_\lambda)\nonumber\\
&&\leq\limsup_{\lambda\to+\infty} I_{a,\lambda}(w^0_\lambda+t^0_\lambda(u_a-\widehat{u}_{a,\lambda}))\leq I_a(t_0(u_a-\widehat{u}_{a})).
\end{eqnarray}
By a similar argument as used in the proof of Lemma~\ref{lem0004} (see also \cite[Proposition~2.3]{SW09}),
 we have that $I_a(t_0(u_a-\widehat{u}_{a}))< I_a(u_a)$.
  Thus, $\frac14S^2> I_a(u_a)\geq\limsup_{\lambda\to+\infty}c_{\lambda,\beta}$
  for $0\leq\beta<1$, which completes the proof.%
\end{proof}


Next we prepare    some estimates which are  useful in the following sections.%
\begin{lemma}\label{lem0020}
Assume  $(D_1)$-$(D_3)$  and $\beta\leq0$.  If $\mathcal{D}_{\lambda}(u,v)$ is  positively definite in $E_\lambda$, then there exists $d_{\lambda,\beta}>0$ such that $\|u\|_{L^4(\bbr^4)}^4\|v\|_{L^4(\bbr^4)}^4-\beta^2\|u^2v^2\|^2_{L^1(\bbr^4)}>d_{\lambda,\beta}$ for all $(u,v)\in\mathcal{N}_{\lambda,\beta}$.%
\end{lemma}
\begin{proof}
Due to Lemmas~\ref{lem0110} and \ref{lem0002}, the conclusion can be obtained by a similar argument as used in the proof of \cite[Lemma~3.3]{WWZ15} and only some trivial modifications needed, so we omit the details here.%
\end{proof}

\begin{lemma}\label{lem5010}
Assume that  $(D_1)$-$(D_3)$ hold and $\mathcal{D}_{\lambda}(u,v)$ is positively definite in $E_\lambda$ in the case of $a_0<0\leq b_0$.  If%
\begin{equation}\label{eq5010}
0<\beta<1-\frac{1}{\alpha_{a,1}(\lambda)},%
\end{equation}
then we have%
\begin{equation}\label{eq5011}
\|u\|_{L^4(\bbr^4)}^2\geq\frac{(1-\frac{1}{\alpha_{a,1}(\lambda)})-\beta }{1-\beta(1-\frac{1}{\alpha_{a,1}(\lambda)})}S>0%
\end{equation}
and
\begin{equation}\label{eq5012}
\|v\|_{L^4(\bbr^4)}^2\geq\frac{(1-\frac{1}{\alpha_{a,1}(\lambda)})-\beta }{(1-\frac{1}{\alpha_{a,1}(\lambda)})-\beta}S>0%
\end{equation}
for all $(u,v)\in\mathcal{N}_{\lambda,\beta}$ with $J_{\lambda,\beta}(u,v)\leq\frac14 S^2$.%
\end{lemma}
\begin{proof}
Since $\mathcal{D}_{\lambda}(u,v)$ is positively definite in $E_\lambda$, without loss of generality, we may assume that $E_\lambda=(\widetilde{\mathcal{F}}_{a,\lambda}^{\perp}\oplus\mathcal{F}_{a,\lambda})
\times(\widetilde{\mathcal{F}}_{b,\lambda}^{\perp}\oplus\mathcal{F}_{b,\lambda})$.  Suppose $(u,v)\in\mathcal{N}_{\lambda,\beta}$ with $J_{\lambda,\beta}(u,v)\leq\frac12 S^2$, then by Lemma~\ref{lem0110}, \eqref{eq0003} and the H\"older inequality, we can see that%
\begin{eqnarray}
&(1-\frac{1}{\alpha_{a,1}(\lambda)})S\leq\|u\|_{L^4(\bbr^4)}^2+\beta\|v\|_{L^4(\bbr^4)}^2,\label{eq5003}\\%
&S\leq\|v\|_{L^4(\bbr^4)}^2+\beta\|u\|_{L^4(\bbr^4)}^2.\label{eq5004}%
\end{eqnarray}
Since $J_{\lambda,\beta}(u,v)\leq\frac12 S^2$, we also have from Lemma~\ref{lem0110} and \eqref{eq0003} that%
\begin{equation}\label{eq5005}
(1-\frac{1}{\alpha_{a,1}(\lambda)})\|u\|_{L^4(\bbr^4)}^2+\|v\|_{L^4(\bbr^4)}^2\leq S.%
\end{equation}
We can obtain \eqref{eq5011} by \eqref{eq5003} and \eqref{eq5005} while \eqref{eq5012} can be obtained by \eqref{eq5004} and \eqref{eq5005} due to \eqref{eq5010}, which completes the proof.%
\end{proof}

\begin{lemma}\label{lem6010}
Assume  that $(D_1)$-$(D_3)$ hold and  that $\mathcal{D}_{\lambda}(u,v)$ is positively definite in $E_\lambda$ in the case of $a_0\leq b_0<0$.  If%
\begin{equation}\label{eq6010}
0<\beta<\min\bigg\{\frac12(1-\frac{1}{\alpha_{a,1}(\lambda)})(1-\frac{1}{\alpha_{b,1}(\lambda)}),
\frac{1-\frac{1}{\alpha_{b,1}(\lambda)}}{1-\frac{1}{\alpha_{a,1}(\lambda)}},
\frac{1-\frac{1}{\alpha_{a,1}(\lambda)}}{1-\frac{1}{\alpha_{b,1}(\lambda)}}\bigg\},%
\end{equation}
then we have%
\begin{equation}\label{eq6011}
\|u\|_{L^4(\bbr^4)}^2\geq\frac{(1-\frac{1}{\alpha_{a,1}(\lambda)})(1-\frac{1}{\alpha_{b,1}(\lambda)})-2\beta }{(1-\frac{1}{\alpha_{b,1}(\lambda)})-\beta(1-\frac{1}{\alpha_{a,1}(\lambda)})}S>0%
\end{equation}
and
\begin{equation}\label{eq6012}
\|v\|_{L^4(\bbr^4)}^2\geq\frac{(1-\frac{1}{\alpha_{b,1}(\lambda)})(1-\frac{1}{\alpha_{a,1}(\lambda)})-2\beta }{(1-\frac{1}{\alpha_{a,1}(\lambda)})-\beta(1-\frac{1}{\alpha_{b,1}(\lambda)})}S>0%
\end{equation}
for all $(u,v)\in\mathcal{N}_{\lambda,\beta}$ with $J_{\lambda,\beta}(u,v)\leq\frac12 S^2$.%
\end{lemma}
\begin{proof}
Since $\mathcal{D}_{\lambda}(u,v)$ is positively definite in $E_\lambda$, without loss of generality,
we may assume that $E_\lambda=(\widetilde{\mathcal{F}}_{a,\lambda}^{\perp}\oplus\mathcal{F}_{a,\lambda})
\times(\widetilde{\mathcal{F}}_{b,\lambda}^{\perp}\oplus\mathcal{F}_{b,\lambda})$.
Suppose $(u,v)\in\mathcal{N}_{\lambda,\beta}$ with $J_{\lambda,\beta}(u,v)\leq\frac12 S^2$, then by Lemma~\ref{lem0110}, \eqref{eq0003} and the H\"older inequality, we can see that%
\begin{eqnarray}
&(1-\frac{1}{\alpha_{a,1}(\lambda)})S\leq\|u\|_{L^4(\bbr^4)}^2+\beta\|v\|_{L^4(\bbr^4)}^2,\label{eq6003}\\%
&(1-\frac{1}{\alpha_{b,1}(\lambda)})S\leq\|v\|_{L^4(\bbr^4)}^2+\beta\|u\|_{L^4(\bbr^4)}^2.\label{eq6004}%
\end{eqnarray}
Since $J_{\lambda,\beta}(u,v)\leq\frac12 S^2$, we also have from Lemma~\ref{lem0110} and \eqref{eq0003} that%
\begin{equation}\label{eq6005}
(1-\frac{1}{\alpha_{a,1}(\lambda)})\|u\|_{L^4(\bbr^4)}^2+(1-\frac{1}{\alpha_{b,1}(\lambda)})\|v\|_{L^4(\bbr^4)}^2\leq2S.%
\end{equation}
We can obtain \eqref{eq6011} by \eqref{eq6003} and \eqref{eq6005} while \eqref{eq6012} can be obtained by \eqref{eq6004} and \eqref{eq6005} due to \eqref{eq6010}, which completes the proof.%
\end{proof}

\section{The existence results}
Note that we have assumed  $b_0\geq a_0$, without loss of generality, one of the following four cases must happen:%
\begin{enumerate}
\item[$(i)$] $b_0\geq a_0\geq0$;%
\item[$(ii)$] $-\mu_{a,1}<a_0<0\leq b_0$;%
\item[$(iii)$] $-\mu_{a,1}<a_0<0$, $-\mu_{b,1}<b_0<0$ and $b_0\geq a_0$;%
\item[$(iv)$] $a_0\leq-\mu_{a,1}$ or $b_0\leq-\mu_{b,1}$ and $a_0\leq b_0$.%
\end{enumerate}
Let us first consider the case of $b_0\geq a_0\geq0$.  In this case, $\mathcal{D}_\lambda(u,v)=\|(u,v)\|_\lambda$ for all $\lambda>0$ due to Lemma~\ref{lem0110}.  Let%
\begin{equation}\label{eq0040}
\mathcal{E}_\beta(u,v):=\frac{1}{2}\|\nabla u\|_{L^2(\bbr^4)}^2+\frac12\|\nabla v\|_{L^2(\bbr^4)}^2-\frac{1}4\mathcal{L}_\beta(u,v).%
\end{equation}
Then $\mathcal{E}_\beta(u,v)$ is a $C^2$ functional on $D^{1,2}(\bbr^4)\times D^{1,2}(\bbr^4)$.  Denote $D^{1,2}(\bbr^4)\times D^{1,2}(\bbr^4)$ by $\mathcal{D}$ and define%
\begin{equation*}
\mathcal{M}^*_\beta:=\{(u,v)\in \mathcal{D}\backslash\{(0,0)\}\mid \langle D[\mathcal{E}_\beta(u,v)],(u,v)\rangle_{\mathcal{D}^*,\mathcal{D}}=0\}%
\end{equation*}
and
\begin{eqnarray}
&\mathcal{N}^*_\beta:=\Big\{(u,v)\in \mathcal{D}\mid u\not=0,v\not=0, \langle D[\mathcal{E}_\beta(u,v)],(u,0)\rangle_{\mathcal{D}^*,\mathcal{D}}\nonumber\\
&\quad\quad\quad\quad\quad\quad\quad \quad =\langle D[\mathcal{E}(u,v)],(0,v)\rangle_{\mathcal{D}^*,\mathcal{D}}=0\Big\},%
\end{eqnarray}
where $D[\mathcal{E}_\beta(u,v)]$ is the Frech\'et derivative of the functional $\mathcal{E}_\beta$ in $\mathcal{D}$ at $(u,v)$ and $\mathcal{D}^*$ is the dual space of $\mathcal{D}$.%
\begin{lemma}\label{lem0010}
Let $\beta>1$.  Then $m_{\beta}^0=m_{\beta}^{**}=\frac{1}{2(1+\beta)}S^2$, where
$$m_{\beta}^0=\inf_{\mathcal{M}^*_\beta}\mathcal{E}_\beta(u,v), \quad m_{\beta}^{**}=\inf_{\mathcal{N}^*_\beta}\mathcal{E}_\beta(u,v).$$
\end{lemma}
\begin{proof}
The idea of this proof comes from \cite{CZ14}.
Clearly, $\mathcal{N}^*_\beta\subset\mathcal{M}^*_\beta$.  Thus, it is easy to see that $m_{\beta}^0\leq m_{\beta}^{**}$.  Furthermore, by \cite[Theorem~3.1]{CZ14} (see also \cite[Theorem~1.5]{CZ121}), we know that%
\begin{equation}    \label{eq0012}
m_{\beta}^{**}=\frac{1}{2(1+\beta)}S^2\quad\text{for }\beta>1.%
\end{equation}
Thus, $m_{\beta}^0\leq m_{\beta}^{**}=\frac{1}{2(1+\beta)}S^2$ for $\beta>1$.
On the other hand, by a standard argument, we also have%
\begin{eqnarray}
m_{\beta}^0&=&\inf_{(u,v)\in \mathcal{D}\backslash\{(0,0)\}}\max_{t\geq0}\mathcal{E}_{\beta}(tu,tv)\notag\\%
&=&\inf_{(u,v)\in  \mathcal{D}\backslash\{(0,0)\}}\frac{(\|\nabla u\|_{L^2(\bbr^4)}^2+\|\nabla v\|_{L^2(\bbr^4)}^2)^2}
{4\mathcal{L}_\beta(u,v)}.\label{eq0039}%
\end{eqnarray}
It follows from the H\"older and Sobolev inequalities that $m_{\beta}^0\geq\frac{S^2}{4(1+\beta)}$.  Next, we will show that $m_{\beta}^0\geq m_{\beta}^{**}$.  Let $\{(u_n,v_n)\}\subset\mathcal{M}^*_\beta$ be a minimizing sequence of $\mathcal{E}_\beta(u,v)$.  Then it is easy to show that $\{(u_n,v_n)\}$ is bounded in $\mathcal{D}$.  Without loss of generality, we assume $(u_n,v_n)\rightharpoonup(u_0,v_0)$ weakly in $\mathcal{D}$ and $(u_n,v_n)\to(u_0,v_0)$ a.e. in $\bbr^4\times \bbr^4$ as $n\to\infty$.  Denote $w_n=u_n-u_0$ and $\sigma_n=v_n-v_0$.  Then by the Sobolev inequality, the Brez\'is-Lieb lemma and \cite[Lemma~2.3]{CZ14}, we have%
\begin{equation}\label{eq0037}
m_{\beta}^0=\mathcal{E}_\beta(u_0,v_0)+\mathcal{E}_\beta(w_n,\sigma_n)+o_n(1)%
\end{equation}
and%
\begin{eqnarray}
0=\langle D[\mathcal{E}_\beta(u_0,v_0)],(u_0,v_0)\rangle_{\mathcal{D}^*,\mathcal{D}}
+\langle D[\mathcal{E}_\beta(w_n,\sigma_n)],(w_n,\sigma_n)\rangle_{\mathcal{D}^*,\mathcal{D}}+o_n(1).\label{eq0038}%
\end{eqnarray}

\noindent {\bf Case~1:}\quad $(u_0,v_0)\not=(0,0)$. In this case, we can see from \eqref{eq0039} that%
\begin{equation*}
\frac{(\|\nabla u_0\|_{L^2(\bbr^4)}^2+\|\nabla v_0\|_{L^2(\bbr^4)}^2)^2}{\mathcal{L}_\beta(u_0,v_0)}\geq4m_{\beta}^0=\|\nabla u_n\|_{L^2(\bbr^4)}^2+\|\nabla v_n\|_{L^2(\bbr^4)}^2+o_n(1).%
\end{equation*}
It follows that%
\begin{equation*}
\|\nabla u_0\|_{L^2(\bbr^4)}^2+\|\nabla v_0\|_{L^2(\bbr^4)}^2\geq\|u_0\|_{L^4(\bbr^4)}^4+2\beta\|u_0^2v_0^2\|_{L^1(\bbr^4)}+\|v_0\|_{L^4(\bbr^4)}^4,%
\end{equation*}
which together with \eqref{eq0040} and \eqref{eq0038}, implies that%
\begin{equation}    \label{eq0043}
\mathcal{E}_\beta(u_0,v_0)>0\text{ and }\langle D[\mathcal{E}_\beta(w_n,\sigma_n)],(w_n,\sigma_n)\rangle_{\mathcal{D}^*,\mathcal{D}}\leq o_n(1).%
\end{equation}
If $\|\nabla w_n\|_{L^2(\bbr^4)}+\|\nabla \sigma_n\|_{L^2(\bbr^4)}\geq C+o_n(1)$, then by \eqref{eq0043}, there exists $0<t_n\leq1+o_n(1)$ such that $(t_nw_n,t_nv_n)\in\mathcal{M}^*_\beta$ for $n$ large enough.  Since $\beta>1$, by Lemma~\ref{lem0008} and similar arguments as used in \eqref{eq0037}, we can see that%
\begin{eqnarray*}
m_{\beta}^0&\geq&\mathcal{E}_\beta(t_nu_n,t_nv_n)\\%
&=&\mathcal{E}_\beta(t_nu_0,t_nv_0)+\mathcal{E}_\beta(t_nw_n,t_n\sigma_n)+o_n(1)\\%
&\geq&\mathcal{E}_\beta(t_nu_0,t_nv_0)+m_{\beta}^0+o_n(1).%
\end{eqnarray*}
It follows that $t_n\to0$ as $n\to\infty$, which is impossible due to $\|\nabla w_n\|_{L^2(\bbr^4)}+\|\nabla \sigma_n\|_{L^2(\bbr^4)}\geq C+o_n(1)$ and $(t_nw_n,t_nv_n)\in\mathcal{M}^*_\beta$ for $n$ large enough.  Therefore, we must have $\|\nabla w_n\|_{L^2(\bbr^4)}+\|\nabla \sigma_n\|_{L^2(\bbr^4)}\to0$ as $n\to\infty$ up to a subsequence.  It follows from the Sobolev inequality and \eqref{eq0037}--\eqref{eq0043} that $\mathcal{E}_\beta(u_0,v_0)=m_\beta^0$ and $(u_0,v_0)\in\mathcal{M}_\beta^*$.  If $u_0=0$ or $v_0=0$, then by the Sobolev inequality and \eqref{eq0039}, we can see that $m_\beta^0\geq\frac14S^2$.  It contradicts to $m_{\beta}^0\leq m_{\beta}^{**}$ and \eqref{eq0012}, since $\beta>1$.  Hence, both $u_0\not=0$ and $v_0\not=0$.  Since $\mathcal{E}_\beta(u,v)$ is $C^2$, by a similar argument as used in the proof of Lemma~\ref{lem5001}, we have $D[\mathcal{E}_\beta(u_0,v_0)]=0$ in $\mathcal{D}^*$.  Hence, $(u_0,v_0)\in\mathcal{N}_\beta^*$ and $m_\beta^0\geq m_\beta^{**}$.%

\vskip0.23in

\noindent {\bf Case~2:}\quad $(u_0,v_0)=(0,0)$.%

In this case, $(w_n,\sigma_n)=(u_n,v_n)$.  By \eqref{eq0037} and $m_{\beta}^0\geq\frac{S^2}{4(1+\beta)}$,
 we must have $\|\nabla w_n\|_{L^2(\bbr^4)}+\|\nabla \sigma_n\|_{L^2(\bbr^4)}\geq C+o_n(1)$.
 If $w_n\to0$ or $\sigma_n\to0$ strongly in $D^{1,2}(\bbr^4)$ as $n\to\infty$,
 then we can see from the Sobolev inequality and \eqref{eq0039} that $m_\beta^0\geq\frac14S^2$,
 which is impossible since $m_{\beta}^0\leq m_{\beta}^{**}$, \eqref{eq0012} holds and $\beta>1$.
  Therefore, we must  have both $w_n\not\to0$ and $\sigma_n\not\to0$ strongly in $D^{1,2}(\bbr^4)$
  as $n\to\infty$.  Now, by a similar argument as used in \cite[Lemma~2.5]{CZ14}, we can get a contradiction.%
\end{proof}

Due to Lemma~\ref{lem0010}, we can give a precise description on $m_{\lambda,\beta}$ and $m_{\lambda,\beta}^*$ in the case of $b_0\geq a_0\geq0$.%
\begin{lemma}\label{lem0100}
Assume that $(D_1)$-$(D_3)$ hold.  If $b_0\geq a_0\geq0$, then
$$m_{\lambda,\beta}=\frac{1}{2(1+\max\{\beta, 0\})}S^2, \quad m_{\lambda,\beta}^*=\frac{1}{2(1+\max\{1,\beta\})}S^2$$
 for all $\beta\in\bbr$ and $\lambda>0$.%
\end{lemma}
\begin{proof}
For the sake of clarity, the proof will be performed through the following five steps.%

\vskip0.11in
\noindent {\bf Step 1.} We prove that $m_{\lambda,\beta}^*=\frac{1}{4}S^2$ and $m_{\lambda,\beta}=\frac12S^2$ for $\lambda>0$ and $\beta\leq0$.%

Indeed, thanks to the Sobolev inequality and the condition $(D_1)$, we have $m_{a,\lambda}\geq\frac{1}{4}S^2$ and $m_{b,\lambda}\geq\frac1{4}S^2$ in the case of $b_0\geq a_0\geq0$.  It follows from Lemmas~\ref{lem5003} and \ref{lem5002} that $m_{\lambda,\beta}^*=\frac{1}{4}S^2$ and $m_{\lambda,\beta}=\frac12S^2$ for $\lambda>0$ and $\beta\leq0$.%

\vskip0.11in
\noindent{\bf Step 2.} We prove that $m_{\lambda,\beta}^*=\frac{1}{4}S^2$ for $\lambda>0$ and $0<\beta\leq1$.  Indeed, by Lemma~\ref{lem5003}, we can see that $m_{\lambda,\beta}^*\leq\frac{1}{4}S^2$ for $\lambda>0$ and $0<\beta\leq1$.  It remains to show that $m_{\lambda,\beta}^*\geq\frac{1}{4}S^2$ for all $\lambda>0$ and $0<\beta\leq1$.  Suppose the contrary, we have $m_{\lambda',\beta'}^*<\frac{1}{4}S^2$ for some $\lambda'>0$ and $0<\beta'\leq1$.  By the definition of $m_{\lambda',\beta'}^*$, there exists $(u_\delta,v_\delta)\in\mathcal{M}_{\lambda',\beta'}$ satisfying $J_{\lambda',\beta'}(u_\delta,v_\delta)\leq m_{\lambda',\beta'}+\delta$ for some $\delta\in(0, \frac{1}{4}S^2-m_{\lambda',\beta'})$.  Since $(u_\delta,v_\delta)\in\mathcal{M}_{\lambda',\beta'}$ with $\lambda'>0$, by the Sobolev inequality, the condition $(D_1)$ and $b_0\geq a_0\geq0$, we have%
\begin{equation}\label{eq0135}
\mathcal{L}_{\beta'}(u_\delta,v_\delta)\leq4m_{\lambda',\beta'}+4\delta
<S^2,%
\end{equation}
and
\begin{equation}\label{eq0133}
S\|u_\delta\|_{L^4(\bbr^4)}^2+S\|v_\delta\|_{L^4(\bbr^4)}^2\leq\mathcal{D}_{\lambda'}(u_\delta,v_\delta)
=\mathcal{L}_{\beta'}(u_\delta,v_\delta).%
\end{equation}
Combining \eqref{eq0135}-\eqref{eq0133}, we can obtain that $\|u_\delta\|_{L^4(\bbr^4)}^2+\|v_\delta\|_{L^4(\bbr^4)}^2<S$.  On the other hand, thanks to the H\"older inequality, $0<\beta'\leq1$ and \eqref{eq0133}, we can see that $\|u_\delta\|_{L^4(\bbr^4)}^2+\|v_\delta\|_{L^4(\bbr^4)}^2\geq S$, which is a contradiction.%

\vskip0.11in
\noindent{\bf Step 3.} We prove that $m_{\lambda,\beta}=\frac{1}{2(1+\beta)}S^2$ for $\lambda>0$ and $\beta\in(0, 1)$.%

Indeed, consider the following family of functions:%
\begin{equation*}
\psi_\ve^*(x)=\frac{2\sqrt{2}\ve}{\ve^2+|x|^2},\quad\ve>0.%
\end{equation*}
Then $\psi_\ve(x)=\psi_\ve^*(x)\eta(x)\in H_0^1(B_R)$, where $\eta\in C_0^\infty(B_R)$.  Furthermore, it is well known that $\|\psi_\ve\|_{L^4(\bbr^4)}^4=S^2+O(\ve^4)$, $\|\nabla \psi_\ve\|_{L^2(\bbr^4)}^2=S^2+O(\ve^2)$ and $\|\psi_\ve\|_{L^2(\bbr^4)}^2=o(\ve)$ (cf. \cite{S96}).  It follows from the condition $(D_1)$ that%
\begin{eqnarray}\label{eq0009}
\|\psi_\ve\|_{L^4(\bbr^4)}^4\mathcal{D}_{a,\lambda}(\psi_\ve,\psi_\ve)
-\beta\|(\psi_\ve)^4\|_{L^1(\bbr^4)}\mathcal{D}_{b,\lambda}(\psi_\ve,\psi_\ve)=S^4(1-\beta+o(\ve))%
\end{eqnarray}
and%
\begin{eqnarray}\label{eq0010}
\|\psi_\ve\|_{L^4(\bbr^4)}^4\mathcal{D}_{b,\lambda}(\psi_\ve,\psi_\ve)
-\beta\|(\psi_\ve)^4\|_{L^1(\bbr^4)}\mathcal{D}_{a,\lambda}(\psi_\ve,\psi_\ve)=S^4(1-\beta+o(\ve)).%
\end{eqnarray}
Since%
\begin{equation}   \label{eq0030}
\|\psi_\ve\|_{L^4(\bbr^4)}^8-\beta^2\|(\psi_\ve)^4\|_{L^1(\bbr^4)}^2
=(1-\beta^2)\|\psi_\ve\|^8_{L^4(\bbr^4)},%
\end{equation}
by \eqref{eq0009}--\eqref{eq0010}, we can see that the proof of Lemma~\ref{lem0002} still works for $\ve$ sufficiently small in the case of $\beta\not=1$.
Thus, there exist $t_{\lambda,\beta}(\psi_\ve,\psi_\ve)$ and $s_{\lambda,\beta}(\psi_\ve,\psi_\ve)$ respectively given by \eqref{eq0005} and \eqref{eq0006} such that%
\begin{equation*}
(t_{\lambda,\beta}(\psi_\ve, \psi_\ve)\psi_\ve, s_{\lambda,\beta}(\psi_\ve, \psi_\ve)\psi_\ve)\in\mathcal{N}_{\lambda,\beta},%
\end{equation*}
which then implies%
\begin{equation*}
m_{\lambda,\beta}\leq J_{\lambda,\beta}(t_{\lambda,\beta}(\psi_\ve, \psi_\ve)\psi_\ve, s_{\lambda,\beta}(\psi_\ve, \psi_\ve)\psi_\ve)
=\frac{1}{2(1+\beta)}S^2+o(\ve).%
\end{equation*}
It follows that $m_{\lambda,\beta}\leq\frac{1}{2(1+\beta)}S^2$.  It remains to show that $m_{\lambda,\beta}\geq\frac{1}{2(1+\beta)}S^2$.  Indeed, let $\{(u_n,v_n)\}\subset\mathcal{N}_{\lambda,\beta}$ be a minimizing sequence of $J_{\lambda,\beta}(u, v)$.  Since $m_{\lambda,\beta}\leq\frac{1}{2(1+\beta)}S^2$, we have%
\begin{equation}\label{eq0017}
\mathcal{L}_{\beta}(u_n,v_n)\leq4m_{\lambda,\beta}+o_n(1)
\leq\frac{2}{1+\beta}S^2+o_n(1).%
\end{equation}
Note that the condition $(D_1)$ holds, $\lambda>0$ and $b_0\geq a_0\geq0$, by the Sobolev inequality and the fact that $(u_n,v_n)\in\mathcal{N}_{\lambda,\beta}$, we have%
\begin{equation}\label{eq0018}
S\|u_n\|_{L^4(\bbr^4)}^2\leq\mathcal{D}_{a,\lambda}(u_n,u_n)
=\|u_n\|_{L^4(\bbr^4)}^4+\beta\|u_n^2v_n^2\|_{L^1(\bbr^4)}%
\end{equation}
and%
\begin{equation}\label{eq0019}
S\|v_n\|_{L^4(\bbr^4)}^2\leq\mathcal{D}_{b,\lambda}(v_n,v_n)
=\|v_n\|_{L^4(\bbr^4)}^4+\beta\|u_n^2v_n^2\|_{L^1(\bbr^4)}.%
\end{equation}
Thanks to \eqref{eq0017}--\eqref{eq0019}, we can see that%
\begin{equation*}
\|u_n\|_{L^4(\bbr^4)}^2+\|v_n\|_{L^4(\bbr^4)}^2\leq\frac{2}{1+\beta}S+o_n(1),%
\end{equation*}
which together with \eqref{eq0018} and \eqref{eq0019} and the H\"older inequality, implies%
\begin{equation}\label{eq0021}
\|u_n\|_{L^4(\bbr^4)}^2\geq\frac{1}{1+\beta}S+o_n(1)\quad\text{and}\quad
\|v_n\|_{L^4(\bbr^4)}^2\geq\frac{1}{1+\beta}S+o_n(1).%
\end{equation}
Since $(u_n,v_n)\in\mathcal{N}_{\lambda,\beta}$ and $\beta>0$, we must have from \eqref{eq0018}--\eqref{eq0021} that
$m_{\lambda,\beta}\geq\frac{1}{2(1+\beta)}S^2+o_n(1)$.  The conclusion follows from letting $n\to\infty$.%

\vskip0.11in
\noindent{\bf Step 4.} We prove that $m_{\lambda,1}=\frac14S^2$ for $\lambda>0$.%

Indeed, for every $\lambda>0$, we consider the following two-component systems of algebraic equations%
\begin{equation}\label{eq0031}
\left\{\aligned \mathcal{D}_{a,\lambda}(\psi_\ve,\psi_\ve)-\|\psi_\ve\|_{L^4(\bbr^4)}^4t-\|\psi_\ve^4\|_{L^1(\bbr^4)}s&=0,\\%
\mathcal{D}_{b,\lambda}(\psi_\ve,\psi_\ve)-\|\psi_\ve\|_{L^4(\bbr^4)}^4s-\|\psi_\ve^4\|_{L^1(\bbr^4)}t&=0,\endaligned\right.%
\end{equation}
where $\psi_\ve$ is given in Step 3.
Since $\|\psi_\ve\|_{L^4(\bbr^4)}^4=S^2+O(\ve^4)$, $\|\nabla \psi_\ve\|_{L^2(\bbr^4)}^2=S^2+O(\ve^2)$ and $\|\psi_\ve\|_{L^2(\bbr^4)}^2=o(\ve)$, by the condition $(D_1)$, we can see that \eqref{eq0031} can be solved in $\bbr^+\times\bbr^+$ for $\ve$ sufficiently small and the solutions $(t_\ve,s_\ve)$ satisfies $t_\ve+s_\ve=1+o(\ve)$.  Thus, we can choose $t_\ve>0$ and $s_\ve>0$ for $\ve$ sufficiently small such that $(\sqrt{t_\ve}\psi_\ve,\sqrt{s_\ve}\psi_\ve)\in\mathcal{N}_{\lambda, 1}$.  It follows that%
\begin{equation}   \label{eq0032}
m_{\lambda, 1}\leq J_{\lambda,1}(\sqrt{t_\ve}\psi_\ve,\sqrt{s_\ve}\psi_\ve)=\frac14 S^2+o(\ve).%
\end{equation}
Letting $\ve\to0^+$ in \eqref{eq0032}, we have $m_{\lambda,1}\leq\frac14 S^2$ for all $\lambda>0$.
 Since $m_{\lambda,1}\geq m_{\lambda,1}^*$ for all $\lambda>0$, by the conclusion of Step 2,
 we can see that $m_{\lambda,1}\geq\frac14 S^2$ for all $\lambda>0$.%

\vskip0.11in
\noindent{\bf Step 5. }We prove that $m_{\lambda,\beta}^*=m_{\lambda,\beta}=\frac{1}{2(1+\beta)}S^2$ for $\lambda>0$ and $\beta>1$.%

Indeed, since $\mathcal{N}_{\lambda,\beta}\subset\mathcal{M}_{\lambda,\beta}$, we can see that $m_{\lambda,\beta}^*\leq m_{\lambda,\beta}$.
 Note that \eqref{eq0009}--\eqref{eq0030} still hold for $\lambda>0$ and $\beta>1$, thus, we also
 have $m_{\lambda,\beta}\leq\frac{1}{2(1+\beta)}S^2$ for $\lambda>0$ and $\beta>1$ by similar arguments as used in Step 3.
 In what follows, we will show that $m_{\lambda,\beta}^*\geq\frac{1}{2(1+\beta)}S^2$ for $\lambda>0$ and $\beta>1$.
 Indeed, for every $\delta>0$, we can take $(u_\delta,v_\delta)\in\mathcal{M}_{\lambda,\beta}$ such
 that $J_{\lambda,\beta}(u_\delta,v_\delta)\leq m_{\lambda,\beta}^*+\delta$.  By a standard argument,
 there exists $t_\delta>0$ such that $(t_\delta u_\delta,t_\delta v_\delta)\in\mathcal{M}_\beta^*$.
 It follows from the condition $(D_1)$, $\lambda>0$, $b_0\geq a_0\geq0$ and Lemmas~\ref{lem0008} and \ref{lem0010} that%
\begin{equation*}
\delta+m_{\lambda,\beta}^*\geq J_{\lambda,\beta}(u_\delta,v_\delta)\geq J_{\lambda,\beta}(t_\delta u_\delta,t_\delta v_\delta)\geq\mathcal{E}_\beta(t_\delta u_\delta,t_\delta v_\delta)\geq m_\beta^0=\frac{S^2}{2(1+\beta)}.%
\end{equation*}
The conclusion follows by letting $\delta\to0^+$.%
\end{proof}

With Lemma~\ref{lem0100} in hands, we can obtain the following
\begin{proposition}\label{prop0004}
Let  $(D_1)$-$(D_3)$ hold.  If $b_0\geq a_0\geq0$, then $m_{\lambda,\beta}$ and $m_{\lambda,\beta}^*$ can not be attained for all $\beta\in\bbr$ and $\lambda>0$.%
\end{proposition}
\begin{proof}
For the sake of clarity, the proof will be performed through the following four steps.%

\vskip0.1in

\noindent{\bf Step 1.}  We prove that $m_{\lambda,\beta}$ and $m_{\lambda,\beta}^*$ can not be attained for $\lambda>0$ and $\beta\leq0$.%

Firstly, we assume that there exists $(u_{\lambda,\beta}, v_{\lambda,\beta})\in\mathcal{N}_{\lambda,\beta}$
such that $$J_{\lambda,\beta}(u_{\lambda,\beta}, v_{\lambda,\beta})=m_{\lambda,\beta} \hbox{ for } \lambda>0 \hbox{  and } \beta\leq0.$$
  Then by Lemma~\ref{lem0100}, we must have%
\begin{equation}   \label{eq0008}
\mathcal{D}_{\lambda}(u_{\lambda,\beta},v_{\lambda,\beta})
=2S^2.%
\end{equation}
On the other hand, since $(u_{\lambda,\beta}, v_{\lambda,\beta})\in\mathcal{N}_{\lambda,\beta}$ with $\lambda>0$ and $\beta\leq0$, by the Sobolev inequality, the condition $(D_1)$ and $b_0\geq a_0\geq0$, we can see that%
\begin{equation}   \label{eq0044}
\|\nabla u_{\lambda,\beta}\|_{L^2(\bbr^4)}^2\geq S^2\text{ and }\|\nabla v_{\lambda,\beta}\|_{L^2(\bbr^4)}^2\geq S^2%
\end{equation}
and
\begin{equation}   \label{eq0045}
\|u_{\lambda,\beta}\|_{L^4(\bbr^4)}^2\geq S\text{ and }\|v_{\lambda,\beta}\|_{L^4(\bbr^4)}^2\geq S.%
\end{equation}
 By \eqref{eq0044},  \eqref{eq0008},   $(D_1)$ and  recall that $b_0\geq a_0\geq0$,  we have that
\begin{equation*}
\|\nabla u_{\lambda,\beta}\|_{L^2(\bbr^4)}^2=\|\nabla v_{\lambda,\beta}\|_{L^2(\bbr^4)}^2=S^2%
\end{equation*}
and
\begin{equation*}
\int_{\bbr^4}a(x)u_{\lambda,\beta}^2dx=\int_{\bbr^4}b(x)v_{\lambda,\beta}^2dx=0.%
\end{equation*}
Thanks to the condition $(D_3)$, $u_{\lambda,\beta}\in H_0^1(\Omega_a)$ and $v_{\lambda,\beta}\in H_0^1(\Omega_b)$ and it follows from \eqref{eq0045} that%
\begin{equation*}
\frac{\|\nabla u_{\lambda,\beta}\|_{L^2(\bbr^4)}^2}{\|u_{\lambda,\beta}\|_{L^4(\bbr^4)}^2}
=\inf_{u\in H_0^1(\Omega_a)\backslash\{0\}}\frac{\|\nabla u\|_{L^2(\bbr^4)}^2}{\|u\|_{L^4(\bbr^4)}^2},
\end{equation*}
\begin{equation*}
\frac{\|\nabla v_{\lambda,\beta}\|_{L^2(\bbr^4)}^2}{\|v_{\lambda,\beta}\|_{L^4(\bbr^4)}^2}
=\inf_{v\in H_0^1(\Omega_b)\backslash\{0\}}\frac{\|\nabla v\|_{L^2(\bbr^4)}^2}{\|v\|_{L^4(\bbr^4)}^2},%
\end{equation*}
which contradicts to the  Talenti's  results in \cite{T76}, since $\Omega_a$ and $\Omega_b$ are both bounded domains.
 Thus, $m_{\lambda,\beta}$ can not be attained for $\lambda>0$ and $\beta\leq0$.

  We next prove that $m_{\lambda,\beta}^*$ can not be attained for $\lambda>0$ and $\beta\leq0$.
   In fact, assume that there exists $(u_{\lambda,\beta}, v_{\lambda,\beta})\in\mathcal{M}_{\lambda,\beta}$ such
   that $J_{\lambda,\beta}(u_{\lambda,\beta}, v_{\lambda,\beta})=m_{\lambda,\beta}^*$ for $\lambda>0$ and $\beta\leq0$.
    Then by Lemma~\ref{lem0100} again, we must have%
\begin{equation}   \label{eq0108}
\mathcal{D}_{\lambda}(u_{\lambda,\beta},v_{\lambda,\beta})
=S^2.%
\end{equation}
On the other hand, since $(u_{\lambda,\beta}, v_{\lambda,\beta})\in\mathcal{M}_{\lambda,\beta}$ with $\lambda>0$ and $\beta\leq0$, by the Sobolev inequality, the condition $(D_1)$ and $b_0\geq a_0\geq0$, we can see that%
\begin{equation}   \label{eq0046}
\|\nabla u_{\lambda,\beta}\|_{L^2(\bbr^4)}^2+\|\nabla v_{\lambda,\beta}\|_{L^2(\bbr^4)}^2\geq S^2%
\end{equation}
and
\begin{equation}   \label{eq0047}
\|u_{\lambda,\beta}\|_{L^4(\bbr^4)}^2+\|v_{\lambda,\beta}\|_{L^4(\bbr^4)}^2\geq S.%
\end{equation}
\eqref{eq0046} together with \eqref{eq0108}, the condition $(D_1)$ and $b_0\geq a_0\geq0$ once more, implies that%
\begin{equation*}
\|\nabla u_{\lambda,\beta}\|_{L^2(\bbr^4)}^2+\|\nabla v_{\lambda,\beta}\|_{L^2(\bbr^4)}^2=S^2%
\end{equation*}
and
\begin{equation*}
\int_{\bbr^4}a(x)u_{\lambda,\beta}^2dx=\int_{\bbr^4}b(x)v_{\lambda,\beta}^2dx=0.%
\end{equation*}
Thanks to the condition $(D_3)$, $u_{\lambda,\beta}\in H_0^1(\Omega_a)$ and $v_{\lambda,\beta}\in H_0^1(\Omega_b)$ and it follows from \eqref{eq0047} and the condition $(D_3)$ again that%
\begin{equation*}
\frac{\|\nabla w_{\lambda,\beta}\|_{L^2(\bbr^4)}^2}{\|w_{\lambda,\beta}\|_{L^4(\bbr^4)}^2}=\inf_{w\in H_0^1(\Omega_a\cup\Omega_b)\backslash\{0\}}\frac{\|\nabla w\|_{L^2(\bbr^4)}^2}{\|w\|_{L^4(\bbr^4)}^2},%
\end{equation*}
where $w_{\lambda,\beta}=u_{\lambda,\beta}+v_{\lambda,\beta}$.  It contradicts to the results in \cite{T76}, since $\Omega_a$ and $\Omega_b$ are both bounded domains.  Thus, $m_{\lambda,\beta}^*$ can not be attained for $\lambda>0$ and $\beta\leq0$.%

\vskip0.1in

\noindent{\bf Step 2.}\quad We prove that $m_{\lambda,\beta}$ and $m_{\lambda,\beta}^*$ can not be attained for $\lambda>0$ and $\beta\in(0, 1)$.%

We first prove that $m_{\lambda,\beta}$ can not be attained for $\lambda>0$ and $\beta\in(0, 1)$.
 Indeed, suppose that there exists $(u_{\lambda,\beta},v_{\lambda,\beta})\in\mathcal{N}_{\lambda,\beta}$
 such that $J_{\lambda,\beta}(u_{\lambda,\beta},v_{\lambda,\beta})=m_{\lambda,\beta}$ for
 $\lambda>0$ and $\beta\in(0, 1)$.  Then by similar arguments as used in Step 3 of Lemma~\ref{lem0100}, we can show that%
\begin{equation}   \label{eq0026}
\|u_{\lambda,\beta}\|_{L^4(\bbr^4)}^2=\frac{1}{1+\beta}S,\quad\|v_{\lambda,\beta}\|_{L^4(\bbr^4)}^2=\frac{1}{1+\beta}S%
\end{equation}
and%
\begin{equation*}   \label{eq0025}
\|u_{\lambda,\beta}^2v_{\lambda,\beta}^2\|_{L^1(\bbr^4)}=\frac{1-\beta}{1+\beta}S.%
\end{equation*}
It follows from $\beta\in(0, 1)$ that%
\begin{equation*}
\|u_{\lambda,\beta}\|_{L^4(\bbr^4)}^4\|v_{\lambda,\beta}\|_{L^4(\bbr^4)}^4-\beta^2\|u_{\lambda,\beta}^2v_{\lambda,\beta}^2\|_{L^1(\bbr^4)}^2
=\frac{(1-\beta)^3}{1+\beta}S^2>0.%
\end{equation*}
On the other hand, by Lemma~\ref{lem0100}, $\mathcal{D}_{\lambda}(u_{\lambda,\beta},v_{\lambda,\beta})
=\frac{2}{1+\beta}S^2$.  It follows from the condition $(D_1)$ and $b_0\geq a_0\geq0$ that%
\begin{equation*}
\lambda\int_{\bbr^4}a(x)u_{\lambda,\beta}^2dx\leq\frac{2}{1+\beta}S^2\quad\text{and}\quad\lambda\int_{\bbr^4}b(x)v_{\lambda,\beta}^2dx\leq\frac{2}{1+\beta}S^2.%
\end{equation*}
Hence, there exists $0<\lambda'<\lambda$ such that


$$\frac{(1-\beta)^3}{4(1+\beta)}S^2 \geq (\lambda-\lambda')(\|v_{\lambda,\beta}\|_{L^4(\bbr^4)}^4\int_{\bbr^4}a(x)u_{\lambda,\beta}^2dx$$
$$\quad\quad\quad\quad\quad\quad\quad\quad\quad +\beta\|u_{\lambda,\beta}^2v_{\lambda,\beta}^2\|_{L^1(\bbr^4)}\int_{\bbr^4}b(x)v_{\lambda,\beta}^2dx)>0$$
and
$$\frac{(1-\beta)^3}{4(1+\beta)}S^2 \geq (\lambda-\lambda')(\|u_{\lambda,\beta}\|_{L^4(\bbr^4)}^4\int_{\bbr^4}b(x)v_{\lambda,\beta}^2dx$$
$$\quad\quad\quad\quad\quad\quad\quad\quad\quad +\beta\|u_{\lambda,\beta}^2v_{\lambda,\beta}^2\|_{L^1(\bbr^4)}\int_{\bbr^4}a(x)u_{\lambda,\beta}^2dx)>0,
$$
which implies $$(t_{\lambda',\beta}(u_{\lambda,\beta}, v_{\lambda,\beta}), s_{\lambda',\beta}(u_{\lambda,\beta}, v_{\lambda,\beta}))\in[\frac{\sqrt{3}}2, 1]\times[\frac{\sqrt{3}}2, 1],$$
 where $t_{\lambda',\beta}(u_{\lambda,\beta}, v_{\lambda,\beta})$ and $s_{\lambda',\beta}(u_{\lambda,\beta}, v_{\lambda,\beta})$
 are given by \eqref{eq0005} and \eqref{eq0006}, respectively.  Therefore, by a similar argument as used in Lemma~\ref{lem0002},
 we must have that%
\begin{equation*}
(t_{\lambda',\beta}(u_{\lambda,\beta}, v_{\lambda,\beta})u_{\lambda,\beta}, s_{\lambda',\beta}(u_{\lambda,\beta}, v_{\lambda,\beta})v_{\lambda,\beta})\in\mathcal{N}_{\lambda',\beta}.%
\end{equation*}
Since $\beta\in(0, 1)$, by Lemmas~\ref{lem0006} and \ref{lem0100}, we have from $0<\lambda'<\lambda$ and the condition $(D_1)$ that%
\begin{equation*}
J_{\lambda',\beta}(t_{\lambda',\beta}(u_{\lambda,\beta}, v_{\lambda,\beta})u_{\lambda,\beta}, s_{\lambda',\beta}(u_{\lambda,\beta}, v_{\lambda,\beta})v_{\lambda,\beta})=m_{\lambda',\beta}.%
\end{equation*}
That is, $(t_{\lambda',\beta}(u_{\lambda,\beta}, v_{\lambda,\beta})u_{\lambda,\beta}, s_{\lambda',\beta}(u_{\lambda,\beta}, v_{\lambda,\beta})v_{\lambda,\beta})$ is the minimizer of $J_{\lambda',\beta}(u,v)$ on $\mathcal{N}_{\lambda',\beta}$.  By similar arguments as used in \eqref{eq0026}, we can also obtain that%
\begin{equation*}
\|t_{\lambda',\beta}(u_{\lambda,\beta}, v_{\lambda,\beta})u_{\lambda,\beta}\|_{L^4(\bbr^4)}^2=\frac{1}{1+\beta}S,\quad
\|s_{\lambda',\beta}(u_{\lambda,\beta}, v_{\lambda,\beta})v_{\lambda,\beta}\|_{L^4(\bbr^4)}^2=\frac{1}{1+\beta}S,%
\end{equation*}
which together with \eqref{eq0026}, implies $t_{\lambda',\beta}(u_{\lambda,\beta}, v_{\lambda,\beta})=1$ and
 $s_{\lambda',\beta}(u_{\lambda,\beta}, v_{\lambda,\beta})=1$.  Note that $(u_{\lambda,\beta}, v_{\lambda,\beta})$
 are the minimizers for  both   $J_{\lambda,\beta}(u,v)$ on $\mathcal{N}_{\lambda,\beta}$ and $J_{\lambda',\beta}(u,v)$
 on $\mathcal{N}_{\lambda',\beta}$. By $\lambda>\lambda'>0$, $\beta\in(0, 1)$, condition $(D_1)$
 and Lemma~\ref{lem0100}, we can see that%
\begin{equation*}
\int_{\bbr^4}a(x)u_{\lambda,\beta}^2dx=\int_{\bbr^4}b(x)v_{\lambda,\beta}^2dx=0.%
\end{equation*}
Thanks to the condition $(D_3)$ and $(u_{\lambda,\beta},v_{\lambda,\beta})\in\mathcal{N}_{\lambda,\beta}$,
we have $u_{\lambda,\beta}\in\mathcal{N}_a$ and $v_{\lambda,\beta}\in\mathcal{N}_b$.  Since $b_0\geq a_0\geq0$, we must have
$\|u_{\lambda,\beta}\|_{L^4(\bbr^4)}^2\geq S$ and $\|v_{\lambda,\beta}\|_{L^4(\bbr^4)}^2\geq S$, which
 contradicts to \eqref{eq0026}.  Since $m_{\lambda,\beta}^*=\frac14S^2$ for all $\lambda>0$ and $\beta\leq1$,
 we can prove that $m_{\lambda,\beta}^*$ can not be attained for $\lambda>0$ and $\beta\in(0, 1)$ by a similar
 argument as used in Step.~1.%

\vskip0.16in

\noindent{\bf Step 3.} We prove that $m_{\lambda,1}$ and $m_{\lambda,1}^*$ can not be attained for $\lambda>0$.%

We first prove that $m_{\lambda,1}$ can not be attained for $\lambda>0$.  Indeed,
suppose that there exists $(u_{\lambda,1},v_{\lambda,1})\in\mathcal{N}_{\lambda,1}$
such that $J_{\lambda,1}(u_{\lambda,1},v_{\lambda,1})=m_{\lambda,1}$.
 By Lemma~\ref{lem0100} and the H\"older inequality, we can see that
 $\|u_{\lambda,1}\|_{L^4(\bbr^4)}^2+\|v_{\lambda,1}\|_{L^4(\bbr^4)}^2\geq S$.
  On the other hand, thanks to a similar argument as used
   in \eqref{eq0135}-\eqref{eq0133}, we have $\|u_{\lambda,1}\|_{L^4(\bbr^4)}^2+\|v_{\lambda,1}\|_{L^4(\bbr^4)}^2\leq S$ due to
   Lemma~\ref{lem0100}.  Thus, we must
   have $\|u_{\lambda,1}\|_{L^4(\bbr^4)}^2+\|v_{\lambda,1}\|_{L^4(\bbr^4)}^2=S$.  Since $(u_{\lambda,1},v_{\lambda,1})\in\mathcal{N}_{\lambda,1}$,
   by similar arguments as used in \eqref{eq0108}--\eqref{eq0047} and the H\"older inequality,
   we can see that $$\int_{\bbr^4}a(x)u_{\lambda,1}^2dx=\int_{\bbr^4}b(x)v_{\lambda,1}^2dx=0, \quad
    S\|u_{\lambda,1}\|_{L^4(\bbr^4)}^2=\|\nabla u_{\lambda,1}\|_{L^2(\bbr^4)}^2$$ and
   $$ S\|v_{\lambda,1}\|_{L^4(\bbr^4)}^2=\|\nabla v_{\lambda,1}\|_{L^2(\bbr^4)}^2.$$
    By the condition $(D_3)$, $u_{\lambda,\beta}\in H_0^1(\Omega_a)$ and
     $v_{\lambda,\beta}\in H_0^1(\Omega_b)$, which contradicts to the  Talenti's  results
     of \cite{T76}, since $\Omega_a$ and $\Omega_b$ are both bounded domains.
      Thus, $m_{\lambda,1}$ can not be attained for $\lambda>0$.
      Since $m_{\lambda,1}^*=\frac14S^2$ for all $\lambda>0$,
      we can prove that $m_{\lambda,1}^*$ can not be attained
      for $\lambda>0$ by a similar argument as used in Step 1.

      \vskip0.12in

\noindent{\bf Step 4.}\quad We prove that $m_{\lambda,\beta}$ and $m_{\lambda,\beta}^*$ can not be attained for $\lambda>0$ and $\beta>1$.%

We first prove that $m_{\lambda,\beta}$ can not be attained for $\lambda>0$ and $\beta>1$.  Indeed, suppose that there exists $(u_{\lambda,\beta},v_{\lambda,\beta})\in\mathcal{N}_{\lambda,\beta}$ such that $J_{\lambda,\beta}(u_{\lambda,\beta},v_{\lambda,\beta})=m_{\lambda,\beta}$.  Without loss of generality, we may assume $u_{\lambda,\beta}\geq0$ and $v_{\lambda,\beta}\geq0$.  Clearly, $(u_{\lambda,\beta},v_{\lambda,\beta})\in\mathcal{M}_{\lambda,\beta}$.  By a standard argument, we can see that there exists $t_{\lambda,\beta}>0$ such that $(t_{\lambda,\beta}u_{\lambda,\beta},t_{\lambda,\beta}v_{\lambda,\beta})\in\mathcal{M}_{\beta}^*$.  Since the condition $(D_1)$ holds and $b_0\geq a_0\geq0$,
by Lemmas~\ref{lem0008} and \ref{lem0100}, we have that%
\begin{equation*}
\frac{1}{2(1+\beta)}S^2=J_{\lambda,\beta}(u_{\lambda,\beta},v_{\lambda,\beta})\geq J_{\lambda,\beta}(t_{\lambda,\beta}u_{\lambda,\beta},t_{\lambda,\beta}v_{\lambda,\beta})
\geq\mathcal{E}_\beta(t_{\lambda,\beta}u_{\lambda,\beta},t_{\lambda,\beta}v_{\lambda,\beta}),%
\end{equation*}
which together with Lemma~\ref{lem0010}, implies $\mathcal{E}_\beta(t_{\lambda,\beta}u_{\lambda,\beta},t_{\lambda,\beta}v_{\lambda,\beta})=m_\beta^0$.
Use  a similar argument as that  in the proof
  of Lemma~\ref{lem0010}, we have
  $$D[\mathcal{E}_\beta(t_{\lambda,\beta}u_{\lambda,\beta},t_{\lambda,\beta}v_{\lambda,\beta})]=0  \hbox{ in } \mathcal{D}^*.$$
Therefore, by the maximum principle, we can see that $t_{\lambda,\beta}u_{\lambda,\beta}>0$ and $t_{\lambda,\beta}v_{\lambda,\beta}>0$ on $\bbr^4$.
 Due to \cite[Theorem~3.1]{CZ14} and $\beta>1$,
 we must have $t_{\lambda,\beta}u_{\lambda,\beta}=t_{\lambda,\beta}v_{\lambda,\beta}=U_{\lambda,\beta}$,
 where $U_{\lambda,\beta}$ is given in \cite[Theorem~3.1]{CZ14} and satisfies
 $$\|U_{\lambda,\beta}\|_{L^4(\bbr^4)}^4=\|\nabla U_{\lambda,\beta}\|_{L^2(\bbr^4)}^2=S^2.$$
  It follows from Lemma~\ref{lem0100} and $(u_{\lambda,\beta},v_{\lambda,\beta})\in\mathcal{N}_{\lambda,\beta}$ that%
\begin{eqnarray*}
\frac{S^2}{2(1+\beta)}=J_{\lambda,\beta}(u_{\lambda,\beta},v_{\lambda,\beta})=\frac{(1+\beta)S^2}{2t_{\lambda,\beta}^4},%
\end{eqnarray*}
which then implies $t_{\lambda,\beta}^2=(1+\beta)$.  By Lemma~\ref{lem0100} and $(u_{\lambda,\beta},v_{\lambda,\beta})\in\mathcal{N}_{\lambda,\beta}$ once more, we must have%
\begin{eqnarray}
\frac{S^2}{2(1+\beta)}&=&J_{\lambda,\beta}(u_{\lambda,\beta},v_{\lambda,\beta})\notag\\%
&=&\frac14\mathcal{D}_{\lambda}(u_{\lambda,\beta},v_{\lambda,\beta})\notag\\%
&=&\frac{S^2}{2(1+\beta)}+\int_{\bbr^4}(\lambda a(x)+a_0)u_{\lambda,\beta}^2+(\lambda b(x)+b_0)v_{\lambda,\beta}^2dx.\notag%
\end{eqnarray}
It is impossible since $u_{\lambda,\beta}>0$, $v_{\lambda,\beta}>0$ on $\bbr^4$, $b_0\geq a_0\geq0$ and the conditions $(D_1)$--$(D_3)$ hold.  Note that $m_{\lambda,\beta}^*=m_{\lambda,\beta}$ for all $\lambda>0$ in the case of $\beta>1$, we can also show that $m_{\lambda,\beta}^*$ can not be attained for $\lambda>0$ and $\beta>1$ by a similar argument as above, which completes the proof.%
\end{proof}

\vskip0.12in

We next consider the case of $-\mu_{a,1}<a_0<0\leq b_0$.  Due to Lemma~\ref{lem0110},
we always assume $\lambda>\overline{\Lambda}_a$ in this case.  Let us consider the Nehari type set $\mathcal{M}_{\lambda,\beta}$ in what follows.  Since $\mathcal{D}_{a,\lambda}(u,u)$ and $\mathcal{D}_{b,\lambda}(v,v)$ are both definite on $E_{a,\lambda}$ and $E_{b,\lambda}$ respectively for $\lambda>\overline{\Lambda}_{a}$ in the case of $-\mu_{a,1}<a_0<0\leq b_0$, we can see that Lemma~\ref{lem0008} holds for all $E_{\lambda}\backslash\{(0, 0)\}$ for $\lambda>\overline{\Lambda}_{a}$ and $\beta\geq0$.  Furthermore, we also have the following.%
\begin{lemma}\label{lem1001}
Assume  $(D_1)$-$(D_3)$  and $-\mu_{a,1}<a_0<0\leq b_0$.  If $\beta\geq0$ and $\lambda>\overline{\Lambda}_{a}$, then $m_{\lambda,\beta}^*\geq\frac{S^2(\alpha_{a,1}(\lambda)-1)^2}{4\max\{1,\beta\}[\alpha_{a,1}(\lambda)]^2}>0$.%
\end{lemma}
\begin{proof}
Since $-\mu_{a,1}<a_0<0\leq b_0$  for $\lambda>\overline{\Lambda}_{a}$, we have $E_\lambda=(\widetilde{\mathcal{F}}_{a,\lambda}^{\perp}\oplus\mathcal{F}_{a,\lambda})
\times\mathcal{F}_{b,\lambda}$.  Thanks to Lemmas~\ref{lem0003} and \ref{lem0110} and $-\mu_{a,1}<a_0<0$, for every $\ve>0$, we have from the Sobolev inequality that%
\begin{equation}\label{eq1013}
\ve+m_{\lambda,\beta}^*\geq\frac{S(\alpha_{a,1}(\lambda)-1)}{4\alpha_{a,1}(\lambda)}(\|u_\ve\|^2_{L^4(\bbr^4)}+\|v_\ve\|^2_{L^4(\bbr^4)})%
\end{equation}
for some $(u_\ve,v_\ve)\in\mathcal{M}_{\lambda,\beta}$ with $\lambda>\overline{\Lambda}_{a}$ and $\beta\geq0$.  On the other hand, since $(u_\ve,v_\ve)\in\mathcal{M}_{\lambda,\beta}$ with $\lambda>\overline{\Lambda}_{a}$ and $\beta\geq0$, we can see from the H\"older and Sobolev inequalities that%
\begin{equation}\label{eq1014}
\|u_\ve\|^2_{L^4(\bbr^4)}+\|v_\ve\|^2_{L^4(\bbr^4)}\geq\frac{S(\alpha_{a,1}(\lambda)-1)}{\alpha_{a,1}(\lambda)\max\{1,\beta\}}.%
\end{equation}
Combining \eqref{eq1013} and \eqref{eq1014}, we can obtain that
$m_{\lambda,\beta}^*\geq\frac{S^2(\alpha_{a,1}(\lambda)-1)^2}{4\max\{1,\beta\}[\alpha_{a,1}(\lambda)]^2}>0$ for $\lambda>\overline{\Lambda}_{a}$ and $\beta\geq0$ by letting $\ve\to0^+$.%
\end{proof}

\begin{proposition}\label{prop0003}
Let the conditions $(D_1)$--$(D_3)$ hold and $-\mu_{a,1}<a_0<0\leq b_0$.  If $\beta\geq0$ and $\lambda>\overline{\Lambda}_{a}$, then $J_{\lambda,\beta}(u_{\lambda,\beta}, v_{\lambda,\beta})=m_{\lambda,\beta}^*$ and $D[J_{\lambda,\beta}(u_{\lambda,\beta}, v_{\lambda,\beta})]=0$ in $E_\lambda^*$ for some $(u_{\lambda,\beta}, v_{\lambda,\beta})\in\mathcal{M}_{\lambda,\beta}$.%
\end{proposition}
\begin{proof}
Let $\{(u_n,v_n)\}\subset\mathcal{M}_{\lambda,\beta}$ be a minimizing sequence of $J_{\lambda,\beta}(u,v)$.  Since $\lambda>\overline{\Lambda}_{a}$, by Lemma~\ref{lem0110}, we can see that $\{(u_n,v_n)\}$ is bounded both in $E_\lambda$ and $\mathcal{D}$.  Without loss of generality, we may assume that $(u_n,v_n)\rightharpoonup(u_0,v_0)$ weakly both in $E_\lambda$ and $\mathcal{D}$ and $(u_n,v_n)\to(u_0,v_0)$ a.e. in $\bbr^4\times\bbr^4$ as $n\to\infty$.  Clearly, one of the following two cases must happen:%
\begin{enumerate}
\item[$(i)$] $(u_0,v_0)=(0,0)$;%
\item[$(ii)$] $(u_0,v_0)\not=(0,0)$.%
\end{enumerate}
If the case $(i)$ happen, then by the Sobolev embedding theorem and the condition $(D_2)$, we have that
\begin{equation}\label{eq9998}
\int_{\bbr^4}(\lambda a(x)+a_0)^-u_n^2dx=\int_{\bbr^4}(\lambda b(x)+b_0)^-v_n^2dx=o_n(1).
\end{equation}
It follows from $\{(u_n,v_n)\}\subset\mathcal{M}_{\lambda,\beta}$ and \eqref{eq0003} that
\begin{eqnarray}
S(\|u_n\|_{L^4(\bbr^4)}^2+\|v_n\|_{L^4(\bbr^4)}^2)&\leq&\|u_n\|_{a,\lambda}^2+\|v_n\|_{b,\lambda}^2\notag\\
&=&\mathcal{D}_\lambda(u_n,v_n)+o_n(1)\notag\\
&=&\mathcal{L}_\beta(u_n,v_n)+o_n(1).\label{eq9999}
\end{eqnarray}
If $\beta\leq1$, then we can see from \eqref{eq9999} that $\|u_n\|_{L^4(\bbr^4)}^2+\|v_n\|_{L^4(\bbr^4)}^2\geq S+o_n(1)$.  On the other hand, by \eqref{eq0003} and Lemma~\ref{lem5003}, we have that $$\|u_n\|_{L^4(\bbr^4)}^2+\|v_n\|_{L^4(\bbr^4)}^2\leq 4\min\{m_a, m_b\}S^{-1}.$$  Note that $\min\{m_a, m_b\}<\frac14 S^2$ in the case of  $-\mu_{a,1}<a_0<0\leq b_0$, we get a contradiction.  Thus, we must have $\beta>1$.  By a similar argument as used in Step 5 to Lemma~\ref{lem0100}, we can see that $m_{\lambda,\beta}^*\leq\frac{1}{2(1+\beta)}S^2$ in the case of $-\mu_{a,1}<a_0<0\leq b_0$ for $\beta>1$.  By \eqref{eq9998} and Lemma~\ref{lem1001}, it is easy to see that there exist $0<t_n\leq1+o_n(1)$ such that $(t_nu_n,t_nv_n)\in\mathcal{M}_{\beta}^*$.  Hence, by Lemma~\ref{lem0010} and the fact that $\{(u_n,v_n)\}\subset\mathcal{M}_{\lambda,\beta}$ is a minimizing sequence of $J_{\lambda,\beta}(u,v)$, we can see that
\begin{eqnarray*}
\frac{1}{2(1+\beta)}S^2\leq\frac14\mathcal{L}_\beta(t_nu_n,t_nv_n)
\leq\frac14\mathcal{L}_\beta(u_n,v_n)
\leq\frac{1}{2(1+\beta)}S^2+o_n(1).
\end{eqnarray*}
It follows that $t_n\to1$ as $n\to\infty$, $m_{\lambda,\beta}^*=\frac{1}{2(1+\beta)}S^2$ and $\{(t_nu_n,t_nv_n)\}\subset\mathcal{M}_{\beta}^*$ is a minimizing sequence of $\mathcal{E}_\beta(u,v)$.  Due to a similar argument as used in Case~2 of Lemma~\ref{lem0010}, we can get a contradiction.  Thus, we must have the case~$(ii)$.  In this case, by \eqref{eq1003} and the Fatou lemma, we can see that%
\begin{equation*}
\frac{\mathcal{D}_\lambda(u_0,v_0)^2}{4\mathcal{L}_\beta(u_0,v_0)}\geq m_{\lambda,\beta}^*=\frac14\mathcal{L}_\beta(u_n,v_n)+o_n(1)\geq\frac14\mathcal{L}_\beta(u_0,v_0)+o_n(1).%
\end{equation*}
It follows that%
\begin{equation}\label{eq1004}
\langle D[J_{\lambda,\beta}(u_0,v_0)],(u_0,v_0)\rangle_{E_\lambda^*,E_\lambda}\geq0.%
\end{equation}
Let $(w_n,\sigma_n)=(u_n-u_0,v_n-v_0)$.  Then by the Brez\'is-Lieb lemma and \cite[Lemma~2.3]{CZ14},
  the Sobolev embedding theorem and  n $(D_2)$, we get
\begin{equation*}
\langle D[J_{\lambda,\beta}(w_n,\sigma_n)],(w_n,\sigma_n)\rangle_{E_\lambda^*,E_\lambda}+\langle D[J_{\lambda,\beta}(u_0,v_0)],(u_0,v_0)\rangle_{E_\lambda^*,E_\lambda}=o_n(1),%
\end{equation*}
which together with \eqref{eq1004}, implies%
\begin{equation}\label{eq1005}
\langle D[J_{\lambda,\beta}(w_n,\sigma_n)],(w_n,\sigma_n)\rangle_{E_\lambda^*,E_\lambda}\leq o_n(1).%
\end{equation}
Due to \eqref{eq1003} and \eqref{eq1004}--\eqref{eq1005}, we can use a similar argument as used in the case~$(i)$ to
show that $(w_n,\sigma_n)\to(0, 0)$ strongly in $L^4(\bbr^4)\times L^4(\bbr^4)$ as $n\to\infty$ up to a subsequence.
By \eqref{eq1005}, the Sobolev embedding theorem and the condition $(D_2)$, $(w_n,\sigma_n)\to(0,0)$ strongly
in $E_\lambda$ as $n\to\infty$ up to a sequence.  Hence, $J_{\lambda,\beta}(u_0,v_0)=m_{\lambda,\beta}^*$.
Thanks to  Lemma~\ref{lem5001}, we have that $D[J_{\lambda,\beta}(u_0,v_0)]=0$ in $E_\lambda^*$,
which completes the proof.%
\end{proof}

By Proposition~\ref{prop0003}, we can see that $(\mathcal{P}_{\lambda,\beta})$ has a general ground
state solution $(u_{\lambda,\beta},v_{\lambda,\beta})\in E_\lambda$ for all $\beta\geq0$ and
 $\lambda>\overline{\Lambda}_{a}$.  Furthermore, we have the following
\begin{lemma}\label{lem1002}
Let $(u_{\lambda,0},v_{\lambda,0})$ be the general ground state solution of $(\mathcal{P}_{\lambda,0})$ obtained by Proposition~\ref{prop0003}.
 Then $(u_{\lambda,0},v_{\lambda,0})$ is a semi-trivial solution of $(\mathcal{P}_{\lambda,0})$ and of the type $(u_{\lambda,0}, 0)$.
 Furthermore, $u_{\lambda,0}$ is a least energy critical point of $I_{a,\lambda}(u)$.%
\end{lemma}
\begin{proof}
Suppose $v_{\lambda,0}\not=0$.  Since $(u_{\lambda,0},v_{\lambda,0})$ is a non-zero solution of $(\mathcal{P}_{\lambda,\beta})$,
by the condition $(D_1)$, $\lambda>0$ and $b_0\geq0$, we can see from the Sobolev inequality
 that $\|v_{\lambda,0}\|_{L^4(\bbr^4)}^4\geq S^2$.  Note that the condition $(D_3)$ holds,
 it is well known that $\|v_{\lambda,0}\|_{L^4(\bbr^4)}^4> S^2$.
 Hence, $$J_{\lambda,\beta}(u_{\lambda,0},v_{\lambda,0})\geq\frac14(\|u_{\lambda,0}\|_{L^4(\bbr^4)}^4+\|v_{\lambda,0}\|_{L^4(\bbr^4)}^4)>\frac14 S^2,$$
  which contradicts to Lemma~\ref{lem5003}.  Hence, $(u_{\lambda,0},v_{\lambda,0})$ is a semi-trivial solution of $(\mathcal{P}_{\lambda,0})$
  and of the type $(u_{\lambda,0}, 0)$.  It follows from $\mathcal{N}_{a,\lambda}\times\{0\}\subset\mathcal{M}_{\lambda,0}$ for
  all $\lambda>\overline{\Lambda}_{a}$ that $u_{\lambda,0}$ is also a least
   energy critical point of $I_{a,\lambda}(u)$, where $\mathcal{N}_{a,\lambda}$ is given by \eqref{eq1149}.%
\end{proof}

\vskip0.12in

By Lemma~\ref{lem0003}, we have $\displaystyle \lim_{\lambda\to+\infty}\alpha_{a,1}(\lambda)=\frac{\mu_{a,1}}{|a_0|}<1$ in
the case of $-\mu_{a,1}<a_0<0\leq b_0$.  It follows that for $0<\beta<1-\frac{\mu_{a,1}}{|a_0|}$,
there exists $\Lambda_1^*>\max\{\overline{\Lambda}_a, \overline{\Lambda}_b\}$ such that
$0<\beta<1-\frac{1}{\alpha_{a,1}(\lambda)}$ for $\lambda\geq\Lambda_1^*$.%

\vskip0.12in

\begin{lemma}\label{lem1003}
Let $(u_{\lambda,\beta},v_{\lambda,\beta})$ be the general ground state solution of $(\mathcal{P}_{\lambda,\beta})$ obtained by Proposition~\ref{prop0003}.  Then we have%
\begin{enumerate}
\item[$(1)$] $(u_{\lambda,0}, 0)$ is a general ground state solution of $(\mathcal{P}_{\lambda,\beta})$ for all $\beta<0$.%
\item[$(2)$] For every $\beta\in(0, 1-\frac{\mu_{a,1}}{|a_0|})$, there exists $\Lambda_\beta>\Lambda_1^*$ such that $(u_{\lambda,\beta},v_{\lambda,\beta})$ is a semi-trivial solution of $(\mathcal{P}_{\lambda,\beta})$ and of the type $(u_{\lambda,\beta}, 0)$ with $\lambda>\Lambda_\beta$.%
\item[$(3)$] There exists $\beta_\lambda>0$ such that $(u_{\lambda,\beta},v_{\lambda,\beta})$ is a non-trivial solution of $(\mathcal{P}_{\lambda,\beta})$ for $\beta>\beta_\lambda$.%
\end{enumerate}
\end{lemma}
\begin{proof} $(1)$\quad Clearly, $(u_{\lambda,0}, 0)$ is a solution of $(\mathcal{P}_{\lambda,\beta})$ for all $\beta<0$.  It follows that $m_{\lambda,0}^*\geq m_{\lambda,\beta}^*$ for all $\beta<0$.  On the other hand, since $\beta<0$, for all $(u,v)\in\mathcal{M}_{\lambda,\beta}$, there exists $0<t\leq1$ such that $(tu,tv)\in\mathcal{M}_{\lambda,0}$.  By Lemma~\ref{lem0008}, we have%
\begin{equation*}
J_{\lambda,\beta}(u,v)\geq J_{\lambda,\beta}(tu,tv)\geq J_{\lambda,0}(tu,tv)\geq m_{\lambda,0}^*,%
\end{equation*}
which implies $m_{\lambda,\beta}^*\geq m_{\lambda,0}^*$ for all $\beta<0$.  Therefore, $(u_{\lambda,0}, 0)$ is a general ground state solution of $(\mathcal{P}_{\lambda,\beta})$ for all $\beta<0$.%

\vskip0.12in

$(2)$\quad  Suppose the contrary, there exists $\{\lambda_n\}$ with $\lambda_n\to+\infty$ as $n\to\infty$ and $\beta\in(0, 1)$ such that $(u_{\lambda_n,\beta}, v_{\lambda_n,\beta})\in\mathcal{N}_{\lambda_n,\beta}$ and%
\begin{equation*}
J_{\lambda_n,\beta}(u_{\lambda_n,\beta}, v_{\lambda_n,\beta})=m_{\lambda_n,\beta}^*.%
\end{equation*}
Thanks to Lemma~\ref{lem5003},
we have that $\{(u_{\lambda_n,\beta}, v_{\lambda_n,\beta})\}$ is bounded in $$\mathcal{D}=D^{1,2}(\bbr^4)\times D^{1,2}(\bbr^4).$$
  Without loss of generality, we assume $(u_{\lambda_n,\beta}, v_{\lambda_n,\beta})\rightharpoonup(u_{0,\beta}, v_{0,\beta})$ weakly in $\mathcal{D}$ and $(u_{\lambda_n,\beta}, v_{\lambda_n,\beta})\to(u_{0,\beta}, v_{0,\beta})$ a.e. in $\bbr^4\times\bbr^4$ as $n\to\infty$.  Note that $\lambda_n\to+\infty$ as $n\to\infty$, we can see from the condition $(D_1)$ and Lemma~\ref{lem5003} once more that%
\begin{equation}\label{eq1002}
\int_{\bbr^4}a(x)u_{\lambda_n,\beta}^2+b(x)v_{\lambda_n,\beta}^2dx\to0\quad\text{as }n\to\infty.%
\end{equation}
By  $(D_3)$ and the Fatou's lemma, we see that  $(u_{0,\beta}, v_{0,\beta})\in H_0^1(\Omega_a)\times H_0^1(\Omega_b)$ with $u_{0,\beta}=0$
outside $\Omega_a$ and $v_{0,\beta}=0$ outside $\Omega_b$.  It follows from the Sobolev embedding theorem, the condition $(D_2)$ and \eqref{eq1002} once more that $(u_{\lambda_n,\beta}, v_{\lambda_n,\beta})\to(u_{0,\beta}, v_{0,\beta})$ strongly in $L^2(\bbr^4)\times L^2(\bbr^4)$ as $n\to\infty$.  Since $H_0^1(\Omega_a)\times H_0^1(\Omega_b)\subset E_\lambda$, by the condition $(D_3)$, it is easy to see that $I_a'(u_{0,\beta})=0$ and $I_b'(v_{0,\beta})=0$ in $H^{-1}(\Omega_a)$ and $H^{-1}(\Omega_b)$, respectively.  Thus, we have from Lemma~\ref{lem5003} that%
\begin{eqnarray}
& &\min\{m_a, m_b\}\\
& &\geq \limsup_{n\to\infty}m_{\lambda_n,\beta}^*\notag\\%
& &\geq\liminf_{n\to\infty}
\frac14\mathcal{D}_{\lambda_n}(u_{\lambda_n,\beta}, v_{\lambda_n,\beta})\notag\\%
& &\geq \frac14(\|\nabla u_{0,\beta}\|_{L^2(\bbr^4)}^2+a_0\|u_{0,\beta}\|_{L^2(\bbr^4)}^2\\
& & \quad +\|\nabla v_{0,\beta}\|_{L^2(\bbr^4)}^2+b_0\| v_{0,\beta}\|_{L^2(\bbr^4)}^2).\label{eq1000}%
\end{eqnarray}
Since $\Omega_b$ is bounded and $-\mu_{a,1}<a_0<0\leq b_0$, it is well known that $I_b(v_{0,\beta})>m_b=\frac14S^4$ if $v_{0,\beta}\not=0$, which together with \eqref{eq1000} and the condition $(D_3)$ once more, implies that $v_{0,\beta}=0$ and $u_{0,\beta}$ is a least energy critical point of $I_a(u)$.  Thanks to \eqref{eq1000} and $(u_{\lambda_n,\beta}, v_{\lambda_n,\beta})\to(u_{0,\beta}, v_{0,\beta})$ strongly in $L^2(\bbr^4)\times L^2(\bbr^4)$ as $n\to\infty$ once more, we can see that $(u_{\lambda_n,\beta}, v_{\lambda_n,\beta})\to(u_{0,\beta}, 0)$ strongly in $H^1(\bbr^4)\times H^1(\bbr^4)$ as $n\to\infty$.  Since
$0<\beta<1-\frac{1}{\alpha_{a,1}(\lambda_n)}$ for $n$ sufficiently large, we can see from Lemma~\ref{lem5010} that
\begin{equation*}
\|v_{\lambda_n,\beta}\|_{L^4(\bbr^4)}^2\geq\frac{(1-\frac{1}{\alpha_{a,1}(\lambda_n)})-\beta }{1-\beta(1-\frac{1}{\alpha_{a,1}(\lambda_n)})}S=\frac{(1-\frac{1}{\alpha_{a,1}})-\beta }{1-\beta(1-\frac{1}{\alpha_{a,1}})}S+o_n(1),
\end{equation*}
which is a contradiction.

$(3)$\quad By \eqref{eq1003} and the condition $(D_2)$, we can see that%
\begin{eqnarray*}
m_{\lambda,\beta}^*&\leq&\frac{(\mathcal{D}_{a,\lambda}(u_{\lambda,0},u_{\lambda,0})+\mathcal{D}_{b,\lambda}(u_{\lambda,0},u_{\lambda,0}))^2}
{8(1+\beta)\|u_{\lambda,0}\|_{L^4(\bbr^4)}^4}\\%
&\leq&\frac{2m_{a,\lambda}}{1+\beta}+\frac{(\lambda(a_\infty+b_\infty)+b_0-a_0)\|u_{\lambda,0}\|_{L^2(\bbr^4)}^2}{2(1+\beta)}\\%
&&+\frac{(\lambda(a_\infty+b_\infty)+b_0-a_0)^2\|u_{\lambda,0}\|_{L^2(\bbr^4)}^4}{8(1+\beta)\|u_{\lambda,0}\|_{L^4(\bbr^4)}^4},%
\end{eqnarray*}
where $m_{a,\lambda}=\inf_{\mathcal{N}_{a,\lambda}}I_{a,\lambda}(u)$.
Thus, $m_{\lambda,\beta}^*\to0$ as $\beta\to+\infty$.  By $-\mu_{a,1}<a_0<0$ and Lemma~\ref{lem0110},
 we have $m_{a,\lambda}>0$ for all $\lambda>\overline{\Lambda}_a$.
  Thus, there exists $\beta_\lambda\in(0, +\infty)$ such that $m_{\lambda,\beta}^*<m_{a,\lambda}$
   for $\beta>\beta_{\lambda}$.  Since $-\mu_{a,1}<a_0<0$, it is easy to show that
    $m_{a,\lambda}\leq\frac14 S^2$.  If $(u_{\lambda,\beta}, v_{\lambda,\beta})$ is
    of the type $(0, v_{\lambda,\beta})$ for some $\beta>\beta_\lambda$, then by a
    similar argument as used in the proof of Lemma~\ref{lem1002}, we can see that
    $m_{\lambda,\beta}>\frac14 S^2$, which is impossible.  If $(u_{\lambda,\beta}, v_{\lambda,\beta})$
     is of the type $(u_{\lambda,\beta}, 0)$ for some $\beta>\beta_\lambda$, then also by a similar
      argument as used in the proof of Lemma~\ref{lem1002}, we can obtain that $m_{\lambda,\beta}^*\geq m_{a,\lambda}$,
       which is also a contradiction.  Therefore,   $(u_{\lambda,\beta}, v_{\lambda,\beta})$ must be  a
       non-trivial solution of $(\mathcal{P}_{\lambda,\beta})$ for $\beta>\beta_\lambda$.%
\end{proof}

\vskip0.13in

Next, we consider the case of $-\mu_{a,1}<a_0<0$, $-\mu_{b,1}<b_0<0$ and $b_0\geq a_0$.  Due to Lemma~\ref{lem0110}, we always assume  $\lambda>\max\{\overline{\Lambda}_a, \overline{\Lambda}_b\}$ in this case.
\begin{lemma}\label{lem7001}
Assume  $(D_1)$-$(D_3)$   and $-\mu_{a,1}<a_0<0$ and $-\mu_{b,1}<b_0<0$.  Then $\lim_{\lambda\to+\infty}m_{a,\lambda}=m_a$ and $\lim_{\lambda\to+\infty}m_{b,\lambda}=m_b$, where $m_{a}$ and $m_b$ are given by Lemma~\ref{lem5003} and $m_{a,\lambda}$ and $m_{b,\lambda}$ are given by Lemma~\ref{lem5002}.
\end{lemma}
\begin{proof}
We only give the proof of $\lim_{\lambda\to+\infty}m_{a,\lambda}=m_a$.  Due to the condition $(D_1)$, it is easy to show that $m_{a,\lambda}$ is nondecreasing by $\lambda$. Thus, combine with  $(D_3)$, it implies $\lim_{\lambda\to+\infty}m_{a,\lambda}\leq m_a$.  By Lemma~\ref{lem1002}, $m_{a,\lambda}$ can be attained by some $u_{\lambda, 0}\in E_\lambda$ for $\lambda>\overline{\Lambda}_{a}$.  Now, thanks to a similar argument as used in the proof of $(2)$ to Lemma~\ref{lem1003}, for every $\lambda_n\to+\infty$ as $n\to\infty$, we can see that $u_{\lambda_n,0}\rightharpoonup u_{0,0}$ weakly in $D^{1,2}(\bbr^4)$, and $u_{\lambda_n,0}\to u_{0,0}$ a.e. in $\bbr^4$ and $u_{\lambda_n,0}\to u_{0,0}$ strongly in $L^2(\bbr^4)$ as $n\to\infty$ up to a subsequence.  It $u_{0,0}=0$, then by the condition $(D_2)$ and \eqref{eq1002}, we have
\begin{eqnarray*}
S\|u_{\lambda_n,0}\|_{L^4(\bbr^4)}^2\leq\|u_{\lambda_n,0}\|_{a,\lambda}^2=\mathcal{D}_{a,\lambda}(u_{\lambda_n,0},u_{\lambda_n,0})+o_n(1)
=\|u_{\lambda_n,0}\|_{L^4(\bbr^4)}^4+o_n(1).
\end{eqnarray*}
It follows that $\|u_{\lambda_n,0}\|_{L^4(\bbr^4)}^4\geq S^2+o_n(1)$.  Thus, we must have that $m_{a,\lambda_n}\geq \frac14 S^2+o_n(1)$.  It is impossible since $m_a<\frac14S^2$ due to $a_0<0$.  Thus, we must have $u_{0,0}\not=0$.  By a similar argument as used in the proof of $(2)$ to Lemma~\ref{lem1003}, we have that $I_a'(u_{0,0})=0$ in $H^{-1}(\Omega_a)$.  Therefore, $\lim_{\lambda\to+\infty}m_{a,\lambda}=m_a$.
\end{proof}

Thanks to Lemma~\ref{lem0110}, we can see that every minimizing sequence of $J_{\lambda,\beta}(u,v)$ on $\mathcal{N}_{\lambda,\beta}$ is bounded in $E_\lambda$ in this case.  Using Lemma~\ref{lem0020}, the implicit function theorem and the Ekeland principle in a standard way (cf. \cite{CZ121}), we can obtain a $(PS)_{m_{\lambda,\beta}}$ sequence of $J_{\lambda,\beta}(u,v)$ in $\mathcal{N}_{\lambda,\beta}$ for $\beta\leq0$, denoted by $\{(u_n,v_n)\}$.%
\begin{proposition}
Let the conditions $(D_1)$--$(D_3)$ hold and $\beta\leq0$.  Then $m_{\lambda,\beta}$ can be attained by a ground state solution of $(\mathcal{P}_{\lambda,\beta})$ for $\lambda>\max\{\overline{\Lambda}_a, \overline{\Lambda}_b\}$.%
\end{proposition}
\begin{proof}
Since $\{(u_n,v_n)\}$ is bounded in $E_\lambda$, without loss of generality, we may assume $(u_n,v_n)\rightharpoonup(u_0,v_0)$ weakly in $E_\lambda$ as $n\to\infty$.  It follows that $D[J_{\lambda,\beta}(u_0,v_0)]=0$ in $E_\lambda^*$.%

\vskip0.1in
\noindent{\it Case~1: } $u_0=0$ and $v_0=0$. In this case, by a similar argument as used in the proof of Proposition~\ref{prop0003}, we can obtain that $\|u_n\|_{L^4(\bbr^4)}^2\geq S+o_n(1)$ and $\|v_n\|_{L^4(\bbr^4)}^2\geq S+o_n(1)$.  It follows that $m_{\lambda,\beta}\geq\frac12 S^2$, which contradicts to Lemma~\ref{lem5003}.

\vskip0.1in
\noindent{\it Case~2: } $u_0=0$ and $v_0\not=0$. Let $\sigma_n=v_n-v_0$.  Then $(u_n,\sigma_n)\rightharpoonup(0,0)$ weakly in $E_\lambda$ as $n\to\infty$.  It follows from the Sobolev inequality, the Brez\'is-Lieb lemma and \cite[Lemma~2.3]{CZ14} that
\be\label{eq7002}
J_{\lambda,\beta}(u_n,v_n)=J_{\lambda,\beta}(u_n,\sigma_n)+I_{b,\lambda}(v_0)+o_n(1),\ee
\be\label{eq7003} \langle D[J_{\lambda,\beta}(u_n,\sigma_n)], (u_n,0)\rangle_{E_{\lambda}^*, E_\lambda}=\langle D[J_{\lambda,\beta}(u_n,v_n)], (u_n,0)\rangle_{E_{\lambda}^*, E_\lambda}+o_n(1),\ee
 $$ \langle D[J_{\lambda,\beta}(u_n,\sigma_n)], (0,\sigma_n)\rangle_{E_{\lambda}^*, E_\lambda}$$
\be\label{eq7004} = \langle D[J_{\lambda,\beta}(u_n,v_n)], (0,v_n)\rangle_{E_{\lambda}^*, E_\lambda}+I_{b,\lambda}'(v_0)v_0+o_n(1).\ee
Since $u_n^2v_n\rightharpoonup0$ in $L^{\frac43}(\bbr^4)$, by $D[J_{\lambda,\beta}(u_n,v_n)]=o_n(1)$ strongly in $E_\lambda^*$ as $n\to\infty$, we must have $I_{b,\lambda}'(v_0)v_0=0$.  It follows that $I_{b,\lambda}(v_0)\geq m_{b,\lambda}$.  If $\|\sigma_n\|_{L^4(\bbr^4)}\geq C+o_n(1)$, then by a similar argument as used in the proof of Proposition~\ref{prop0003}, we can see from \eqref{eq7003} and \eqref{eq7004} that $\|u_n\|_{L^4(\bbr^4)}^2\geq S+o_n(1)$ and $\|\sigma_n\|_{L^4(\bbr^4)}^2\geq S+o_n(1)$, which together with \eqref{eq7002}, implies $m_{\lambda,\beta}\geq\frac12 S^2$.  It contradicts to Lemma~\ref{lem5003}.  Thus, we must have $\sigma_n\to0$ strongly in $L^4(\bbr^4)$ as $n\to\infty$.  Since $u_0=0$, we still have $\|u_n\|_{L^4(\bbr^4)}^2\geq S+o_n(1)$.  Now, by \eqref{eq7002} once more, we can see that $m_{\lambda,\beta}\geq \frac14 S^2+m_{b, \lambda}$.  Thanks to Lemmas~\ref{lem5003} and \ref{lem7001}, it is impossible for $\lambda$ sufficient large.  Without loss of generality, we assume $m_{\lambda,\beta}\geq \frac14 S^2+m_{b, \lambda}$ can not hold for $\lambda>\max\{\overline{\Lambda}_a, \overline{\Lambda}_b\}$.

\vskip0.1in
\noindent {\bf Case~3:} $u_0\not=0$ and $v_0=0$.  We can exclude this case by a similar argument as used in the Case~2.%

\vskip0.12in

Combine with the above 3 cases, now we must have $u_0\not=0$ and $v_0\not=0$.  It follows that $(u_0,v_0)\in\mathcal{N}_{\lambda,\beta}$,
 which together with the Fatou's lemma, implies that $J_{\lambda,\beta}(u_0,v_0)=m_{\lambda,\beta}$.
   Thus, $(u_0,v_0)$ is a ground state solution of $(\mathcal{P}_{\lambda,\beta})$ for $\lambda\geq\Lambda_1^*$ and $\beta\leq0$.%
\end{proof}

\vskip0.3in

We next consider the general ground state solution of $(\mathcal{P}_{\lambda,\beta})$ in the case of $-\mu_{a,1}<a_0<0$, $-\mu_{b,1}<b_0<0$ and $b_0\geq a_0$.
By Lemma~\ref{lem0003}, we have $\lim_{\lambda\to+\infty}\alpha_{a,1}(\lambda)=\frac{\mu_{a,1}}{|a_0|}<1$ and $\lim_{\lambda\to+\infty}\alpha_{b,1}(\lambda)=\frac{\mu_{b,1}}{|b_0|}<1$ in this case.  Let%
\begin{equation*}
\beta_0:=\min\bigg\{\frac12(1-\frac{|a_0|}{\mu_{a,1}})(1-\frac{|b_0|}{\mu_{b,1}}),
\frac{1-\frac{|b_0|}{\mu_{b,1}}}{1-\frac{|a_0|}{\mu_{a,1}}},
\frac{1-\frac{|a_0|}{\mu_{a,1}}}{1-\frac{|b_0|}{\mu_{b,1}}}\bigg\}.%
\end{equation*}
Then it is easy to see that $\beta_0\leq1$.  It follows that for $0<\beta<\beta_0$, there exists $\Lambda_2^*>\max\{\overline{\Lambda}_a, \overline{\Lambda}_b\}$ such that%
\begin{equation*}
0<\beta<\min\bigg\{\frac12(1-\frac{1}{\alpha_{a,1}(\lambda)})(1-\frac{1}{\alpha_{b,1}(\lambda)}),
\frac{1-\frac{1}{\alpha_{b,1}(\lambda)}}{1-\frac{1}{\alpha_{a,1}(\lambda)}},
\frac{1-\frac{1}{\alpha_{a,1}(\lambda)}}{1-\frac{1}{\alpha_{b,1}(\lambda)}}\bigg\}%
\end{equation*}
as long as  $\lambda\geq\Lambda_2^*$.  By checking the proofs of Proposition~\ref{prop0003} and Lemma~\ref{lem1003}, we can see that they still work for $\lambda\geq\Lambda_2^*$ and $\beta<\beta_0$ since Lemma~\ref{lem6010} holds.  Thus, we can obtain the following.%
\begin{proposition}\label{prop0006}
Let the conditions $(D_1)$--$(D_3)$ hold and $\beta\geq0$. If $-\mu_{a,1}<a_0<0$, $-\mu_{b,1}<b_0<0$ and $\lambda>\max\{\overline{\Lambda}_a, \overline{\Lambda}_b\}$, then $m_{\lambda,\beta}^*$ can be attained by a general ground state solution of $(\mathcal{P}_{\lambda,\beta})$.  Moreover, we have the following.%
\begin{enumerate}
\item[$(1)$] The general ground state solution of $(\mathcal{P}_{\lambda,0})$ is also a general ground state solution of $(\mathcal{P}_{\lambda,\beta})$ for $\beta<0$.%
\item[$(2)$] For every $\beta\in(0, \beta_0)$, there exists $\Lambda_\beta>\Lambda_1^*$ such that $(u_{\lambda,\beta},v_{\lambda,\beta})$ is a semi-trivial solution of $(\mathcal{P}_{\lambda,\beta})$ and of the type $(u_{\lambda,\beta}, 0)$ with $\lambda>\Lambda_\beta$.
\item[$(3)$] There exists $\beta_\lambda>0$ such that $m_{\lambda,\beta}^*$ can be attained by a ground state solution of $(\mathcal{P}_{\lambda,\beta})$ with $\beta\geq\beta_\lambda$.%
\end{enumerate}
\end{proposition}

\vskip0.3in
Now, we are ready to  give the proof of Theorem~\ref{thm0001}.%

\noindent\textbf{Proof of Theorem~\ref{thm0001}:} In fact, it is a straightforward consequence of
 Lemmas~\ref{lem0100} and \ref{lem1002}-\ref{lem1003} and Propositions~\ref{prop0004}-\ref{prop0006}.
\qquad\raisebox{-0.5mm}
 {
\rule{1.5mm}{4mm}}\vspace{6pt}

In the following part of this section, we will consider the case of $a_0\leq-\mu_{a,1}$ or $b_0\leq-\mu_{b,1}$.
Let $\mathcal{G}_{\lambda,\beta}$ be the Nehari-Pankov type set related to $\mathcal{M}_{\lambda,\beta}$,
which is given by \eqref{eq6001}.  Then it is easy to see that $\mathcal{G}_{\lambda,\beta}$ contains all
 non-zero solutions of $(\mathcal{P}_{\lambda,\beta})$.  In what follows, we will borrow some ideas
 from \cite{SW09} to show that $c_{\lambda,\beta}=\inf_{\mathcal{G}_{\lambda,\beta}}J_{\lambda,\beta}(u,v)$
 can be attained by some non-zero solutions of $(\mathcal{P}_{\lambda,\beta})$ for $\lambda$ sufficiently
 large and $0\leq\beta<1$.  Denote the map $(u,v)\to(u_{\lambda,\beta}^0,v_{\lambda,\beta}^0)$
 by $(\check{m}_{\lambda,\beta}^0(u,v), \hat{m}_{\lambda,\beta}^0(u,v))$,
  where $(u_{\lambda,\beta}^0,v_{\lambda,\beta}^0)$ is given by Lemma~\ref{lem0011}.  Then we can obtain the following

\begin{lemma}\label{lem0013}
Assume  $(D_1)$-$(D_3)$   and $a_0\leq-\mu_{a,1}$ or $b_0\leq-\mu_{b,1}$.  If $0\leq\beta<1$ and $\lambda\geq\Lambda_0^*$, then the map $(u,v)\to(\check{m}_{\lambda,\beta}^0(u,v), \hat{m}_{\lambda,\beta}^0(u,v))$ is continuous on $\widetilde{E}_{\lambda}$, where $\Lambda_0^*$ is given by Lemma~\ref{lem0007}.%
\end{lemma}
\begin{proof}
We only give the proof for the case $a_0\leq-\mu_{a,1}$ and $b_0\leq-\mu_{b,1}$, since other cases are more simple and can be proved in a similar way due to Lemma~\ref{lem0110}.  Let $(u_n,v_n)\to(u,v)$ strongly in $\widetilde{E}_\lambda$ as $n\to\infty$.  By Lemma~\ref{lem0011}, we can see that%
\begin{equation*}
(\check{m}_{\lambda,\beta}^0(u_n,v_n), \hat{m}_{\lambda,\beta}^0(u_n,v_n))=(w_n^0+t_n^0\widetilde{u}_n, \sigma_n^0+t_n^0\widetilde{v}_n)%
\end{equation*}
and
\begin{equation*}
(\check{m}_{\lambda,\beta}^0(u,v), \hat{m}_{\lambda,\beta}^0(u,v))=(w^0+t^0\widetilde{u}, \sigma^0+t^0\widetilde{v}).%
\end{equation*}
Since $0\leq\beta<1$ and dim$(\widehat{\mathcal{F}}_{a,\lambda}^{\perp}\oplus\bbr^+ \widetilde{u})\times(\widehat{\mathcal{F}}_{b,\lambda}^{\perp}\oplus\bbr^+ \widetilde{v})<+\infty$ due to Lemma~\ref{lem0001}, there exists $R_\lambda>0$ such that $J_{\lambda,\beta}(R_\lambda w,R_\lambda\sigma)\leq-1$ for all $(w,\sigma)\in(\widehat{\mathcal{F}}_{a,\lambda}^{\perp}\oplus\bbr^+ \widetilde{u})\times(\widehat{\mathcal{F}}_{b,\lambda}^{\perp}\oplus\bbr^+ \widetilde{v})$ with $\|w\|_{a,\lambda}^2+\|\sigma\|_{b,\lambda}^2=1$.  Since $(u_n,v_n)\to(u,v)$ strongly in $\widetilde{E}_\lambda$ as $n\to\infty$, we have $J_{\lambda,\beta}(R_\lambda w,R_\lambda\sigma)\leq0$ for all $(w,\sigma)\in(\widehat{\mathcal{F}}_{a,\lambda}^{\perp}\oplus\bbr^+ \widetilde{u}_n)\times(\widehat{\mathcal{F}}_{b,\lambda}^{\perp}\oplus\bbr^+ \widetilde{v}_n)$ with $\|w\|_{a,\lambda}^2+\|\sigma\|_{b,\lambda}^2=1$ and $n$ sufficiently large.  It follows from \eqref{eq0057} that $\{(w_n^0+t_n^0\widetilde{u}_n, \sigma_n^0+t_n^0\widetilde{v}_n)\}$ is bounded in $\widetilde{E}_\lambda$.  Without loss of generality, we may assume that $(w_n^0+t_n^0\widetilde{u}_n, \sigma_n^0+t_n^0\widetilde{v}_n)\rightharpoonup(w_0^0+t_0^0\widetilde{u}, \sigma_0^0+t_0^0\widetilde{v})$ weakly in $\widetilde{E}_\lambda$ and $(w_n^0+t_n^0\widetilde{u}_n, \sigma_n^0+t_n^0\widetilde{v}_n)\to(w_0^0+t_0^0\widetilde{u}, \sigma_0^0+t_0^0\widetilde{v})$ a.e. in $\bbr^4\times\bbr^4$ as $n\to\infty$.  Since dim$\widehat{\mathcal{F}}_{a,\lambda}^{\perp}\times\widehat{\mathcal{F}}_{b,\lambda}^{\perp}<+\infty$ and $(u_n,v_n)\to(u,v)$ strongly in $\widetilde{E}_\lambda$ as $n\to\infty$, we have $(w_n^0+t_n^0\widetilde{u}_n, \sigma_n^0+t_n^0\widetilde{v}_n)\to(w_0^0+t_0^0\widetilde{u}, \sigma_0^0+t_0^0\widetilde{v})$ strongly in $\widetilde{E}_\lambda$ as $n\to\infty$.  Now, by \eqref{eq0057}, we can see that%
$$
J_{\lambda,\beta}(w_n^0+t_n^0\widetilde{u}_n, \sigma_n^0+t_n^0\widetilde{v}_n)\geq
J_{\lambda,\beta}(w^0+t^0\widetilde{u}_n, \sigma^0+t^0\widetilde{v}_n)$$
$$\quad\quad\quad\quad\quad\quad\quad\quad\quad\quad\quad\quad\quad\quad\quad=J_{\lambda,\beta}(w^0+t^0\widetilde{u}, \sigma^0+t^0\widetilde{v})+o_n(1).$$
and
$$
J_{\lambda,\beta}(w^0+t^0\widetilde{u}, \sigma^0+t^0\widetilde{v})\geq J_{\lambda,\beta}(w_0^0+t_0^0\widetilde{u}, \sigma_0^0+t_0^0\widetilde{v})$$
$$\quad\quad\quad\quad\quad\quad\quad\quad\quad\quad\quad\quad\quad\quad\quad\quad
= J_{\lambda,\beta}(w_n^0+t_n^0\widetilde{u}_n, \sigma_n^0+t_n^0\widetilde{v}_n)+o_n(1).$$
Note that $(w^0,\sigma^0,t^0)$ is the unique one satisfying \eqref{eq0057} for $(u,v)$,
we must have $w_0^0=w^0$, $\sigma^0_0=\sigma^0$ and $t^0_0=t^0$.  Since $(u_n,v_n)\to(u,v)$
strongly in $\widetilde{E}_\lambda$ as $n\to\infty$, we can see that the
map $(\check{m}_{\lambda,\beta}^0(u,v), \hat{m}_{\lambda,\beta}^0(u,v))$ is continuous
on $\widetilde{E}_{\lambda}$ for $0\leq\beta<1$ and $\lambda\geq\Lambda_0^*$.%
\end{proof}

Let%
\begin{equation*}
\mathbb{B}_1^+:=\{(u,v)\in \widetilde{E}_\lambda
\mid\|u\|_{a,\lambda}^2+\|v\|_{b,\lambda}^2=1\}%
\end{equation*}
and consider the following functional%
\begin{equation*}
\Phi_{\lambda,\beta}(u,v):=J_{\lambda,\beta}(\check{m}_{\lambda,\beta}^0(u,v), \hat{m}_{\lambda,\beta}^0(u,v)).%
\end{equation*}
Then by Lemmas~\ref{lem0011} and \ref{lem0013}, $\Phi_{\lambda,\beta}(u,v)$ is well defined and
 continuous on $\mathbb{B}_1^+$ for $0\leq\beta<1$ and $\lambda\geq\Lambda_0^*$.  Furthermore, we also have the following
\begin{lemma}\label{lem0012}
Assume  that  $(D_1)$-$(D_3)$ hold and that either     $a_0\leq-\mu_{a,1}$ or $b_0\leq-\mu_{b,1}$.  If $0\leq\beta<1$ and $\lambda\geq\Lambda_0^*$, then $\Phi_{\lambda,\beta}(u,v)$ is $C^1$ on $\mathbb{B}_1^+$.  Moreover,%
\begin{equation}\label{eq0060}
\langle D[\Phi_{\lambda,\beta}(u,v)], (w,\sigma)\rangle_{E^*_\lambda, E_\lambda}=\langle D[J_{\lambda,\beta}(\check{m}_{\lambda,\beta}^0(u,v), \hat{m}_{\lambda,\beta}^0(u,v))], (t_\lambda^0\widetilde{w}, t_\lambda^0\widetilde{\sigma})\rangle_{E^*_\lambda, E_\lambda}%
\end{equation}
for all $(u,v)$ and $(w,\sigma)\in\mathbb{B}_1^+$, where $t_\lambda^0\in\bbr^+$ is given by \eqref{eq0057} and only depends
 on $(u,v)$;  $\widetilde{w}$  and $\widetilde{\sigma}$ are the projections of $w$ and $\sigma$
  on $\widetilde{\mathcal{F}}_{a,\lambda}^{\perp}\oplus\mathcal{F}_{a,\lambda}$ and
  $\widetilde{\mathcal{F}}_{b,\lambda}^{\perp}\oplus\mathcal{F}_{b,\lambda}$.%
\end{lemma}
\begin{proof}
We just  give the proof for the case $a_0\leq-\mu_{a,1}$ and $b_0\leq-\mu_{b,1}$.
  Let $(w,\sigma)$ and $(u,v)\in \mathbb{B}_1^+$.
  Then $w=\widehat{w}+\widetilde{w}$, $u=\widehat{u}+\widetilde{u}$ and $\sigma=\widehat{\sigma}+\widetilde{\sigma}$,
  $v=\widehat{v}+\widetilde{v}$ with $\widetilde{u}\not=0$ or $\widetilde{v}\not=0$ and $\widetilde{w}\not=0$
   or $\widetilde{\sigma}\not=0$, where $\widehat{w}$, $\widehat{u}$ and $\widehat{\sigma}$, $\widehat{v}$
    are the projections of $w$, $u$ and $\sigma$, $v$ on $\widehat{\mathcal{F}}_{a,\lambda}^{\perp}$ and
    $\widehat{\mathcal{F}}_{b,\lambda}^{\perp}$ while $\widetilde{w}$, $\widetilde{u}$ and $\widetilde{\sigma}$,
     $\widetilde{v}$ are the projections of $w$, $u$ and $\sigma$, $v$ on
      $\widetilde{\mathcal{F}}_{a,\lambda}^{\perp}\oplus\mathcal{F}_{a,\lambda}$ and
       $\widetilde{\mathcal{F}}_{b,\lambda}^{\perp}\oplus\mathcal{F}_{b,\lambda}$.
       By the implicit function theorem, it is easy to see that there exist $\delta>0$ and a $C^1$
       function $t(l)$ on $(-\delta, \delta)$ satisfying $t(l)\in[\frac12, \frac32]$, $t(0)=1$ and
        $t'(0)=-(\langle u,w\rangle_{a,\lambda}+\langle v,\sigma\rangle_{b,\lambda})$ such that
        $(w_l,\sigma_l)\in\mathbb{B}_1^+$ for $l\in(-\delta, \delta)$, where $(w_l,\sigma_l)=(t(l)u+lw,t(l)v+l\sigma)$.
         Now, by the mean value theorem and Lemma~\ref{lem0011}, we can see that%
\begin{eqnarray}
&&\Phi_{\lambda,\beta}(u,v)-\Phi_{\lambda,\beta}(w_l,\sigma_l)\notag\\
&&\leq J_{\lambda,\beta}(w_\lambda^0+t_\lambda^0\widetilde{u}, \sigma_\lambda^0+t_\lambda^0\widetilde{v})\notag\\
&&\quad -J_{\lambda,\beta}(w_\lambda^0+t_\lambda^0(t(l)\widetilde{u}
+l\widetilde{w}), \sigma_\lambda^0+t_\lambda^0(t(l)\widetilde{v}+l\widetilde{\sigma}))\label{eq0059}\\
&&=-\langle D[J_{\lambda,\beta}(\rho_{1}(l),\rho_{2}(l))], (t_\lambda^0((t(l)-1)\widetilde{u}+l\widetilde{w}), t_\lambda^0((t(l)-1)\widetilde{v}+l\widetilde{\sigma}))\rangle_{E^*_\lambda, E_\lambda}\notag
\end{eqnarray}
and
\begin{eqnarray}
&&\Phi_{\lambda,\beta}(u,v)-\Phi_{\lambda,\beta}(w_l,\sigma_l)\notag\\
& & \geq J_{\lambda,\beta}(w_{\lambda,l}^0+t_{\lambda,l}^0\widetilde{u}, \sigma_{\lambda,l}^0+t_{\lambda,l}^0\widetilde{v})\notag\\
&&\quad -J_{\lambda,\beta}(w_{\lambda,l}^0+t_{\lambda,l}^0(t(l)\widetilde{u}+l\widetilde{w}),
 \sigma_{\lambda,l}^0+t_{\lambda,l}^0(t(l)\widetilde{v}+l\widetilde{\sigma}))\label{eq0058}\\%
&&=-l\langle D[J_{\lambda,\beta}(\rho^*_{1}(l),\rho^*_{2}(l))], (t_{\lambda,l}^0((t(l)-1)\widetilde{u}+l\widetilde{w}),
t_{\lambda,l}^0((t(l)-1)\widetilde{v}+l\widetilde{\sigma}))\rangle_{E^*_\lambda, E_\lambda},\notag%
\end{eqnarray}
where $$\rho_{1}(l)=w_\lambda^0+t_\lambda^0(t(l)\rho_l^1+(1-\rho_l^1))\widetilde{u}+l(1-\rho_l^1)\widetilde{w},$$
 $$\rho_{2}(l)=\sigma_\lambda^0+t_\lambda^0(t(l)\rho_l^2+(1-\rho_l^2))\widetilde{v}+l(1-\rho_l^2)\widetilde{\sigma},$$
  $$\rho^*_{1}(l)=w_{\lambda,l}^0+t_{\lambda,l}^0(t(l)\rho_l^{1,*}+(1-\rho_l^{1,*}))\widetilde{u}+l(1-\rho_l^{1,*})\widetilde{w}$$
   and $$\rho^*_{2}(l)=\sigma_{\lambda,l}^0+t_{\lambda,l}^0(t(l)\rho_l^{2,*}+(1-\rho_l^{2,*}))\widetilde{v}+l(1-\rho_l^{2,*})\widetilde{\sigma}$$ with
$\rho_l^1,\rho_l^2,\rho_l^{1,*},\rho_l^{2,*}\in(0, 1)$.  Since $(w_l,\sigma_l)\to(u,v)$ as $l\to0$, by Lemmas~\ref{lem0011} and \ref{lem0013} and \eqref{eq0059}--\eqref{eq0058}, we have from $(\check{m}_{\lambda,\beta}^0(u,v), \hat{m}_{\lambda,\beta}^0(u,v)))\in\mathcal{G}_{\lambda,\beta}$ that%
\begin{eqnarray*}
&&\lim_{l\to0^+}\frac{\Phi_{\lambda,\beta}(w_l,\sigma_l)-\Phi_{\lambda,\beta}(u,v)}{l}\\
&=&\langle D[J_{\lambda,\beta}(\check{m}_{\lambda,\beta}^0(u,v), \hat{m}_{\lambda,\beta}^0(u,v))], (t_\lambda^0(t'(0)\widetilde{u}+\widetilde{w}), t_\lambda^0(t'(0)\widetilde{v}+\widetilde{\sigma}))\rangle_{E^*_\lambda, E_\lambda}\\%
&=&\langle D[J_{\lambda,\beta}(\check{m}_{\lambda,\beta}^0(u,v), \hat{m}_{\lambda,\beta}^0(u,v))], (t_\lambda^0\widetilde{w}, t_\lambda^0\widetilde{\sigma})\rangle_{E^*_\lambda, E_\lambda},%
\end{eqnarray*}
where $t_\lambda^0\in\bbr^+$ is given by \eqref{eq0057} and only relies   on $(u,v)$.
It follows that $\Phi_{\lambda,\beta}(u,v)$ is of  $C^1$ on $\mathbb{B}_1^+$ and \eqref{eq0060} holds.%
\end{proof}

Thanks to Lemma~\ref{lem0012}, we can obtain the following
\begin{proposition}\label{prop0001}
Let  $(D_1)$-$(D_3)$ hold. Assume that  either  $a_0\leq-\mu_{a,1}$ with $-a_0\not\in\sigma(-\Delta, H_0^1(\Omega_a))$
or $b_0\leq-\mu_{b,1}$ with $-b_0\not\in\sigma(-\Delta, H_0^1(\Omega_b))$.  If $0\leq\beta<1$ and $\lambda\geq\Lambda_0^*$, then $c_{\lambda,\beta}$ can be attained by a non-zero solution of $(\mathcal{P}_{\lambda,\beta})$.%
\end{proposition}
\begin{proof}
We only consider  the case $a_0\leq-\mu_{a,1}$ with $-a_0\not\in\sigma(-\Delta, H_0^1(\Omega_a))$
and $b_0\leq-\mu_{b,1}$ with $-b_0\not\in\sigma(-\Delta, H_0^1(\Omega_b))$,  the  other cases are more simple  in view of  Lemma~\ref{lem0110}.  Let
\begin{equation*}
\widetilde{c}_{\lambda,\beta}=\inf_{\mathbb{B}_1^+}\Phi_{\lambda,\beta}(u,v).%
\end{equation*}
Then by Lemmas~\ref{lem0004} and \ref{lem0011}, it is easy to see that $\widetilde{c}_{\lambda,\beta}=c_{\lambda,\beta}$.  Since $\Phi_{\lambda,\beta}$ is $C^1$ due to Lemma~\ref{lem0012}, we can apply the implicit function theorem and the Ekeland principle in a standard way (cf. \cite{CZ121}) to show that there exists
$\{(u_n,v_n)\}\subset\mathbb{B}_{1}^+$ such that $\Phi_{\lambda,\beta}(u_n,v_n)=\widetilde{c}_{\lambda,\beta}+o_n(1)$
 and $\langle D[\Phi_{\lambda,\beta}(u_n,v_n)], (w,\sigma)\rangle_{E^*_\lambda, E_\lambda}=o_n(1)$
  for all $(w,\sigma)\in\mathbb{B}_1^+$.  For the sake of clarity, the remaining  proof will be   performed by several steps.%

\vskip0.12in

\noindent{\bf Step 1. }We prove that $\{\check{m}_{\lambda,\beta}^0(u_n,v_n), \hat{m}_{\lambda,\beta}^0(u_n,v_n)\}$ is bounded in $E_\lambda$.

For the sake of convenience, we denote $(\check{m}_{\lambda,\beta}^0(u_n,v_n), \hat{m}_{\lambda,\beta}^0(u_n,v_n))$ by $(\phi_n,\psi_n)$.  By the definition of $\Phi_{\lambda,\beta}(u,v)$, it is easy to see that $(\phi_n,\psi_n)\in\mathcal{G}_{\lambda,\beta}$ and $J_{\lambda,\beta}(\phi_n,\psi_n)=c_{\lambda,\beta}+o_n(1)$.  It follows that%
\begin{equation}   \label{eq0054}
c_{\lambda,\beta}+o_n(1)=J_{\lambda,\beta}(\phi_n,\psi_n)=\frac14\mathcal{L}_{\beta}(\phi_n,\psi_n).%
\end{equation}
If $\|\phi_n\|_{L^4(\bbr^4)}+\|\psi_n\|_{L^4(\bbr^4)}\to+\infty$ as $n\to\infty$ up to a subsequence, then by \eqref{eq0054}, $0\leq\beta<1$ and the H\"older inequality, we can see that%
\begin{eqnarray*}
o_n(1)=\frac{c_{\lambda,\beta}+o_n(1)}{\|\phi_n\|_{L^4(\bbr^4)}^4+\|\psi_n\|_{L^4(\bbr^4)}^4}\geq\frac14(1+\beta)>0,%
\end{eqnarray*}
which is a contradiction.  Thus, $\{(\phi_n,\psi_n)\}$ is bounded in $L^4(\bbr^4)\times L^4(\bbr^4)$.  Since $\lambda\geq\Lambda_0^*>\max\{\Lambda_{a,0}, \Lambda_{b,0}\}$, by the definitions of $\Lambda_{a,0}$ and $\Lambda_{b,0}$ given by \eqref{eq0056} and Remark~\ref{rmk0001}, we can obtain that%
\begin{equation}  \label{eq0062}
\int_{\bbr^4}(\lambda a(x)+a_0)^-\phi_n^2dx\leq|a_0||\mathcal{A}_\lambda|^{\frac12}\|\phi_n\|_{L^4(\bbr^4)}^2,\\
\end{equation}
\begin{equation}  \label{eq0062+0}
\int_{\bbr^4}(\lambda b(x)+b_0)^-v\psi_n^2dx\leq|b_0||\mathcal{B}_\lambda|^{\frac12}\|\psi_n\|_{L^4(\bbr^4)}^2,%
\end{equation}
where $\mathcal{A}_{\lambda}$ and $\mathcal{B}_{\lambda}$ are   given by \eqref{eq0055} and Remark~\ref{rmk0001}.  This implies%
\begin{eqnarray*}
& &c_{\lambda,\beta}+o_n(1)\\
&=&J_{\lambda,\beta}(\phi_n,\psi_n)\\%
&=&\frac14(\mathcal{D}_{a,\lambda}(\phi_n,\phi_n)+\mathcal{D}_{b,\lambda}(\psi_n,\psi_n))\\%
&\geq&\frac14(\|\phi_n\|^2_{a,\lambda}+\|\psi_n\|_{b,\lambda}^2)
-\frac14(|a_0||\mathcal{A}_\lambda|^{\frac12}\|\phi_n\|_{L^4(\bbr^4)}^2+|b_0||\mathcal{B}_\lambda|^{\frac12}\|\psi_n\|_{L^4(\bbr^4)}^2).%
\end{eqnarray*}
Hence $\{(\phi_n,\psi_n)\}$ is bounded in $E_\lambda$.%

\vskip0.11in
\noindent{\bf Step 2.}We prove that $\{\phi_n,\psi_n\}$ is a $(PS)_{c_{\lambda,\beta}}$ sequence of $J_{\lambda,\beta}(u,v)$.
Indeed, by the definition of $\Psi_{\lambda,\beta}(u,v)$, it is easy to see that%
\begin{equation*}
J_{\lambda,\beta}(\phi_n,\psi_n)=c_{\lambda,\beta}+o_n(1).%
\end{equation*}
It remains to show that $D[J_{\lambda,\beta}(\phi_n,\psi_n)]=o_n(1)$ strongly in $E_\lambda^*$.  Let $(\varphi, \eta)\in E_\lambda$.  Without loss of generality, we may assume $(\varphi, \eta)\not=(0, 0)$ and $\|\varphi\|_{a,\lambda}^2+\|\eta\|_{b,\lambda}^2=1$.  If $(\varphi, \eta)\in\widehat{\mathcal{F}}_{a,\lambda}^{\perp}\times\widehat{\mathcal{F}}_{b,\lambda}^{\perp}$, then due to $(\phi_n,\psi_n)\in\mathcal{G}_{\lambda,\beta}$, we have%
\begin{equation*}
\langle D[J_{\lambda,\beta}(\phi_n,\psi_n)], (\varphi, \eta)\rangle_{E^*_\lambda, E_\lambda}=0\quad\text{for all }n\in\bbn.%
\end{equation*}
Otherwise, by Lemma~\ref{lem0012}, we can see that%
\begin{equation*}
\langle D[J_{\lambda,\beta}(\phi_n,\psi_n)], (t_{\lambda,n}^0\widetilde{\varphi}, t_{\lambda,n}^0\widetilde{\eta})\rangle_{E^*_\lambda, E_\lambda}=o_n(1),%
\end{equation*}
where   $\widetilde{\varphi}$  and $\widetilde{\eta}$ are the projections of
$\varphi$ and $\eta$ on $\widetilde{\mathcal{F}}_{a,\lambda}^{\perp}\oplus\mathcal{F}_{a,\lambda}$
and $\widetilde{\mathcal{F}}_{b,\lambda}^{\perp}\oplus\mathcal{F}_{b,\lambda}$ and
 $t_{\lambda,n}^0\in\bbr^+$ is given by \eqref{eq0057} and only depends on
  $(\phi_n,\psi_n)$.  If $t_{\lambda,n}^0\to0$ as $n\to\infty$, then by
   Lemma~\ref{lem0110}, \eqref{eq0057} and Step 1, we must have
   $J_{\lambda,\beta}(\phi_n,\psi_n)\leq0$, which implies
   $c_{\lambda,\beta}\leq0$.  It is impossible since Lemma~\ref{lem0007} holds
    for $\lambda\geq\Lambda_0^*$.  Therefore, by $(\phi_n,\psi_n)\in\mathcal{G}_{\lambda,\beta}$, we have%
\begin{equation*}
\langle D[J_{\lambda,\beta}(\phi_n,\psi_n)], (\varphi, \eta)\rangle_{E^*_\lambda, E_\lambda}=o_n(1)\quad\text{for all }(\varphi, \eta)\in E_\lambda.%
\end{equation*}

\noindent{\bf Step 3.}We prove that there exists $(u_{\lambda,\beta}, v_{\lambda,\beta})\in\mathcal{G}_{\lambda,\beta}$ such that
$$D[J_{\lambda,\beta}(u_{\lambda,\beta}, v_{\lambda,\beta})]=0$$ in $E_\lambda^*$ and $J_{\lambda,\beta}(u_{\lambda,\beta}, v_{\lambda,\beta})=c_{\lambda,\beta}$.%
Without loss of generality, we may assume that $(\phi_n,\psi_n)\rightharpoonup(u_{\lambda,\beta},v_{\lambda,\beta})$ weakly in $E_\lambda$ and $(\phi_n,\psi_n)\to(u_{\lambda,\beta},v_{\lambda,\beta})$ a.e. in $\bbr^4\times\bbr^4$ as $n\to\infty$.
Clearly, $D[J_{\lambda,\beta}(u_{\lambda,\beta},v_{\lambda,\beta})]=0$ in $E_\lambda^*$.  In this situation, two cases may occur:%
\begin{enumerate}
\item[$(1)$] $(u_{\lambda,\beta},v_{\lambda,\beta})=(0, 0)$;%
\item[$(2)$] $(u_{\lambda,\beta},v_{\lambda,\beta})\not=(0, 0)$.%
\end{enumerate}
If case~$(1)$ happens, then $(\phi_n,\psi_n)\rightharpoonup(0,0)$ weakly in $E_\lambda$ as $n\to\infty$.  Since $|\mathcal{A}_\lambda|<+\infty$ and $|\mathcal{B}_\lambda|<+\infty$, by the condition $(D_1)$, we can see that%
\begin{eqnarray*}
\int_{\bbr^4}(\lambda a(x)+a_0)\phi_n^2dx\geq\int_{\bbr^4}(\lambda a(x)+a_0)^+\phi_n^2dx+o_n(1)\geq o_n(1)%
\end{eqnarray*}
and
\begin{eqnarray*}
\int_{\bbr^4}(\lambda b(x)+b_0)\psi_n^2dx\geq\int_{\bbr^4}(\lambda b(x)+b_0)^+\psi_n^2dx+o_n(1)\geq o_n(1),%
\end{eqnarray*}
where $\mathcal{A}_\lambda$ and $\mathcal{B}_\lambda$ are given by \eqref{eq0055} and Remark~\ref{rmk0001}.
By using  the Sobolev inequality and $(\phi_n,\psi_n)\in\mathcal{G}_{\lambda,\beta}$, we have  that%
\begin{eqnarray*}
S(\|\phi_n\|_{L^4(\bbr^4)}^2+\|\psi_n\|_{L^4(\bbr^4)}^2)\leq\|\phi_n\|_{L^4(\bbr^4)}^4+\|\psi_n\|_{L^4(\bbr^4)}^4+2\beta\|\phi_n^2\psi_n^2\|_{L^1(\bbr^4)}+o_n(1).%
\end{eqnarray*}
Note that $0\leq\beta<1$ and $J_{\lambda,\beta}(\phi_n,\psi_n)=c_{\lambda,\beta}+o_n(1)$, by Lemma~\ref{lem0007}, we have $S+o_n(1)\leq\|\phi_n\|_{L^4(\bbr^4)}^2+\|\psi_n\|_{L^4(\bbr^4)}^2$, which then implies%
\begin{equation*}
\|\phi_n\|_{L^4(\bbr^4)}^4+\|\psi_n\|_{L^4(\bbr^4)}^4+2\beta\|\phi_n^2\psi_n^2\|_{L^1(\bbr^4)}\geq S^2+o_n(1).%
\end{equation*}
Now, by Lemma~\ref{lem0009}, we can see that%
\begin{eqnarray*}
\frac14 S^2>c_{\lambda,\beta}+o_n(1)&=&J_{\lambda,\beta}(\phi_n,\psi_n)\\%
&=&\frac14(\|\phi_n\|_{L^4(\bbr^4)}^4+\|\psi_n\|_{L^4(\bbr^4)}^4+2\beta\|\phi_n^2\psi_n^2\|_{L^1(\bbr^4)})\\%
&\geq&\frac14S^2+o_n(1),%
\end{eqnarray*}
it is impossible.  Therefore, we must have the case~$(2)$.  In this case we can easily see that $J_{\lambda,\beta}(u_{\lambda,\beta},v_{\lambda,\beta})\geq c_{\lambda,\beta}$.  On the other hand, since $0\leq\beta<1$, we must have%
\begin{equation*}
|\phi_n|^4+|\psi_n|^4+2\beta|\phi_n|^2|\psi_n|^2\geq0\quad\text{all }n\in\bbn\text{ and }x\in\bbr^4.%
\end{equation*}
It follows from the Fatou's lemma that%
 $$\liminf_{n\to\infty}\int_{\bbr^4}(|\phi_n|^4+|\psi_n|^4+2\beta|\phi_n|^2|\psi_n|^2)dx$$
$$\geq\int_{\bbr^4}(|u_{\lambda,\beta}|^4+|v_{\lambda,\beta}|^4+2\beta|u_{\lambda,\beta}|^2|v_{\lambda,\beta}|^2)dx,$$
hence
\begin{eqnarray*}
c_{\lambda,\beta}&=&\lim_{n\to\infty}J_{\lambda,\beta}(\phi_n,\psi_n)-\frac12\langle D[J_{\lambda,\beta}(\phi_n,\psi_n)], (\phi_n,\psi_n)\rangle_{E^*_\lambda, E_\lambda}\\%
&=&\frac14\lim_{n\to\infty}\int_{\bbr^4}(|\phi_n|^4+|\psi_n|^4+2\beta|\phi_n|^2|\psi_n|^2)dx\\%
&\geq&\frac14\int_{\bbr^4}(|u_{\lambda,\beta}|^4+|v_{\lambda,\beta}|^4+2\beta|u_{\lambda,\beta}|^2|v_{\lambda,\beta}|^2)dx\\%
&=&J_{\lambda,\beta}(u_{\lambda,\beta},v_{\lambda,\beta})-\frac12\langle D[J_{\lambda,\beta}(u_{\lambda,\beta},v_{\lambda,\beta})], (u_{\lambda,\beta},v_{\lambda,\beta})\rangle_{E^*_\lambda, E_\lambda}\\%
&=&J_{\lambda,\beta}(u_{\lambda,\beta},v_{\lambda,\beta}).%
\end{eqnarray*}
Therefore, $(u_{\lambda,\beta},v_{\lambda,\beta})$ is a general ground state solution of $(\mathcal{P}_{\lambda,\beta})$ for $\lambda\geq\Lambda_0^*$ and $0\leq\beta<1$.%
\end{proof}

For the general ground state solution of $(\mathcal{P}_{\lambda,\beta})$ obtained by Proposition~\ref{prop0001},
 we also have the following properties.%
\begin{proposition}\label{prop0002}
Let $(u_{\lambda,\beta}, v_{\lambda,\beta})$ be the general ground state solution of $(\mathcal{P}_{\lambda,\beta})$
obtained by Proposition~\ref{prop0001}.  Then
\begin{enumerate}
\item[$(1)$] $(u_{\lambda,0}, v_{\lambda,0})$ is a semi-trivial solution of $(\mathcal{P}_{\lambda,0})$ and must be  of the type $(u_{\lambda,0}, 0)$ in the case of $a_0\leq-\mu_{a,1}$ with $-a_0\not\in\sigma(-\Delta, H_0^1(\Omega_a))$ and $b_0\geq0$.  Furthermore, $(u_{\lambda,0}, 0)$ is a least energy critical point of $I_{a,\lambda}(u)$.%
\item[$(2)$] For every $0\leq\beta<1$, there exists $\Lambda_\beta^*\geq\Lambda_0^*$
such that $(u_{\lambda,\beta}, v_{\lambda,\beta})$ is a
 semi-trivial solution of $(\mathcal{P}_{\lambda,\beta})$ and must be
 of the type $(u_{\lambda,\beta}, 0)$ for $\lambda>\Lambda_\beta^*$ in the case
 of $a_0\leq-\mu_{a,1}$ with $-a_0\not\in\sigma(-\Delta, H_0^1(\Omega_a))$ and $b_0\geq0$.%
\end{enumerate}
\end{proposition}
\begin{proof}
$(1)$\quad Since Lemma~\ref{lem0009} holds, by a similar argument as used in the proof of Lemma~\ref{lem1002}, we can get the conclusion.%

$(2)$\quad Suppose the contrary, $(u_{\lambda_n,\beta_0}, v_{\lambda_n,\beta_0})$ is non-trivial for
 some $\beta_0\in[0, 1)$ and $\{\lambda_n\}$ with $\lambda_n\to+\infty$ as $n\to\infty$.
 Since Lemma~\ref{lem0009} holds and $\beta_0\in[0, 1)$, by a similar argument as used in Step 1 of the
 proof to Proposition~\ref{prop0001}, we can see that $\{(u_{\lambda_n,\beta_0}, v_{\lambda_n,\beta_0})\}$ is
 bounded in $L^4(\bbr^4)\times L^4(\bbr^4)$.  Without loss of generality,
  we assume $(u_{\lambda_n,\beta_0}, v_{\lambda_n,\beta_0})\rightharpoonup(u_{0,\beta_0}, v_{0,\beta_0})$
  weakly in $L^4(\bbr^4)\times L^4(\bbr^4)$ and $(u_{\lambda_n,\beta_0}, v_{\lambda_n,\beta_0})\to(u_{0,\beta_0}, v_{0,\beta_0})$
   a.e. in $\bbr^4\times\bbr^4$ as $n\to\infty$.  Since $|\mathcal{A}_{\lambda_n}|<+\infty$ and is nonincreasing
   for $n\in\bbn$ due to $\lambda_n\to+\infty$, by similar arguments as used in \eqref{eq1002} and \eqref{eq0062}, we can obtain that%
\begin{equation}\label{eq1102}
\int_{\bbr^4}(a(x)+\frac{a_0}{\lambda_n})^+u_{\lambda_n,\beta_0}^2+b(x)v_{\lambda_n,\beta_0}^2dx\to0\quad\text{as }n\to\infty,%
\end{equation}
where $\mathcal{A}_{\lambda_n}$ is given by \eqref{eq0055}.  By the Fatou lemma, we have from \eqref{eq1102} that $(u_{0,\beta_0}, v_{0,\beta_0})\in H_0^1(\Omega_a)\times H_0^1(\Omega_b)$ with $u_{0,\beta_0}=0$ outside $\Omega_a$ and $v_{0,\beta_0}=0$ outside $\Omega_b$.  It follows from the Sobolev embedding theorem, the condition $(D_2)$ and \eqref{eq1102} once more that $(u_{\lambda_n,\beta_0}, v_{\lambda_n,\beta_0})\to(u_{0,\beta_0}, v_{0,\beta_0})$ strongly in $L^2(\bbr^4)\times L^2(\bbr^4)$ as $n\to\infty$.  Since $H_0^1(\Omega_a)\times H_0^1(\Omega_b)\subset E_\lambda$ and $(u_{\lambda_n,\beta_0}, v_{\lambda_n,\beta_0})$ is the general ground state solution of $(\mathcal{P}_{\lambda_n,\beta_0})$ for all $n\in\bbn$, by the condition $(D_3)$, it is easy to see that $I_a'(u_{0,\beta_0})=0$ and $I_b'(v_{0,\beta_0})=0$ in $H^{-1}(\Omega_a)$ and $H^{-1}(\Omega_b)$, respectively.  Thanks to Lemma~\ref{lem0009} and a similar argument as used in \eqref{eq1000}, we must have $v_{0,\beta_0}=0$.  Let $w_{\lambda_n,\beta_0}=u_{\lambda_n,\beta_0}-u_{0,\beta_0}$.  Then $w_{\lambda_n,\beta_0}\to0$ strongly in $L^2(\bbr^4)$ and $w_{\lambda_n,\beta_0}\rightharpoonup0$ weakly in $L^4(\bbr^4)$ as $n\to\infty$.  Furthermore, by the condition $(D_3)$, the Brez\'is-Lieb lemma and \cite[Lemma~2.3]{CZ14}, we also have%
\begin{equation}  \label{eq1201}
J_{\lambda_n,\beta_0}(u_{\lambda_n,\beta_0}, v_{\lambda_n,\beta_0})=J_{\lambda_n,\beta_0}(w_{\lambda_n,\beta_0},v_{\lambda_n,\beta_0})+I_a(u_{0,\beta_0})+o_n(1).%
\end{equation}
Now,  similar to  \eqref{eq1000}, we can get that $u_{0,\beta_0}\not=0$.
Due to $-a_0\not\in\sigma(-\Delta, H_0^1(\Omega_a))$ and the result of \cite{CSZ12},
$u_{0,\beta_0}$ is a least energy critical point of $I_a(u)$.  This together
with \eqref{eq1201}, implies $(w_{\lambda_n,\beta_0},v_{\lambda_n,\beta_0})\to(0, 0)$ strongly
in $L^4(\bbr^4)\times L^4(\bbr^4)$ as $n\to\infty$.  It follows from the condition $(D_1)$ and
the fact that $(u_{\lambda_n,\beta_0}, v_{\lambda_n,\beta_0})$ is a non-trivial solution
of $(\mathcal{P}_{\lambda_n,\beta_0})$ that $\|u_{0,\beta_0}\|_{L^4(\bbr^4)}^2\geq \frac{S}{\beta_0}$,
 which is impossible.  It remains to show that $v_{\lambda,\beta}\not=0$ for $|\beta|<1$ and $\lambda>\Lambda_\beta^*$.
 Indeed, if $u_{\lambda,\beta}=0$, then by $b_0\geq0$ and the condition $(D_1)$,
 we can see that $c_{\lambda,\beta}\geq\frac14 S^2$, which contradicts to Lemma~\ref{lem0004} once more.%
\end{proof}

We close this section by%

\noindent\textbf{Proof of Theorem~\ref{thm0004}:}\quad It follows immediately from Propositions~\ref{prop0001}--\ref{prop0002}.%
\qquad\raisebox{-0.5mm}{%
\rule{1.5mm}{4mm}}\vspace{6pt}

\section{Concentration behaviors}
This section is devoted to the concentration behaviors of the ground state solutions and the general
ground state solutions to $(\mathcal{P}_{\lambda,\beta})$ obtained by Theorems~\ref{thm0001} and \ref{thm0004}.
 We first study such a property as  $\lambda\to+\infty$.%

\vskip0.11in

\noindent\textbf{Proof of Theorem~\ref{thm0002}:}\quad Let $(u_{\lambda_n,\beta}, v_{\lambda_n,\beta})$ be the solution obtained by Theorems~\ref{thm0001} and \ref{thm0004} with $\lambda_n\to+\infty$ as $n\to\infty$.  Then by a similar argument as used in the proof of $(3)$ to Proposition~\ref{prop0002}, we can see that $(u_{\lambda_n,\beta}, v_{\lambda_n,\beta})\to(u_{0,\beta}, v_{0,\beta})\in H_0^1(\Omega_a)\times H_0^1(\Omega_b)$ strongly in $(L^4(\bbr^4)\times L^4(\bbr^4))\cap (L^2(\bbr^4)\times L^2(\bbr^4))$ as $n\to\infty$ up to a subsequence if $(u_{\lambda_n,\beta}, v_{\lambda_n,\beta})$ is the ground state solution with $\beta<\beta_0$ in the case of $a_0\leq b_0<0$ or $(u_{\lambda_n,\beta}, v_{\lambda_n,\beta})$ is the general ground state solution in the case of $a_0<0$.  Furthermore, $(u_{0,\beta}, v_{0,\beta})$ is a solution of the equations \eqref{eq9001}.  It follows from $D[J_{\lambda,\beta}(u_{\lambda,\beta}, v_{\lambda,\beta})]=0$ in $E_\lambda^*$ that $(u_{\lambda_n,\beta}, v_{\lambda_n,\beta})\to(u_{0,\beta}, v_{0,\beta})$ in $D^{1,2}(\bbr^4)\times D^{1,2}(\bbr^4)$ as $n\to\infty$ up to a subsequence.  Thus, $(u_{\lambda_n,\beta}, v_{\lambda_n,\beta})\to(u_{0,\beta}, v_{0,\beta})$ in $\h\times\h$ as $n\to\infty$ up to a subsequence.  If $(u_{\lambda_n,\beta}, v_{\lambda_n,\beta})$ is the ground state solution with $\beta<\beta_0$ in the case of $a_0\leq b_0<0$, then by Lemmas~\ref{lem0020} and \ref{lem5010}, $(u_{0,\beta}, v_{0,\beta})$ is non-trivial.  Thanks to lemma~\ref{lem5003}, $(u_{0,\beta}, v_{0,\beta})$ must be the ground state solution of \eqref{eq9001}.  If $(u_{\lambda_n,\beta}, v_{\lambda_n,\beta})$ is the general ground state solution in the case of $a_0<0$, then by Lemma~\ref{lem5003} once more, we have that either $(u_{0,\beta}, v_{0,\beta})$ is semi-trivial or $(u_{0,\beta}, v_{0,\beta})=(0, 0)$.  Note that Lemma~\ref{lem0008} holds, thus, it is easy to show that $m_{\lambda_n,\beta}$ is nondecreasing for $n$ and $m_{\lambda_n,\beta}>0$ for all $n\in\bbn$.  Therefore, we must have $(u_{0,\beta}, v_{0,\beta})$ is semi-trivial.  Due to lemma~\ref{lem5003}, we can also see that $(u_{0,\beta}, v_{0,\beta})$ is a general ground state solution to \eqref{eq9001}, which completes the proof.%
\qquad\raisebox{-0.5mm}{%
\rule{1.5mm}{4mm}}\vspace{6pt}

\vskip0.12in

We next study the concentration behaviors when  $\beta\to-\infty$.%

\noindent\textbf{Proof of Theorem~\ref{thm0003}:}\quad Let $(u_{\lambda,\beta_n}, v_{\lambda,\beta_n})$ be the ground state solution of $(\mathcal{P}_{\lambda,\beta_n})$ obtained by Theorem~\ref{thm0001} with $\beta_n\to-\infty$ as $n\to\infty$.  By Lemma~\ref{lem5003}, we have $\|(u_{\lambda,\beta_n}, v_{\lambda,\beta_n})\|_\lambda^2\leq4(m_a+m_b)<S^2$.  It follows from \eqref{eq0001} that $\{(u_{\lambda,\beta_n}, v_{\lambda,\beta_n})\}$ is bounded in $E_\lambda$ and $\mathcal{H}$.  Without loss of generality, we assume $(u_{\lambda,\beta_n}, v_{\lambda,\beta_n})\rightharpoonup(u_{\lambda,\infty}, v_{\lambda,\infty})$ weakly in $E_\lambda$ and $\mathcal{H}$ and $(u_{\lambda,\beta_n}, v_{\lambda,\beta_n})\to(u_{\lambda,\infty}, v_{\lambda,\infty})$ a.e. in $\bbr^4\times\bbr^4$ as $n\to\infty$.  For the sake of clarity, the following proof will be further performed by several steps.%

\vskip0.13in

\noindent {\it  Step 1. }We prove that $(u_{\lambda,\beta_n}, v_{\lambda,\beta_n})\to(u_{\lambda,\infty}, v_{\lambda,\infty})$
 strongly in $\mathcal{H}$ as $n\to\infty$ with $u_{\lambda,\infty}\not=0$ and $v_{\lambda,\infty}\not=0$. Indeed, one of the following two cases must happen:%
\begin{enumerate}
\item[$(1)$] $u_{\lambda,\infty}=0$ or $v_{\lambda,\infty}=0$;%
\item[$(2)$] $u_{\lambda,\infty}\not=0$ and $v_{\lambda,\infty}\not=0$.%
\end{enumerate}
If the  Case~$(1)$ happens, then without loss of generality, we may assume $u_{\lambda,\infty}=0$.  By a similar argument as used in Proposition~\ref{prop0003}, we can see that $\|u_{\lambda,\beta_n}\|_{L^4(\bbr^4)}^4\\
\geq S^2+o_n(1)$.  On the other hand, since $\beta_n<0$, there exists $0<s_n\leq1+o_n(1)$ such that $s_nv_{\lambda,\beta_n}\in\mathcal{N}_{b,\lambda}$.  Now, by Lemma~\ref{lem0002}, we have
\begin{equation}\label{eq5020}
J_{\lambda,\beta_n}(u_{\lambda,\beta_n}, v_{\lambda,\beta_n})\geq J_{\lambda,\beta_n}(u_{\lambda,\beta_n}, s_nv_{\lambda,\beta_n})\geq\frac14 S^2+m_{b,\lambda}+o_n(1).%
\end{equation}
Thanks to Lemma~\ref{lem7001}, $m_{b,\lambda}\to m_b$ as $\lambda\to\infty$.  Thus,
there exists $\Lambda_3>0$ such that \eqref{eq5020} is impossible for $\lambda>\Lambda_3$ due
 to Lemma~\ref{lem5003}.  Therefore, we must have the Case~$(2)$.  In this case, it is easily
 see from the Fatou's lemma that $\int_{\bbr^4}u_{\lambda,\infty}^2v_{\lambda,\infty}^2=0$.
 On the other hand, since $\beta_n<0$, multiplying $(\mathcal{P}_{\lambda,\beta_n})$ with $(u_{\lambda,\infty}, v_{\lambda,\infty})$,
 we obtain that%
\begin{equation*}
\mathcal{D}_{a,\lambda}(u_{\lambda,\infty},u_{\lambda,\infty})\leq\|u_{\lambda,\infty}\|_{L^4(\bbr^4)}^4\text{ and }\mathcal{D}_{b,\lambda}(v_{\lambda,\infty},v_{\lambda,\infty})\leq\|v_{\lambda,\infty}\|_{L^4(\bbr^4)}^4.%
\end{equation*}
Thanks to Lemma~\ref{lem0002}, there exists $0<t_n^*,s_n^*\leq1$ such that
 $(t_n^*u_{\lambda,\infty}, s_n^*v_{\lambda,\infty})\in\mathcal{N}_{\lambda,\beta_n}$ for all $n$,
 which, together with the Sobolev embedding theorem, implies
\begin{equation*}
\mathcal{D}_\lambda(t_n^*u_{\lambda,\infty}, s_n^*v_{\lambda,\infty})\geq\mathcal{D}_\lambda(u_{\lambda,\beta_n}, v_{\lambda,\beta_n})\geq\mathcal{D}_\lambda(u_{\lambda,\infty}, v_{\lambda,\infty})+o_n(1).%
\end{equation*}
It follows from Lemma~\ref{lem0110} and \eqref{eq0001} that $(u_{\lambda,\beta_n}, v_{\lambda,\beta_n})\to(u_{\lambda,\infty}, v_{\lambda,\infty})$ strongly in $\mathcal{H}$ as $n\to\infty$.%

\vskip0.11in
\noindent{\it Step 2.} We prove that $u_{\lambda,\infty}\in H^1_0(\{u_{\lambda,\infty}>0\})$ and $v_{\lambda,\infty}\in H^1_0(\{v_{\lambda,\infty}>0\})$.%

Indeed, since the conditions $(D_1)$--$(D_3)$ hold, $u_{\lambda,\beta_n}$ and
 $v_{\lambda,\beta_n}$ are both positive in $\bbr^4$ by the maximum principle and
 Lemma~\ref{lem5005}.  It follows that $u_{\lambda,\infty}$ and $v_{\lambda,\infty}$
 are both nonnegative in $\bbr^4$.  Thanks to Step 1 and \cite[Propositions~3.8 and 3.9]{AFS09},
  $\{(u_{\lambda,\beta_n}, v_{\lambda,\beta_n})\}$ is also bounded in $L^\infty(\bbr^4)\times L^\infty(\bbr^4)$.
   By \cite[Theorem~1.7]{SZ14}, $\{(\nabla u_{\lambda,\beta_n}, \nabla v_{\lambda,\beta_n})\}$
   is bounded in $L^\infty(\bbr^4)\times L^\infty(\bbr^4)$.
    It follows from $\beta_n<0$ and the classical elliptic
    regularity theory (cf. \cite[Corollary~3.5]{SZ02})
    that $u_{\lambda,\beta_n}$ and $v_{\lambda,\beta_n}$ are both local Lipschitz.
     Thus, $\partial\{u_{\lambda,\infty}>0\}$ and $\partial\{v_{\lambda,\infty}>0\}$ are
      both local Lipschitz and $H^1_0(\{u_{\lambda,\infty}>0\})$
      and $H^1_0(\{v_{\lambda,\infty}>0\})$ are both well defined.
       Furthermore, due to the extension theorem, we also
        have $u_{\lambda,\infty}\in H^1_0(\{u_{\lambda,\infty}>0\})$ and
         $v_{\lambda,\infty}\in H^1_0(\{v_{\lambda,\infty}>0\})$.
Now, we can use similar arguments as used in the proof of \cite[Theorem~1.4]{CZ121} (see also \cite[Theorem~1.3]{WWZ15}) to show the conclusions.%
\qquad\raisebox{-0.5mm}{%
\rule{1.5mm}{4mm}}\vspace{6pt}

We close this section by

\noindent\textbf{Proof of Theorem~\ref{thm0007}:}\quad By Lemmas~\ref{lem5003} and \ref{lem5002},
 we have $m_{\lambda,\beta}\in[m_{a,\lambda}+m_{b,\lambda}, m_a+m_b]$ for all $\lambda>\Lambda_{0}$
  and $\beta\leq0$ in the case of $a_0\leq b_0<0$.  Furthermore, by Lemma~\ref{lem7001},
  we have $m_{a,\lambda}+m_{b,\lambda}\to m_a+m_b$ as $\lambda\to+\infty$.  Thus,
  for every $\{(\lambda_n,\beta_n)\}$ satisfying $\lambda_n\to+\infty$ and $\beta_n\to-\infty$ as $n\to\infty$,
   by a similar argument as used in Theorem~\ref{thm0002}, we can see that the ground state solution of
    $(\mathcal{P}_{\lambda_n,\beta_n})$ obtained in Theorem~\ref{thm0001} in the case of $a_0\leq b_0<0$ has
    the same concentration behaviors as described in Theorem~\ref{thm0002}.
\qquad\raisebox{-0.5mm}{%
\rule{1.5mm}{4mm}}

\end{document}